\documentclass[a4paper,12pt]{amsart}
\usepackage{tikz, amssymb, stmaryrd, hyperref}

\newtheorem{thm}{Theorem}[section]
\newtheorem{lem}[thm]{Lemma}
\newtheorem{prp}[thm]{Proposition}
\newtheorem{cor}[thm]{Corollary}

\theoremstyle{remark}
\newtheorem{rmk}[thm]{Remark}
\newtheorem{exa}[thm]{Example}

\theoremstyle{definition}
\newtheorem{ntn}[thm]{Notation}
\newtheorem{dfn}[thm]{Definition}

\numberwithin{equation}{section}

\newcommand{\Bb}{\mathcal{B}}
\newcommand{\Ff}{\mathcal{F}}

\newcommand{\Hh}{\mathcal{H}}
\newcommand{\Kk}{\mathcal{K}}
\newcommand{\Ll}{\mathcal{L}}
\newcommand{\Mm}{\mathcal{M}}

\newcommand{\Qq}{\mathcal{Q}}
\newcommand{\Ss}{\mathcal{S}}
\newcommand{\Tt}{\mathcal{T}}

\newcommand{\Zz}{\mathcal{Z}}

\newcommand{\CC}{\mathbb{C}}
\newcommand{\NN}{\mathbb{N}}

\newcommand{\RR}{\mathbb{R}}
\renewcommand{\SS}{\mathbb{S}}
\newcommand{\TT}{\mathbb{T}}
\newcommand{\ZZ}{\mathbb{Z}}

\newcommand{\SG}[2][{}]{\Ss^{#1}\!#2}

\newcommand{\Ad}{\operatorname{Ad}}
\newcommand{\Aut}{\operatorname{Aut}}
\newcommand{\lsp}{\operatorname{span}}
\newcommand{\clsp}{\overline{\lsp}}
\newcommand{\coker}{\operatorname{coker}}
\newcommand{\lt}{\operatorname{lt}}

\newcommand{\Id}{\operatorname{Id}}
\newcommand{\supp}{\operatorname{supp}}

\newcommand{\op}{\operatorname{op}}
\newcommand{\id}{\operatorname{id}}
\newcommand{\Me}{\operatorname{Me}}

\title{The suspension of a graph, and associated \texorpdfstring{$C^*$}{C*}-algebras}

\author{Aidan Sims}

\email{asims@uow.edu.au}

\address{School of Mathematics and Applied Statistics\\
University of Wollongong\\
NSW 2522\\
AUSTRALIA}

\thanks{This research was supported by the Australian Research Council.}
\date{\today}

\begin{document}

\begin{abstract}
Given a directed graph $E$, we construct for each real number $l$ a quiver whose vertex
space is the topological realisation of $E$, and whose edges are directed paths of length
$l$ in the vertex space. These quivers are not topological graphs in the sense of
Katsura, nor topological quivers in the sense of Muhly and Tomforde. We prove that when
$l = 1$ and $E$ is finite, the infinite-path space of the associated quiver is
homeomorphic to the suspension of the one-sided shift of $E$. We call this quiver the
suspension of $E$. We associate both a Toeplitz algebra and a Cuntz--Krieger algebra to
each of the quivers we have constructed, and show that when $l = 1$ the Cuntz--Krieger
algebra admits a natural faithful representation on the $\ell^2$-space of the suspension
of the one-sided shift of $E$. For graphs $E$ in which sufficiently many vertices both
emit and receive at least two edges, and for rational values of $l$, we show that the
Toeplitz algebra and the Cuntz--Krieger algebra of the associated quiver are homotopy
equivalent to the Toeplitz algebra and Cuntz--Krieger algebra respectively of a graph
that can be regarded as encoding the $l$\textsuperscript{th} higher shift associated to
the one-sided shift space of $E$.
\end{abstract}

\maketitle

\section{Introduction}

Each finite directed graph $E$ determines a corresponding $1$-sided shift space $(X_E,
\sigma)$. Cuntz--Krieger algebras \cite{CK}, and more generally graph $C^*$-algebras
\cite{EW, KPRR, KPR}, encode the dynamics $(X_E, \sigma)$ $C^*$-algebraically, and there
has been intense interest in these $C^*$-algebras ever since their introduction---see,
for example \cite{BCW, CPR, DT, ET, HS, Sorensen} and the bibliography of
\cite{CBMSbook}.

In addition to the shift space $(X_E, \sigma)$ itself, a directed graph $E$ determines a
family of dynamical systems indexed by the rational numbers. Given a rational number
$m/n$ with $n > 0$, we first form the directed graph $D_n(E)$ obtained by inserting $n-1$
new vertices along every edge of $E$---so that each edge of $E$ corresponds to a path of
length $n$ in $D_n(E)$---and then consider the $m$\textsuperscript{th} higher-power graph
$D_n(E)(0,m)$ of $D_n(E)$, whose vertices are those of $D_n(E)$ and whose edges are paths
of length $m$ in $D_n(E)$. The shift space associated to this graph can be regarded as
encoding the $m/n$\textsuperscript{th} higher shift of $E$ in the sense that the system
$(X_{D_n(E)(0,m)}, \sigma_{D_n(E)(0,m)}^n)$ decomposes into $n$ disjoint copies of $(X_E,
\sigma_E^m)$.

All of these fractional higher shifts associated to a finite directed graph $E$ are
naturally encoded in the suspension of the base shift $X_E$. The suspension of $X_E$ is
the quotient $M(\sigma) := (X_E \times [0, \infty))/{\sim}$ by the equivalence relation
in which $(x, t+m) \sim (\sigma^m(x), t)$ for any positive integer $m$. For $l \in [0,
\infty)$, the map $(x, t) \mapsto (x, t+l)$ induces a continuous map $\lt_l$ on
$M(\sigma)$. For positive rationals $m/n$, the restriction of $\lt_{m/n}$ to $\{[x, t] :
t \in \frac{1}{n}\ZZ\}$ is a copy of $(X_{D_n(E)(0,m)}, \sigma_{D_n(E)(0,m)})$. By
analogy, we can regard the dynamics $\lt_l$ for arbitrary $l$ as corresponding to the
$l$\textsuperscript{th} higher shift of $E$.

In this paper we construct from each locally finite directed graph $E$ with no sources a
family of $C^*$-algebras parameterised by $l \in \RR$, and we prove that for $l = m/n$
rational, and for graphs $E$ in which enough vertices both emit and receive at least two
edges, the corresponding $C^*$-algebra is isomorphic to $C([0,1], C^*(D_n(E)(0,m)))$.

Our approach is to associate to each directed graph $E$ a family of quivers $\SG[l]{E}$,
one for each real parameter $l$. By a quiver here we mean a quadruple $(Q^0, Q^1, r, s)$
where $Q^0$ and $Q^1$ are topological spaces, and $r,s : Q^1 \to Q^0$ are continuous
maps. The vertex space $\SG[l]{E}^0$ of the quiver $\SG[l]{E}$ is the topological
realisation of $E$: it is the quotient of $E^1 \times [0,1]$ by the equivalence relation
under which $(e, 1)$ is glued to $(f, 0)$ if $s(e) = r(f)$. The edge space $\SG[l]{E}^1$
is, roughly speaking, the collection of directed paths of length $l$ in the vertex space.
We link this construction to suspension flows by showing that if $E$ is finite, then the
infinite-path space of $\SG[1]{E}$ is homeomorphic to the suspension flow $M(\sigma)$ of
the one-sided shift associated to $E$. Based on this, we call $\SG[1]{E}$ the
\emph{suspension} of $E$, and denote it $\SG{E}$.

The quivers $\SG[l]{E}$, for $l \not= 0$, are typically not topological quivers in the
sense of Muhly and Tomforde \cite{MT} because neither $r$ nor $s$ is typically an open
map: if $l > 0$, then $s$ is open map if and only if every vertex of $E$ receives exactly
one edge and $r$ is open if and only if each vertex emits exactly one edge; if $l < 0$
then the roles of emitters and receivers are reversed. In particular, our quivers are not
topological graphs in the sense of Katsura, and $C_c(\SG[l]{E}^1)$ does not admit a
natural $C_0(\SG[l]{E}^0)$-valued inner-product. Thus Katsura's modification of Pimsner
theory does not help us to construct a $C^*$-algebra from $\SG[l]{E}$. The source map for
the natural groupoid associated to the shift map on the infinite-path space of
$\SG[l]{E}$ is also typically not open, so this groupoid does not admit a Haar system,
and we cannot employ Renault's theory of groupoid $C^*$-algebras to construct a
$C^*$-algebra from $\SG[l]{E}$.

Instead, we associate a Toeplitz algebra and a Cuntz--Krieger algebra to each $\SG[l]{E}$
by considering natural actions by bounded linear operators of $C_0(\SG[l]{E}^0)$ and of
$C_c(\SG[l]{E}^1)$ on the (nonseparable) Hilbert space $\ell^2(\SG[l]{E}^*)$ with basis
indexed by all finite paths in $\SG[l]{E}$. We define the Toeplitz algebra $\Tt
C^*(\SG[l]{E})$ of $\SG[l]{E}$ to be the $C^*$-algebra generated by all of these
operators. For each vertex $\omega$ of $\SG[l]{E}$, the subspace $\ell^2(\SG[l]{E}^*
\omega)$, with basis indexed by paths whose source is $\omega$, is invariant for $\Tt
C^*(\SG[l]{E})$. Thus, for each $\omega$, we obtain a homomorphism from $\Tt
C^*(\SG[l]{E})$ to the Calkin algebra $\Qq(\ell^2(\SG[l]{E}^*\omega)) =
\Bb(\ell^2(\SG[l]{E}^*\omega))/\Kk(\ell^2(\SG[l]{E}^*\omega))$. We define the
Cuntz--Krieger algebra $C^*(\SG[l]{E})$ to be the image of $\Tt C^*(\SG[l]{E})$ in
$\bigoplus_\omega \Qq(\ell^2(\SG[l]{E}^*\omega))$ under the direct sum of these
homomorphisms. If $E$ has just one vertex $v$ and one edge $e$ (so that $r(e) = s(e) =
v$), then $\SG[l]{E}$ is the topological graph associated to the rotation-by-$l$
homeomorphism of the circle $\RR/\ZZ$, and $C^*(\SG[l]{E})$ is isomorphic to the rotation
algebra $A_l$ (see Example~\ref{exa:rotation}).

Our main results give a complete description of each of $\Tt C^*(\SG[m/n]{E})$ and
$C^*(\SG[m/n]{E})$ for integers $m \in \ZZ$ and $n \ge 1$ provided that enough vertices
of $E$ emit and receive at least two edges (the exact technical hypothesis is
complicated, but a simple sufficient condition is that every vertex both emits and
receives at least two edges). We prove that $\Tt C^*(\SG[m/n]{E})$ is homotopy equivalent
to $\Tt C^*(D_n(E)(0,m))$, and that $C^*(\SG[m/n]{E})$ is isomorphic to $C([0,1],
C^*(D_n(E)(0,m)))$. This suggests that the algebras $C^*(\SG[l]{E})$ for irrational $l$
are potential candidates for the role of the $C^*$-algebras of the
$l$\textsuperscript{th} higher-power shifts of $X_E$.

We obtain our description of $\Tt C^*(\SG[m/n]{E})$ and $C^*(\SG[m/n]{E})$ in stages. We
first reduce the problem to that of describing $\Tt C^*(\SG[m]{E})$ and $C^*(\SG[m]{E})$
for nonnegative integers $m$. To do this, we first eliminate the case $m = 0$ by showing
that $\Tt C^*(\SG[0]{E}) \cong C_0(\SG{E}^0) \otimes \Tt$, where $\Tt$ is the classical
Toeplitz algebra, and that this isomorphism descends to an isomorphism $C^*(\SG[0]{E})
\cong C_0(\SG{E}^0) \otimes C(\TT)$. We then show that $\Tt C^*(\SG[l]{E}) \cong \Tt
C^*(\SG[|l|]{E^{\op}})$ and $C^*(\SG[l]{E}) \cong C^*(\SG[|l|]{E^{\op}})$ for negative
$l$. We then show that for $m > 0$, we have $\Tt C^*(\SG[m/n]{E}) \cong \Tt
C^*(\SG[m]{D_n(E)})$ and $C^*(\SG[m/n]{E}) \cong C^*(\SG[m]{D_n(E)})$. Thus for any graph
$E'$ and any rational $m/n$, we have $\Tt C^*(\SG[m/n]{E'}) \cong \Tt C^*(\SG[|m|]{E})$
and $C^*(\SG[m/n]{E'}) \cong C^*(\SG[|m|]{E})$ for a suitable graph $E$.

The bulk of the technical work in the paper goes into analysing $\Tt C^*(\SG[m]{E})$ and
then $C^*(\SG[m]{E})$ for a locally finite directed graph $E$ with no sinks or sources.
We do this by showing that they are both $C(\SS)$-algebras, where $\SS$ is the circle
$\RR/\ZZ$, and analysing their fibres. We make use of the higher dual graphs $E(1,m+1)$
and $E(0,m)$ studied by Bates \cite{Bates}: $E(1,m+1)$ is the graph with vertices $E^1$
and edges $E^{m+1}$, and with range and source maps given by $\mu_1 \cdots \mu_{m+1}
\mapsto \mu_1$ and $\mu_1 \cdots \mu_{m+1} \mapsto \mu_{m+1}$ respectively; and $E(0,m)$
is the graph with vertices $E^0$ and edges $E^m$ and the usual range and source maps. It
is relatively straightforward to show that the fibre of $\Tt C^*(\SG[m]{E})$ over each $t
\in \SS \setminus \{0\}$ is canonically isomorphic to $\Tt C^*(E(1, m+1))$ and indeed
that the ideal of $\Tt C^*(\SG[m]{E})$ corresponding to $0 \in \SS$ is isomorphic to
$C_0((0,1), \Tt C^*(E(1, m+1))$. The fibre over $0 \in \SS$ is more complicated. Writing
$\SG[m]{E}^*_t$ for the space of paths in $\SG[m]{E}$ whose source lies in the image of
$E^1 \times \{t\}$ in $\SG{E}^0$, we describe natural unitary transformations $U_t :
\ell^2(E(1, m+1)^*) \to \ell^2(\SG[m]{E}^*_t)$ for $t \not= 0$. For each $a \in \Tt
C^*(\SG[m]{E})$ and $t \not= 0$, this allows us to view the image of $a_t$ in the fibre
of $\Tt C^*(\SG[m]{E})$ over $t$ as an operator on $\ell^2(E(1, m+1)^*)$. We prove that
$a_0^- := \lim_{t \nearrow 0} a_t$ and $a_0^+ := \lim_{t \searrow 0} a_t$ exist in
$\Bb(\ell^2(E(1, m+1)))$ and belong to the image of $\Tt C^*(E(1,m+1))$, and that the map
$a \mapsto a_0^+ \oplus a_0^-$ descends to an injective homomorphism $\eta$ from the
fibre of $\Tt C^*(\SG[m]{E})$ over $0$ to $\Tt C^*(E(1,m+1)) \oplus \Tt C^*(E(1,m+1))$.
Following the arguments of \cite{Bates}, we show that there is a canonical injection
$\jmath_{1, m+1} : \Tt C^*(E(0,m)) \to \Tt C^*(E(1, m+1))$ that descends to the
isomorphism $C^*(E(1, m+1)) \cong C^*(E(0,m))$ of \cite[Theorem~3.1]{Bates}. We then show
that if the set of vertices of $E$ that emit at least two edges has hereditary closure
$E^0$ in $E(0,m)$, then the image of $\Tt C^*(\SG[m]{E})_0$ under $\eta$ is precisely
$\Tt C^*(E(1, m+1)) \oplus \jmath_{1, m+1}(\Tt C^*(E(0,m)))$. It is then straightforward
to obtain our description of $\Tt C^*(\SG[m]{E})$, and we prove that it is homotopy
equivalent to $\Tt C^*(\SG{E}(0,m))$, allowing us to compute its $K$-theory. We prove
that the inclusion $\Tt C^*(\SG[m]{E}) \hookrightarrow C([0,1], \Tt C^*(E(1,m+1)))$ takes
$\Tt C^*(\SG[m]{E}) \cap \oplus_{\omega \in \SG[m]{E}^0} \Kk(\ell^2(\SG[m]{E}^* \omega))$
into $C([0,1], \Kk(\ell^2(E(1, m+1)^*)))$. We then deduce that $C^*(\SG[m]{E}) \cong
C([0,1], C^*(E(1, m + 1))$.

By showing that $C^*(D_n(E)(1, m+1))$ is Morita equivalent to $C^*(E(0,m))$ when $m,n$
are coprime, we describe $K_*(C^*(\SG[l]{E}))$ for rational $l$ provided that $E$ is
locally finite and enough vertices of $E$ admit and receive at least two edges. A
sufficient condition is that $E$ is a finite, strongly connected graph that is not a
simple cycle and has period 1 (in the sense of Perron--Frobenius theory).

\smallskip

The paper is organised as follows. We present the background that we will need on
directed graphs and their $C^*$-algebras, on topological graphs, and on $C(X)$-algebras
in Section~\ref{sec:background}. We then present the definition of $\SG{E}$ in
Section~\ref{sec:the graph torus}. In Section~\ref{sec:path spaces of SGE} we describe
the path space and the infinite-path space of $\SG{E}$, and prove that the latter
coincides with $M(\sigma)$ if $E$ is finite. We define the $C^*$-algebras of $\SG{E}$ in
Section~\ref{sec:C*SGE}. In Section~\ref{sec:SG[l]E}, we describe our general
construction of $\SG[l]{E}$ and its $C^*$-algebras, and prove that the case $l = 1$
coincides with our previous definitions of $\SG{E}$ and its $C^*$-algebras, and also that
$l = -1$ corresponds to the suspension of the opposite graph of $E$. This is not the most
efficient order of presentation: we could have simply defined $\SG[l]{E}$ and its
$C^*$-algebras immediately after Section~\ref{sec:background}, and then defined $\SG{E}
:= \SG[1]{E}$. But we feel that $\SG{E}$, which is the key point of contact with
suspension flows, is important enough to warrant separate discussion, and also that the
later definitions of $\SG[l]{E}$ and its $C^*$-algebras are better motivated by first
discussing $\SG{E}$ and its $C^*$-algebras. Section~\ref{sec:SG[l]E} also contains our
reduction of the analysis of the $C^*$-algebras of $\SG[l]{E}$ for rational $l$ to the
analysis of the $C^*$-algebras of $\SG[m]{E}$ for positive integers $m$. Our analyses of
$\Tt C^*(\SG[m]{E})$ and $C^*(\SG[m]{E})$ occupy Sections \ref{sec:TC*
analysis}~and~\ref{sec:C* analysis} respectively.

\section{Background}\label{sec:background}

We recall some background about directed graphs and their $C^*$-algebras, on topological
graphs, and on $C(X)$-algebras.

\subsection{Graphs}
We take our conventions for graph $C^*$-algebras from \cite{CBMSbook}. Throughout this
paper, all of the graphs that we consider are directed graphs in the sense that the edges
have an orientation. We omit the adjective throughout.

A \emph{graph} is a quadruple $E = (E^0, E^1, r, s)$ consisting of finite or countably
infinite sets $E^0$ and $E^1$, and functions $r, s : E^1 \to E^0$. We think of the
elements of $E^0$ as vertices, and draw them as dots, and we regard the elements of $E^1$
as directed edges connecting vertices, and draw each as an arrow from the vertex $s(e)$
the vertex $r(e)$.

It is convenient to think of $E^0$ as the objects and $E^1$ as the indecomposable
morphisms in the countable category $E^*$ of finite paths in $E$ as follows. For $n \ge
2$, we define $E^n := \{\mu_1 \mu_2 \dots \mu_n : \mu_i \in E^1\text{ and } s(\mu_i) =
r(\mu_{i+1})\text{ for all }i\}$, and refer to the elements of $E^n$ as \emph{paths of
length $n$} in $E$. We think of vertices as paths of length 0 and edges as paths of
length 1, and then write $E^* := \bigcup_{n=0}^\infty E^n$ for the path space of $E$. For
$v \in E^0$, we write $r(v) = s(v) = v$, and for $n \ge 1$ and $\mu \in E^n$ we write
$r(\mu) = r(\mu_1)$ and $s(\mu) = s(\mu_n)$. We can concatenate $\mu,\nu \in E^*
\setminus E^0$ to form the path $\mu\nu$ if $r(\nu) = s(\mu)$. For $v \in E^0$ and $\mu
\in E^*$, the concatenation $v \mu$ is defined if $r(\mu) = v$, in which case, we have
$v\mu = \mu$, and similarly the concatenation $\mu v$ is defined if $v = s(\mu)$, in
which case $\mu v = \mu$. If $\mu \in E^n$, we write $|\mu| = n$.

Given $U, V \subseteq E^*$, we define $UV := \{\mu\nu : \mu \in U, \nu \in V, \text{ and
} s(\mu) = r(\nu)\}$. When $U$ is a singleton $U = \{\mu\}$, we write $\mu V$ rather than
$\{\mu\} V$ for the set $\{\mu\nu : \nu \in V\text{ and }r(\nu) = s(\mu)\}$. In
particular, for $v \in E^0$ we have
\[
v E^1 = \{e \in E^1 : r(e) = v\}\qquad\text{ and }\qquad E^1 v = \{e \in E^1 : s(e) = v\}.
\]
We extend this notational convention in the obvious ways, so that for example if $v,w \in
E^0$ then $v E^1 w = v E^1 \cap E^1 w$.

The \emph{adjacency matrix} of the graph $E$ is the integer matrix $A_E$ given by
$A_E(v,w) = |v E^1 w|$. We then have $A_E^n(v,w) = |v E^n w|$ for all $v,w,n$.

We say that $E$ is \emph{finite} if $E^0$ and $E^1$ are both finite sets. We say that $E$
is \emph{locally finite} if each $E^1 v \cup v E^1$ is a finite set; that is, if the
row-sums and column sums of the matrix $A_E$ are finite. A \emph{sink} in $E$ is a vertex
$v$ such that $E^1 v = \emptyset$, and a source is a vertex $v$ such that $vE^1 =
\emptyset$. In this paper, we are concerned exclusively with graphs that are
locally-finite and have no sources.

A \emph{cycle} in $E$ is a path $\mu \in E^* \setminus E^0$ such that $r(\mu) = s(\mu)$.
We say that $\mu$ has an \emph{entrance} if there exists $i \le |\mu|$ such that
$|r(\mu_i)E^1| \ge 2$.

\subsection{Infinite paths} An \emph{infinite path} in $E$ is a string
$x = x_1x_2x_3 \cdots$ of edges of $E$ such that $x_i \in E^1 r(x_{i+1})$ for all $i$. We
write $E^\infty$ for the set of all infinite paths in $E$ and call it the
\emph{infinite-path space} of $E$. For $x \in E^\infty$, we define $r(x) = r(x_1) \in
E^0$; and for $\mu \in E^*$ and $x \in E^\infty$ with $r(x) = s(\mu)$ we write $\mu x$
for the infinite path $\mu_1 \cdots \mu_n x_1 x_2 \cdots$. We write $\mu E^\infty :=
\{\mu x : x \in E^\infty\}$.

We endow $E^\infty$ with the topology that it inherits as a subspace of
$\prod^\infty_{i=1} E^\infty$, a basis for which is the collection $\{\mu E^\infty : \mu
\in E^*\}$. The set $\mu E^\infty$ is called the \emph{cylinder set} of $\mu$ and is
often denoted $Z(\mu)$ elsewhere in the literature. When $E$ is locally finite, the sets
$\mu E^\infty$ are compact open sets in $E^\infty$, and the topology is a locally compact
Hausdorff totally disconnected topology.

The \emph{shift map} $\sigma : E^\infty \to E^\infty$ is defined by $\sigma(x)_i =
x_{i+1}$, so $\sigma(x_1 x_2 x_3 \cdots) = x_2 x_3\cdots$, and $\sigma(ex) = x$ for all
$x \in E^\infty$ and $e \in E^1$ with $s(e) = r(x)$. This $\sigma$ is a local
homeomorphism, as it restricts to a homeomorphism $e E^\infty \to s(e)E^\infty$ for each
$e \in E^1$.

\subsection{Graph \texorpdfstring{$C^*$}{C*}-algebras}\label{sec:graph algs}

Let $E$ be a locally finite graph with no sources. A \emph{Toeplitz--Cuntz--Krieger}
family for $E$ in a $C^*$-algebra $A$ consists of a map $t : E^1 \to A$, written $e
\mapsto t_e$ and a map $q : E^0 \to A$, written $v \mapsto q_v$ such that the elements
$q_v$ are mutually orthogonal projections in $A$, and such that
\begin{itemize}
\item[(TCK1)] $t^*_e t_e = q_{s(e)}$ for all $e \in E^1$, and
\item[(TCK2)] $q_v \ge \sum_{e \in vE^1} t_e t^*_e$ for all $v \in E^0$.
\end{itemize}
A \emph{Cuntz--Krieger} family for $E$ is a Toeplitz--Cuntz--Krieger family $(t,q)$ for
$E$ such that
\begin{itemize}
\item[(CK)] $q_v = \sum_{e \in vE^1} t_e t^*_e$ for all $v \in E^0$.
\end{itemize}

Relations (TCK1)~and~(TCK2) imply that for each $\mu \in E^n$ the element $t_\mu =
t_{\mu_1} t_{\mu_2} \dots t_{\mu_n}$ is a partial isometry. As a notational convenience,
we write $t_v := q_v$ for $v \in E^0$. With this notation, the $C^*$-algebra generated by
the elements $t_e$ and the elements $q_v$ is equal to the closed linear span
\[
C^*(t,q) = \clsp\{t_\mu t^*_\nu : \mu,\nu \in E^*, s(\mu) = s(\nu)\},
\]
and we have $t_\mu t_\nu = \delta_{s(\mu),r(\nu)} t_{\mu\nu}$. An induction shows that
$t^*_\mu t_\mu = t_{s(\mu)}$ for all $\mu \in E^*$, and we have
\[
t_\mu t^*_\nu t_\eta t^*_\zeta
    = \begin{cases}
        t_\mu t^*_{\zeta\nu'} &\text{ if $\nu = \eta \nu'$}\\
        t_{\mu\eta'} t^*_\zeta &\text{ if $\eta = \nu\eta'$}\\
        0 &\text{ otherwise.}
    \end{cases}
\]

There is a $C^*$-algebra $\Tt C^*(E)$ generated by a Toeplitz--Cuntz--Krieger family
$(T,Q)$ that is universal in the sense that given any other Toeplitz--Cuntz--Krieger
family $(t,q)$ in a $C^*$-algebra $A$, there is a homomorphism $\pi_{t,q} : \Tt C^*(E)
\to A$ such that $\pi_{t,q}(T_e) = t_e$ and $\pi_{t,q}(Q_v) = q_v$. The universal
property ensures that there is an action $\gamma : \TT \to \Aut(\Tt C^*(E))$ called the
\emph{gauge action} such that $\gamma_z(T_e) = zT_e$ and $\gamma_z(Q_v) = Q_v$ for all $e
\in E^1$ and $v \in E^0$.

There is a faithful representation $\pi : \Tt C^*(E) \to \Bb(\ell^2(E^*))$ called the
path-space representation, and determined by $\pi(T_e) h_\mu = \delta_{s(e),r(\mu)}
h_{e\mu}$ and $\pi(Q_v) h_\mu = \delta_{v, r(\mu)} h_\mu$. (The existence of $\pi$
follows from the universal property, and injectivity follows from an application of
\cite[Theorem~4.1]{FR}.)

There is also a $C^*$-algebra $C^*(E)$ generated by a Cuntz--Krieger $E$-family $(s,p)$
that is universal in the sense that given any other Cuntz--Krieger family $(s',p')$ there
is a homomorphism of $C^*(E)$ taking each $s_e$ to $s'_e$ and each $p_v$ to $p'_v$. This
$C^*(E)$ is isomorphic to the quotient of $\Tt C^*(E)$ by the ideal $I_E$ generated by
the projections $\Delta_v := Q_v - \sum_{e \in vE^1} T_e T^*_e$ indexed by $v \in E^0$.

Let $E$ be a locally finite graph with no sources and let $(t, q)$ be a Cuntz--Krieger
$E$-family. Since the projections $\Delta_v$ are fixed by the gauge action on $\Tt
C^*(E)$, it descends to an action, also called the gauge action and denoted $\gamma$, on
$C^*(E)$. The gauge-invariant uniqueness theorem \cite{aHR, CBMSbook} states that if
there is an action $\beta : \TT \to \Aut(C^*(t,q))$ such that $\beta_z(t_e) = zt_e$ for
all $e \in E^1$, then $\pi_{q,t} : C^*(E) \to C^*(t, q)$ is injective if and only if each
$q_v$ is nonzero. The Cuntz--Krieger uniqueness theorem \cite{CK, CBMSbook} says that if
every cycle in $E$ has an entrance, then $\pi_{q,t}$ is injective if and only if each
$q_v$ is nonzero.

With a little work, one can check that the path-space representation of $\Tt C^*(E)$
carries the ideal $I_E$ to $\pi(\Tt C^*(E)) \cap \Kk(\ell^2(E^*))$. It follows that there
is a homomorphism from $C^*(E)$ to the Calkin algebra $\Qq(\ell^2(E^*)) :=
\Bb(\ell^2(E^*))/\Kk(\ell^2(E^*))$ given by $p_v \mapsto Q_v + \Kk(\ell^2(E^*))$ and $s_e
\mapsto T_e + \Kk(\ell^2(E^*))$. We call this homomorphism the \emph{Calkin
representation} of $C^*(E)$. An argument using the gauge-invariant uniqueness theorem
shows that the Calkin representation of $C^*(E)$ is injective. The subspaces $\ell^2(E^*
v) \subseteq \ell^2(E^*)$ indexed by $v \in E^0$ are invariant for $\pi$, and we have
$\pi(\Tt C^*(E)) \cap \Kk(\ell^2(E^*)) = \bigoplus_{v \in E^0} \Kk(\ell^2(E^* v))$. So
the Calkin representation can be regarded as an injective homomorphism of $C^*(E)$ into
$\bigoplus_{v \in E^0} \Qq(\ell^2(E^* v))$ given by $s_e \mapsto \bigoplus_v
(T_e|_{\ell^2(E^* v)} + \Kk(\ell^2(E^* v)))$.

A set $H \subseteq E^0$ is called \emph{hereditary} if $s(H E^1) \subseteq H$, or
equivalently $HE^* \subseteq E^* H$.

\subsection{Dual graphs}\label{sec:dual graph}
Given a locally finite graph $E$ with no sources, the \emph{dual graph} $\widehat{E}$ is
the graph $\widehat{E} = (E^1, E^2, \hat{r}, \hat{s})$ where $\hat{r}(ef) = e$ and
$\hat{s}(ef) = f$ for all $ef \in E^2$. The properties of being row-finite or locally
finite and of having no sinks or no sources pass from $E$ to $\widehat{E}$ and vice
versa.

There is a homeomorphism $E^\infty \cong \widehat{E}^\infty$ that carries the infinite
path $x = x_1 x_2 x_3 \cdots$ of $E$ to the infinite path $\hat{x} = (x_1 x_2) (x_2 x_3)
(x_3 x_4) \cdots$ of $\widehat{E}$.

Corollary~2.5 of \cite{BPRS} shows that there is an isomorphism $C^*(\widehat{E}) \cong
C^*(E)$ satisfying $s_{ef} \mapsto s_e s_f s^*_f$ and $p_e \mapsto s_e s^*_e$ for all $ef
\in \widehat{E}^1$ and all $e \in \widehat{E}^0$.

More generally (see \cite{Bates}), we can construct from any pair of integers $0 < p < q$
a new graph $E(p,q)$ from $E$. We define $E(p,q)$ to be the graph with
\[
E(p,q)^0 := E^p \quad\text{ and }\quad E(p,q)^1 := E^q,
\]
with range and source maps given by
\[
r_{p,q}(e_1 \cdots e_q) := \begin{cases}
    e_1 \cdots e_p &\text{if $p \ge 1$}\\
    r(e_1) &\text{if $p = 0$}
    \end{cases}
\quad\text{ and }\quad
s_{p,q}(e_1 \cdots e_q) := \begin{cases}
    e_{q-p+1} \cdots e_q &\text{if $p \ge 1$}\\
    s(e_q) &\text{if $p = 0$.}
    \end{cases}
\]
So $E(0,1) \cong E$, $E(1,2) \cong \widehat{E}$, and $E(0, p)$ is the
$p$\textsuperscript{th} higher-power graph $(E^0, E^p, r, s)$. Bates shows in
\cite[Theorem~3.1]{Bates} that $C^*(E(p+1,q+1)) \cong C^*(E(p,q))$ for all $0 < p < q$,
and hence, by induction, that $C^*(E(p,q)) \cong C^*(E(0,q-p))$ for all $0 < p < q$,
generalising the usual isomorphism $C^*(\widehat{E}) \cong C^*(E)$ of
\cite[Corollary~2.5]{BPRS}. We will need the following analogue of this result for
Toeplitz algebras.

Recall that for any graph $E$, we denote by $I_E$ the ideal of $\Tt C^*(E)$ generated by
the projections $\Delta_v := Q_v - \sum_{e \in vE^1} T_e T^*_e$ indexed by $v \in E^0$,
and that then $C^*(E) = \Tt C^*(E)/I_E$.

\begin{lem}\label{lem:mth Toeplitz injection}
Let $E$ be a row-finite graph with no sources and fix integers $0 < p < q$. For $v \in
E^0$ let $q_v := \sum_{\mu \in v E^p} Q_\mu \in \Tt C^*(E(p,q))$, and for $\mu \in
E^{q-p} = E(0,q-p)^1$, let $t_\mu := \sum_{\nu \in s(\mu)E^p} T_{\mu \nu} \in \Tt
C^*(E(p,q))$. Then there is an injective homomorphism $\jmath_{p,q} : \Tt C^*(E(0, q-p))
\hookrightarrow \Tt C^*(E(p,q))$ such that $\jmath_{p,q}(Q_v) = q_v$ for $v \in E^0$, and
$\jmath_{p,q}(T_e) = t_e$ for $e \in E^1$. We have
\begin{equation}\label{eq:jmath(Delta)}
\jmath_{p,q}(\Delta_v) = \sum_{\mu \in vE^p} \Delta_\mu \quad\text{ for all $v \in E^0$,}
\end{equation}
and $\jmath_{p,q}(I_{E(0, q-p)}) \subseteq I_{E(p,q)}$. There is an isomorphism
$\tilde\jmath_{p,q} : C^*(E(0,q-p)) \to C^*(E(p,q))$ such that $\tilde\jmath_{p,q}(a +
I_{E(0, q-p)}) = \jmath_{p,q}(a) + I_{E(p,q)}$ for all $a \in \Tt C^*(E(0,q-p))$.
\end{lem}
\begin{proof}
The elements $q_v$ are mutually orthogonal projections in $\Tt C^*(E(p,q))$ because
$\{Q_\mu : \mu \in E(p,q)^0\}$ is a collection of mutually orthogonal projections. For
$\mu,\nu \in E^{q-p}$, we have
\[
t_\mu^* t_\nu = \sum_{\alpha \in s(\mu)E^p, \beta \in s(\nu)E^p} T_{\mu\alpha}^* T_{\nu\beta}
    = \sum_{\alpha \in s(\mu)E^p, \beta \in s(\nu)E^q} \delta_{\mu\alpha, \nu\beta} Q_\alpha
    = \delta_{\mu,\nu} \sum_{\alpha \in s(\mu)E^p} Q_\alpha
    = \delta_{\mu,\nu} q_{s(\mu)}.
\]
Hence $(t,q)$ satisfies~(TCK1), and the partial isometries $\{t_\mu : \mu \in E^{q-p}\}$
have mutually orthogonal range projections. For $v \in E^0$ and $\mu \in vE^{q-p}$, we
have
\[
q_v t_\mu
    = \sum_{\nu \in vE^{q-p}} Q_\nu \sum_{\alpha \in s(\mu)E^p} T_{\mu\alpha}
    = \sum_{\nu \in vE^{q-p}, \alpha \in s(\mu)E^p} \delta_{\mu,\nu} T_{\mu\alpha}
    = \sum_{\alpha \in s(\mu) E^p} T_{\mu\alpha}
    = t_\mu.
\]
In particular, each $t_\mu t^*_\mu \le q_{r(\mu)}$, and since the $t_\mu t^*_\mu$ are
mutually orthogonal, it follows that $(q, t)$ satisfies~(TCK2). So the universal property
of $\Tt C^*(E(0, q-p))$ gives a homomorphism $\jmath_{p,q} : \Tt C^*(E(0, q-p)) \to \Tt
C^*(E(p,q))$ such that $\jmath_{p,q}(Q_v) = q_v$ for $v \in E^0$, and $\jmath_{p,q}(T_e)
= t_e$ for $e \in E^1$.

To see that this homomorphism is injective, observe that for $v \in E^0$, we have
\begin{equation}\label{eq:gap computation}
\begin{split}
q_v - \sum_{\nu \in vE^{q-p}} t_\nu t^*_\nu
    &= \sum_{\mu \in v E^p} Q_\mu - \sum_{\nu \in vE^{q-p}} \sum_{\alpha \in s(\nu)E^p} T_{\nu\alpha} T^*_{\nu\alpha}\\
    &= \sum_{\mu \in v E^p} Q_\mu - \sum_{\lambda \in vE^q} T_\lambda T^*_\lambda
    = \sum_{\mu \in v E^p} \Big(Q_\mu - \sum_{\beta \in s(\mu)E^{q-p}} T_{\mu\beta} T^*_{\mu\beta}\Big).
\end{split}
\end{equation}
Let $\pi : \Tt C^*(E(p,q)) \to \Bb(\ell^2(E(p,q)^*))$ be the path-space representation.
For each $\mu \in vE^p$, we have $\pi(Q_\mu - \sum_{\alpha \in s(\mu)E^{q-p}}
T_{\mu\alpha} T^*_{\mu\alpha}) = \theta_{h_\mu, h_\mu} \in \Kk(\ell^2(E(p,q)^*))$, and
hence $\pi(Q_\mu - \sum_{\alpha \in s(\mu)E^{q-p}}) \not= 0$. Hence $q_v - \sum_{\mu \in
vE^{q-p}} t_\mu t^*_\mu \not= 0$. Theorem~4.1 of \cite{FR} therefore shows that
$\jmath_{p,q}$ is injective.

The calculation~\eqref{eq:gap computation} establishes~\eqref{eq:jmath(Delta)}, and since
$\jmath_{p,q}(I_E)$ is an ideal of $\jmath(\Tt C^*(E(0,q-p)))$, it follows that it is
contained in the ideal of $\Tt C^*(E(p,q))$ generated by the $\Delta_\mu$, which is
$I_{E(p,q)}$ by definition. So $\jmath_{p,q}$ descends to a homomorphism
$\tilde\jmath_{p,q} : C^*(E(0, q-p)) \to C^*(E(p,q))$ satisfying
\begin{align*}
\tilde\jmath_{p,q}(p_v) &= \tilde\jmath_{p,q}(Q_v + I_{E(0,q-p)}) = \jmath_{p,q}(Q_v) + I_{E(p,q)} = \sum_{\mu \in vE^p} p_\mu,\text{ and}\\
\tilde\jmath_{p,q}(s_\nu) &= \tilde\jmath_{p,q}(T_\nu + I_{E(0,q-p)}) = \jmath_{p,q}(T_\nu) + I_{E(p,q)} = \sum_{\alpha \in s(\nu)E^p} s_{\nu\alpha}.
\end{align*}
This is an isomorphism because it agrees on generators with that obtained from
\cite[Corollary~3.3]{Bates}.
\end{proof}

\subsection{Topological graphs and their
\texorpdfstring{$C^*$}{C*}-algebras}\label{sec:topgraphs}

We present here a very brief introduction to those parts of Katsura's theory of
topological graphs and their $C^*$-algebras that we will need later. For a more
comprehensive overview, see \cite[Chapter~8]{CBMSbook}; for full details, see \cite{Kat1,
Kat2, Kat3, Kat4}.

As defined by Katsura \cite{Kat1}, a \emph{topological graph} is a quadruple $E = (E^0,
E^1, r, s)$ where $E^0$ and $E^1$ are locally compact Hausdorff spaces, $r : E^1 \to E^0$
is a continuous map, and $s : E^1 \to E^0$ is a local homeomorphism.

The associated \emph{graph bimodule} is defined as follows. The space $C_c(E^1)$ is a
$C_0(E^0)$-bimodule with respect to the actions $\big(a \cdot \xi\big)(e) =
a(r(e))\xi(e)$ and $\big(\xi \cdot a\big)(e) = \xi(e) a(s(e))$ for $a \in C_0(E^0)$ and
$\xi \in C_c(E^1)$. The formula $\langle \xi \mid \eta\rangle_{C_0(E^0)}(v) := \sum_{s(e)
= v} \overline{\xi(e)}\eta(e)$ defines a $C_0(E^0)$-valued inner-product on $C_c(E^1)$,
and the completion $X(E)$ of $C_c(E^1)$ in the norm $\|\xi\| = \|\langle \xi,
\xi\rangle_A\|^{1/2}$ is a Hilbert $C_0(E^0)$-bimodule called the \emph{graph
correspondence} of $E$. Katsura proves that $X(E)$ is equal to the space $C_d(E^1)$ of
functions
\begin{equation}\label{eq:CdX}
C_d(E^1) := \Big\{\xi \in C_0(E^1) : \Big(v \mapsto \sum_{s(e) = v} |\xi(e)|^2\Big) \in C_0(E)^0\Big\}.
\end{equation}

The left action of $C_0(E^0)$ on $X(E)$ determines a homomorphism $\phi : C_0(E^0) \to
\Ll(X(E))$ by $\phi(a)\xi = a \cdot \xi$. If $r : E^1 \to E^0$ is a proper map in the
sense that the preimages of compact sets are compact, then $\phi$ takes values in
$\Kk(X(E))$.

A representation of $X(E)$ in a $C^*$-algebra $A$ is a pair $(\pi, \psi)$ such that $\pi
: C_0(E^0) \to A$ is a homomorphism, $\psi : X(E) \to A$ is a linear map, and we have
$\pi(a)\psi(\xi) = \psi(a \cdot \xi)$, $\psi(\xi)\pi(a) = \psi(\xi \cdot a)$ and
$\psi(\xi)^* \psi(\eta) = \pi(\langle \xi, \eta\rangle_{C_0(E^0)})$ for all $\xi, \eta
\in X(E)$. The pair $(\pi, \psi)$ induces a homomorphism $\psi^{(1)} : \Kk(X(E)) \to A$
such that $\psi^{(1)}(\theta_{\xi,\eta}) = \psi(\xi)\psi(\eta)^*$ for all $\xi,\eta \in
X(E)$ (see \cite[Lemma~3.2]{Pimsner}). The representation $(\psi, \pi)$ is
\emph{Cuntz--Pimsner covariant} if we have $\psi^{(1)}(\phi(a)) = \pi(a)$ whenever
$\phi(a) \in \Kk(X(E))$; in particular, if $r : E^1 \to E^0$ is proper, then $(\psi,
\pi)$ is Cuntz--Pimsner covariant if $\psi^{(1)} \circ \phi = \pi$.

The topological graph $C^*$-algebra, denoted here by $C^*(E)$ is the $C^*$-algebra
generated by a universal Cuntz--Pimsner covariant representation of $E$. Its Toeplitz
algebra $\Tt C^*(E)$ is the $C^*$-algebra generated by a universal representation of $E$.

If $E$ is a graph and $Y$ is a locally compact Hausdorff space then the topological graph
$E \times Y$ is defined by $(E \times Y)^0 := E^0 \times Y$, $(E \times Y)^1 := E^1
\times Y$ and $r(e, y) := (r(e), y)$ and $s(e, y) := (s(e), y)$. If $E$ is locally
finite, then $r : (E \times Y)^1 \to (E \times Y)^0$ is a proper map. The universal
properties of $C^*(E \times Y)$ and of the tensor product $C^*(E) \otimes C_0(Y)$ imply
(see \cite[Proposition~7.7]{Kat2}) that there is an isomorphism $C^*(E) \otimes C_0(Y)
\cong C^*(E \times Y)$ that carries $s_e \otimes h$ to $\psi\big((f, y) \mapsto
\delta_{e,f} h(y)\big)$ and carries $p_v \otimes h$ to $\pi\big((w, y) \mapsto
\delta_{v,w} h(y)\big)$ for all $e \in E^1$, $v \in E^0$ and $h \in C_0(Y)$.

\subsection{\texorpdfstring{$C(X)$}{C(X)}-algebras}\label{sec:C(X)-algebras}

We need only the bare bones of the theory of $C(X)$-algebras here. For details, see
\cite[Appendix~C]{Williams}.

Let $X$ be a locally compact Hausdorff space. A $C^*$-algebra $A$ is called a
$C(X)$-algebra if there exists a nondegenerate homomorphism $\iota : C(X) \to \Zz \Mm(A)$
of $C(X)$ into the centre of the multiplier algebra of $A$. For each $x \in X$, the
maximal ideal $J_x = \{f \in C(X) : f(x) = 0\}$ of $C(X)$ generates an ideal $I_x =
\iota(J_x) A$ of $A$. We define $A_x := A/I_x$ to be the corresponding quotient. For $a
\in A$, the map $x \mapsto \|a + I_x\|$ is upper semicontinuous. There is a unique
topology on $\mathcal{A} := \bigsqcup_{x \in X} A_x$ under which the functions $x \mapsto
a + I_x$ are all continuous. In this topology,
\begin{equation}\label{eq:norm limit zero}
\lim_{n \to \infty} \|a + I_{x_n}\| = 0
    \qquad\Longrightarrow\qquad
    \lim_{n \to \infty} a + I_{x_n} = 0.
\end{equation}
So $\mathcal{A}$ is an upper-semicontinuous bundle of $C^*$-algebras over $X$, and each
$a \in A$ determines a section $\gamma_a : x \mapsto a + I_x$ of $\mathcal{A}$ that
vanishes at infinity in the sense that for each $\varepsilon > 0$ there exists a compact
set $K \subseteq X$ such that $\|\gamma_a(x)\| < \varepsilon$ for all $x \not\in K$. The
map $a \mapsto \gamma_a$ is an isomorphism of $A$ onto the algebra $\Gamma_0(X,
\mathcal{A})$ of continuous sections of $\mathcal{A}$ that vanish at infinity. So every
$C(X)$-algebra is the algebra of sections of an upper-semicontinuous bundle of
$C^*$-algebras over $X$. Conversely, if $\mathcal{A}$ is an upper-semicontinuous bundle
over $X$ then there is a nondegenerate homomorphism $\iota : C(X) \to \Zz \Mm(\Gamma_0(X,
\mathcal{A}))$ characterised by $(\iota(f)\gamma)(x) = f(x)\gamma(x)$.

\section{The suspension of a graph}\label{sec:the graph torus}

In this section we define the suspension $\SG{E}$ of a graph $E$ and describe its basic
properties. Our motivation is the relationship between the infinite-path space of
$\SG{E}$ and the suspension flow of the shift-space associated to $E$, which we establish
in Section~\ref{sec:path spaces of SGE}. The constructions in this section will be
subsumed by the more-general construction of the quivers $\SG[l]{E}$ parameterised by $l
\in \RR$ in Section~\ref{sec:SG[l]E}.

Let $E$ be a locally finite graph with no sources. Let $\sim$ denote the smallest
equivalence relation on $((E^* \setminus E^0) \times [0,1]) \sqcup E^*$ such that
\begin{equation}\label{eq:simdef}
(\mu f, 0) \sim \mu \text{ and } (e \mu,1) \sim \mu \quad\text{ for all $\mu \in E^*$, $e \in E^1 r(\mu)$ and $f \in s(\mu)E^1$.}
\end{equation}
Observe that $\sim$ restricts to an equivalence relation on $(E^{n+1} \times [0,1]) \cup
E^n$ for each $n \ge 0$.

For $\mu \in E^* \setminus E^0$ and $t \in [0,1]$, we write $[\mu,t]$ for the
equivalence-class of $(\mu,t)$ under $\sim$, and for $\mu \in E^*$ we write $[\mu]$ for
the equivalence class of $\mu$ under $\sim$. Putting $\mu = v \in E^0$
in~\eqref{eq:simdef}, we see that $[e, 1] \sim [f,0]$ whenever $s(e) = r(f)$. For $n \in
\NN = \{0, 1, 2, \dots\}$, we write
\[
    \SG{E}^n := \big((E^{n+1} \times [0,1]) \sqcup E^n\big)/{\sim},
\]
and we define
\[
    \SG{E}^* := \big(((E^* \setminus E^0) \times [0,1]) \sqcup E^*\big)/{\sim} = \bigsqcup_n \SG{E}^n.
\]

It is straightforward to check that there are well-defined maps $r,s : \SG{E}^* \to
\SG{E}^0$ such that $r([e\mu,t]) := [e,t]$ and $s([\mu f,t]) := [f,t]$ for $\mu \in E^*$
and $e \in E^1 r(\mu)$ and $f \in s(\mu)E^1$. These maps satisfy $r([\mu]) = [r(\mu)]$
and $s([\mu]) = [s(\mu)]$ for $\mu \in E^*$.

\begin{dfn}
Let $E$ be a locally finite graph with no sources. We call the quadruple $\SG{E} :=
(\SG{E}^0, \SG{E}^1, r, s)$ the \emph{suspension} of $E$.
\end{dfn}

The following lemma will make it easier to work with the suspension of a graph.

\begin{lem}\label{lem:normal form}
Let $E$ be a locally finite graph with no sources. For $\mu, \nu \in E^*$ and $s,t \in
[0,1]$, we have $(\mu,s) \sim (\nu,t)$ if and only if one of the following holds:
\begin{enumerate}
    \item\label{it:normal form1} $\mu = \nu$ and $s=t$;
    \item\label{it:normal form2} $s = t = 0$ and $\mu_1 \cdots \mu_{|\mu|-1} = \nu_1
        \cdots \nu_{|\nu|-1}$, or $s = t = 1$ and $\mu_2 \cdots \mu_{|\mu|} = \nu_2
        \cdots \nu_{|\nu|}$; or
    \item\label{it:normal form3} $s = 1, t = 0$ and $\mu_2 \cdots \mu_{|\mu|} = \nu_1
        \cdots \nu_{|\nu|-1}$, or $s = 0, t = 1$ and $\mu_1 \cdots \mu_{|\mu|-1} =
        \nu_2 \cdots \nu_{|\nu|}$.
\end{enumerate}
Each $\alpha \in \SG{E}^*$ has a representative of the form $(\mu,t)$ where $\mu \in E^*$
and $t \in [0,1)$. If $\alpha \not\in \{[\mu] : \mu \in E^*\}$ then this representative
is unique. The map $\mu \mapsto [\mu]$ is an injection of $E^*$ into $\SG{E}^*$.
\end{lem}
\begin{proof}
Consider the relation $R_0$ on $(E^* \setminus E^0) \times [0,1]$ given by
(\ref{it:normal form1})--(\ref{it:normal form3}). Then $R_0 \subseteq {\sim}$. It is
clear that $R_0$ is reflexive and symmetric, and a quick case-by-case check of the
possible combinations of (\ref{it:normal form2})~and~(\ref{it:normal form3}) shows that
it is transitive. If $(\mu,s) \sim (\nu,t)$, then there is a sequence $(\mu,s) = (\mu_0,
s_0) \sim \nu_1 \sim (\mu_1, s_1) \sim \nu_2 \sim \cdots \sim \nu_k \sim (\mu_k, s_k) =
(\nu,t)$ where each of the equivalences in the chain is one of the forms appearing
in~\eqref{eq:simdef}. It then follows that each equivalence $(\mu_i, s_i) \sim
(\mu_{i+1}, s_{i+1})$ is of one of the forms appearing in (\ref{it:normal
form2})~or~(\ref{it:normal form3}). Thus $\sim$ is contained in $R_0$, and the two are
equal.

If $\alpha \in \SG{E}^*$ then by definition we have either $\alpha = [\mu]$ for some $\mu
\in E^*$ or $\alpha = [\nu,t]$ for some $(\nu,t) \in (E^* \setminus E^0) \times [0,1]$.
If $\alpha = [\mu]$ then, since $E$ has no sources, we can find $e \in E^1$ with $r(e) =
s(\mu)$, and then $\alpha = [\mu e,0]$. Likewise if $\alpha = [\nu,1]$ then we can find
$f \in E^1$ with $r(f) = s(\nu)$, and then $\alpha = [\nu f,0]$. Otherwise $\alpha =
[\nu,t]$ already has the desired form. If $\alpha \not\in \{[\mu] : \mu \in E^*\}$, then
$\alpha = [\nu,t]$ for some $t \in (0,1)$ and $\nu \in E^* \setminus E^0$, and then if we
also have $\alpha = [\nu',s]$ then in particular $(\nu,t) \sim (\nu',s)$ with $t \not=
\{0,1\}$. In particular, this equivalence does not appear in (\ref{it:normal
form2})~or~(\ref{it:normal form3}), and we deduce that it is of the form~(\ref{it:normal
form1}), so $\nu = \nu'$ and $t = s$.

Suppose that $\mu,\nu \in E^*$ satisfy $\mu \sim \nu$. We can write $[\mu] = [\mu e, 0]$
and $[\nu] = [\nu f, 0]$ for any $e \in s(\mu)E^1$ and $f \in s(\nu)E^1$. We then have
$(\mu e,0) \sim (\nu f,0)$ which means that the equivalence is of the
form~(\ref{it:normal form1}) or the first of the forms appearing in~(\ref{it:normal
form2}), each of which forces $\mu = \nu$. So $\mu \mapsto [\mu]$ is an injection.
\end{proof}

We call elements of $\SG{E}^*$ \emph{paths} in $\SG{E}$, and elements of
\[
\SG{E}^n := \big((E^{n+1} \times [0,1]) \sqcup E^n\big)/{\sim}
\]
\emph{paths of length $n$} in $\SG{E}$.

\begin{ntn}
Using Lemma~\ref{lem:normal form}, we regard $E^*$ as a subset of $\SG{E}^*$. In
particular, for $n \in \NN$, we write
\[
\SG{E}^n \setminus E^n = \{[\mu,t] : \mu \in E^{n+1}\text{ and }t \in (0,1)\}.
\]
\end{ntn}

There is a partially defined composition map on $\SG{E}^*$ as follows: the pair $([\mu,
t], [\nu, s])$, is composable if and only if $s([\mu, t]) = r([\mu, s])$, which in
particular forces $t = s$. If $s([\mu, t]) = r([\nu, t])$, and if these representatives
have been chosen with $t \not= 1$ as above, then the composition is defined as
\[
[\mu, t][\nu,t] = [\mu_1 \cdots \mu_{|\mu|-1}\nu_1 \cdots \nu_{|\nu|}, t].
\]
It then follows that for $\mu,\nu \in E^*$ with $s(\mu) = r(\nu)$, we have $[\mu][\nu] =
[\mu\nu]$. We also have $r([\mu,t][\nu,t]) = r([\mu,t])$ and $s([\mu,t][\nu,t]) =
s([\nu,t])$ whenever $[\mu,t][\nu,t]$ makes sense. Under this concatenation, $\SG{E}^*$
is a small category with objects $\SG{E}^0$.

Given $\alpha = [\mu, t] \in \SG{E}^*$ with $t \in [0,1)$, we can write $\mu = \mu_1
\cdots \mu_{n+1}$ with each $\mu_i \in E^1$. Defining $\alpha_i := [\mu_i\mu_{i+1}, t]
\in \SG{E}^1$ for $i \le n$, we obtain a factorisation $\alpha = \alpha_1 \cdots
\alpha_n$. This is the unique factorisation of $\alpha$ as a composition of edges of
$\SG{E}$.

As with directed graphs, given subsets $U,V \subseteq \SG{E}^*$, we write $UV :=
\{\alpha\beta : \alpha \in U, \beta \in V\text{ and } s(\alpha) = r(\beta)\}$. If $U$ is
a singleton $U = \{\alpha\}$, then we write $\alpha V$ and $V\alpha$ in place of
$\{\alpha\}V$ and $V\{\alpha\}$. In particular, for $\omega \in \SG{E}^0$ and a subset $U
\subseteq \SG{E}^*$ we have $\omega U = U \cap r^{-1}(\omega)$ and $U\omega = U \cap
s^{-1}(\omega)$.

Throughout the paper, we write $\SS$ for the circle $\RR/\ZZ$. Each element of $\SS$ has
a unique representative in $[0,1)$. We often abuse notation slightly and regard elements
of $[0,1]$ as elements of $\SS$ (so $1$ and $0$ are equal as elements of $\SS$).

\begin{lem}\label{lem:surjection}
Let $E$ be a locally finite graph with no sources. There is a continuous map $\varpi :
\SG{E}^* \to \SS$ such that $\varpi([\mu, t]) = t$ for $\mu \in E^* \setminus E^0$ and $t
\in [0,1)$.
\end{lem}
\begin{proof}
Lemma~\ref{lem:normal form} shows that $[\mu, t] \mapsto t$ is well-defined from
$\SG{E}^*$ to $\SS$. To see that it is continuous, let $q : \big((E^* \setminus E^0)
\times [0,1]) \to \SG{E}^0$ be the quotient map. For $0 < a < b < 1$, we have
$q^{-1}\big(\varpi^{-1}((a,b))\big) = (E^* \setminus E^0) \times (a,b)$, which is open.
So $\varpi^{-1}((a,b))$ is open. Likewise, for $\varepsilon < 1/2$, we have
$q^{-1}\big(\varpi^{-1}((-\varepsilon, \varepsilon))\big) = (E^* \setminus E^0) \times
([0,\varepsilon) \cup (1-\varepsilon, 1])$, which again is open. So
$\varpi^{-1}((-\varepsilon, \varepsilon))$ is open.
\end{proof}

Observe that the map $\varpi$ of Lemma~\ref{lem:surjection} satisfies $\varpi(\alpha) =
\varpi(s(\alpha)) = \varpi(r(\alpha))$ for all $\alpha \in \SG{E}^*$.

\begin{ntn}\label{ntn:path-space fibres}
For $t \in \SS$, we define
\[
    \SG{E}^*_t = \varpi^{-1}(t) = \{[\mu,t] : \mu \in E^*\},
\]
and for $n \in \NN$, we define
\[
    \SG{E}^n_t := \SG{E}^n \cap \SG{E}^*_t = \{[\mu, t] : \mu \in E^{n+1}\}.
\]
We then have $\SG{E}^n_t = (\SG{E}^n) (\SG{E}^0_t) = (\SG{E}^0_t) (\SG{E}^n)$ for all
$n,t$. With this notation, $\SG{E}^*_1 = \SG{E}^*_0 = \{[\mu] : \mu \in E^*\}$.
\end{ntn}

We aim to construct a $C^*$-algebra from $\SG{E}$. The following lemma shows that we
cannot employ Katsura's theory of topological-graph $C^*$-algebras, or Muhly and
Tomforde's theory of topological-quiver $C^*$-algebras: to get off the ground, both
theories require at least that the source map $s$ is an open map.

\begin{lem}\label{lem:cts structure}
Let $E$ be a finite graph with no sources. The maps $s, r : \SG{E}^1 \to \SG{E}^0$ are
continuous maps, and restrict to local homeomorphisms from $\SG{E}^1 \setminus E^1$ to
$\SG{E}^0 \setminus E^0$. For $e \in E^1$, the map $s$ is open at $[e] \in \SG{E}^1$ if
and only if $|E^1 s(e)| = 1$, and the map $r$ is open at $[e]$ if and only if $|r(E)E^1|
= 1$. In particular, $s$ is an open map if and only if $|E^1 v| = 1$ for all $v \in E^0$,
and $r$ is an open map if and only if $|vE^1| = 1$ for all $v \in E^0$.
\end{lem}
\begin{proof}
The range and source maps are continuous by construction. If $0 < t < 1$ and $ef \in
E^2$, then for any $\varepsilon$ such that $(t - \varepsilon, t+ \varepsilon) \subseteq
(0,1)$, the restrictions of $s, r$ to $\{[ef, s] : |t - s| < \varepsilon\}$ are
homeomorphisms onto open sets.

Fix $e \in E^1$. To see that $s$ is open at $[e]$ if and only if $|E^1 s(e)| = 1$, first
suppose that $|E^1 s(e)| = 1$, so $E^1 s(e) = \{e\}$. Then the sets
\[
U_\varepsilon := \{[fe, t] : f \in E^1 r(e), t > 1-\varepsilon\} \cup \{[ef, t] : f \in s(e)E^1, t < \varepsilon\}
\]
indexed by $\varepsilon \in (0,1/2)$ form a neighbourhood base at $[e]$ and we have
$s(U_\varepsilon) = \{[e, t] : t > 1-\varepsilon\} \cup \{[f,t] : f \in s(e)E^1, t <
\varepsilon\}$. Since $E^1 s(e) = \{e\}$, we deduce that, writing $v = s(e)$,
\[
s(U_\varepsilon) = \{[f, t] : f \in E^1 v, t > 1-\varepsilon\} \cup \{[f,t] : f \in vE^1, t <
\varepsilon\},
\]
which is a basic open neighbourhood of $[v]$. So $s$ is open at $[e]$.

Now suppose that $|E^1 s(e)| \ge 2$, say $f \in E^1 s(e) \setminus \{e\}$, and write $v
:= s(e)$. Consider the set
\[
U := \{[he,t] : h \in E^1 r(e), 1/2 < t \le 1\} \cup \{[eh,t] : h \in vE^1, 0\le t < 1/2\}.
\]
This set is open in the quotient topology. We have
\[
s(U) = \{[e,t] : 1/2 < t \le 1\} \cup \{[h,t] : h \in vE^1, 0\le t < 1/2\}.
\]
In particular, we have $[v] = [e,1] \in s(U)$, but the sequence $([f,
(n-1)/n])^\infty_{n=1}$ is contained in the complement of $s(U)$ and converges to $[v]$.
So $s$ is not open at $[e]$.

This completes the proof that $s$ is open if and only if each $|E^1v| = 1$. A symmetric
argument (or the same argument applied to the opposite graph) shows that $r$ is open at
$[e]$ if and only if $r(e) E^1$ is a singleton.

The first statement implies in particular that $r,s$ are open at $\alpha$ for every
$\alpha \in \SG{E}^1 \setminus E^1$, so the final statement follows immediately.
\end{proof}

\begin{exa}
Consider the simplest example of a finite graph with no sources: $E^0 = \{v\}$ and $E^1 =
\{e\}$, so $r(e) = s(e) = v$. Then the map $\varpi$ of Lemma~\ref{lem:surjection}
restricts to a homeomorphism $\varpi : [e,t] \mapsto t$ from $\SG{E}^0$ to $\SS$. For
each $w \in \SG{E}^0$ there is a unique $f_w \in \SG{E}^1$ with $r(f_w) = s(f_w) = w$,
and then $\SG{E}^1 = \{f_w : w \in \SG{E}^0\} \cong \SS$, and $s$ and $r$ are
homeomorphisms.
\end{exa}

\begin{exa}
Now consider the finite graph such that $E^0 = \{v\}$ and $E^1 = \{e,f\}$, so $r(e) =
s(e) = r(f) = s(f) = v$. So $\SG{E}^0$ is equal to the union $\{e,f\} \times \SS$ of two
circles glued at a point by gluing $(e,0)$ to $(f,0)$. For $g \in E^1$ and $t \in \SS$,
consider the vertex $[g,t] \in \SG{E}^0$. We have $\SG{E}^1 [g,t] = \{[eg,t], [fg,t]\}$,
and the ranges of these edges are $[e,t]$ and $[f,t]$ respectively. So as a set, we have
$\SG{E}^1 \cong \{e,f\} \times \SG{E}^0$. The sequences $([ee, 1-1/n])^\infty_{n=1}$ and
$([fe, 1-1/n])^\infty_{n=1}$ converge in $\SG{E}^1$ to $[ee,1] = [fe,1]$, and $s([ee,
1-1/n]) = [e, 1 - 1/n] = s([fe, 1-1/n])$ for all $n$, so $s$ is not a local homeomorphism
at $[ee, 1]$.
\end{exa}

\section{The infinite-path space of the suspension of a graph}\label{sec:path spaces of SGE}

In this section we describe the infinite-path space of the suspension of $E$, and we show
that if $E$ is finite, then $\SG{E}^\infty$ is homeomorphic to the one-sided suspension
flow of the shift space of the graph. In Section~\ref{sec:C*SGE} we will define the
Toeplitz algebra $\Tt C^*(\SG{E})$ and the Cuntz--Krieger algebra $C^*(\SG{E})$ of the
suspension of a graph $E$ using analogues of the path-space representation and Calkin
representation of a graph $C^*$-algebra. We will then link this to symbolic dynamics by
showing that $C^*(\SG{E})$ has a natural representation on $\ell^2(\SG{E}^\infty)$.

An \emph{infinite path} in $\SG{E}$ is a sequence $\alpha_1\alpha_2\alpha_3\cdots$ of
edges $\alpha_i \in \SG{E}^1$ such that $r(\alpha_{i+1}) = s(\alpha_i)$ for all $i$. We
write $\SG{E}^\infty$ for the set of all infinite paths in $\SG{E}$.

\begin{lem}\label{lem:identification}
Let $E$ be a locally finite graph with no sources. There is a bijection $\theta^\infty :
E^\infty \times [0,1) \to \SG{E}^\infty$ such that $\theta^\infty (e_1e_2e_3\cdots, t) =
[e_1e_2,t] [e_2e_3,t] [e_3e_4,t] \cdots$ for all $e_1 e_2 \cdots \in E^\infty$ and $t \in
[0,1)$.
\end{lem}
\begin{proof}
To see that $\theta^\infty$ is surjective, fix $\xi = \alpha_1 \alpha_2 \cdots$ in
$\SG{E}^\infty$. Write each $\alpha_i = (e_if_i, t_i)$ with $t_i \in [0,1)$ as in
Lemma~\ref{lem:normal form}. Then $t_i = t_1$ for all $i$. Let $t := t_1$. If $t \not=
0$, then $s(\alpha_i) = r(\alpha_{i+1})$ forces $[f_i, t] = [e_{i+1}, t]$, and hence
Lemma~\ref{lem:normal form} gives $f_i = e_{i+1}$, for all $i$. Thus $x = e_1e_2e_3\cdots
\in E^\infty$ and we have $\xi = \theta^\infty(x, t)$. If $t = 0$, then $s(\alpha_i) =
r(\alpha_{i+1})$ forces $s(e_i) = r(e_{i+1})$ for all $i$, and then since $t = 0$ we have
$\alpha_i = [e_ie_{i+1}, 0]$ for each $i$. So again, $x = e_1e_2\cdots$ belongs to
$E^\infty$ and $\xi = \theta^\infty(x,t)$. For injectivity, suppose that
$\theta^\infty(x, s) = \theta^\infty(y,t) = \alpha_1\alpha_2\cdots$. Put write $x =
e_1e_2\cdots$ and $y = f_1f_2\cdots$. Then $[e_1e_2, s] = \alpha_1 = [f_1f_2,t]$, and
since $s,t \not=1$, Lemma~\ref{lem:normal form} forces $s = t$. The definition of
$\approx$ shows that if $s,t \not= 1$ and $[ef,s] = [gh,t]$ then $e = g$. Since
$[e_ie_{i+1}, s] = \alpha_i = [f_i f_{i+1}, s]$ for all $i$, we deduce that $e_i = f_i$
for all $i$, and so $\theta^\infty$ is injective.
\end{proof}

We next describe a natural topology on $\SG{E}^\infty$.

\begin{lem}\label{lem:topology}
Let $E$ be a locally finite graph with no sources. For $\mu \in E^*$ and $0 < a < b < 1$,
let
\[
Z(\mu, (a,b)) := \{\theta^\infty(x,t) : x \in \mu E^\infty\text{ and }t \in (a,b)\},
\]
and for $\mu \in E^*$ and $0 < \varepsilon < \frac12$, let
\begin{align*}
Z(\mu, \varepsilon) := \{\theta^\infty(e\mu x, t) : {}&e \in E^1 r(\mu), x \in s(\mu) E^\infty,\text{ and } t \in (1-\varepsilon, 1)\} \\
    &{}\cup \{\theta^\infty(\mu x, t) : x \in s(\mu) E^\infty \text{ and } t \in [0,\varepsilon)\}
\end{align*}
Then there is a second-countable Hausdorff topology on $\SG{E}^\infty$ with basis
\begin{align*}
\mathcal{B} = \{Z(\mu,(a,b)) : {}&\mu \in E^*\text{ and }0 < a < b < 1\} \\
    &{}\cup \{Z(\mu, \varepsilon) : \mu \in E^*\text{ and } 0 < \varepsilon < \frac12\}.
\end{align*}
\end{lem}
\begin{proof}
To see that $\mathcal{B}$ is a basis, first observe that for $t \not= 0$ any element of
the form $\theta^\infty(x,t)$ belongs to $Z(r(x), (a,b))$ for any $0 < a < t < b < 1$;
and any element of the form $\theta^\infty(x, 0)$ belongs to $Z(x_1, \frac12)$ for any $e
\in E^1$ with $s(e) = r(x)$, so $\bigcup\Bb = \SG{E}^\infty$.

Given $\mu,\nu \in E^*$, we define
\[
\mu \vee \nu := \begin{cases}
    \mu &\text{ if $\mu = \nu\mu'$}\\
    \nu &\text{ if $\nu = \mu\nu'$}\\
    \infty &\text{ otherwise,}
\end{cases}
\]
where $\infty$ is used here purely as a formal symbol. As a notational convenience, we
define $Z(\infty, (a,b)) = \emptyset = Z(\infty,\varepsilon)$ for any $a, b,
\varepsilon$. We then have
\begin{align*}
    Z(\mu, (a,b)) \cap Z(\nu, (c,d)) &= Z(\mu \vee \nu, (\max\{a,c\}, \min\{b,d\})), \\
    Z(\mu, \varepsilon) \cap Z(\nu, \delta) &=  Z(\mu\vee\nu, \min\{\varepsilon, \delta\}),\quad\text{ and}\\
    Z(\mu, (a,b)) \cap Z(\nu,\varepsilon) &= \Big(\bigcup_{e \in E^1 r(\nu)} Z\big(\mu\vee e\nu, (a,b)\cap(1-\varepsilon, 1)\big)\Big) \\
                                                &\hskip9em{}\cup Z\big(\mu \vee \nu, (a,b) \cap [0, \varepsilon)\big).
\end{align*}
So $\mathcal{B}$ is a base for a topology on $\SG{E}^\infty$.

This topology is second countable because restricting the values of $a$, $b$, and
$\varepsilon$ to rational values in the definition of $\mathcal{B}$ yields a countable
base for the same topology.

To see that this topology is Hausdorff, fix distinct elements $\theta^\infty(x, s)$ and
$\theta^\infty(y,t)$ of $\SG{E}^\infty$. First suppose that $s = t$. Then $x \not= y$ so
we can find $\mu,\nu \in E^* \setminus E^0$ such that $x \in \mu E^\infty$, $y \in \nu
E^\infty$ and $\mu \vee \nu = \infty$. If $s \not= 0$ then for any $0 < a < s < b < t$,
the sets $Z(\mu, (a,b))$ and $Z(\nu, (a,b))$ are disjoint neighbourhoods of
$\theta^\infty(x, s)$ and $\theta^\infty(y, t)$. If $s = 0$, then that $\mu,\nu \not\in
E^0$ implies that $e\mu \vee f \nu = \infty$ for any $e \in E^1 r(\mu)$ and $f \in E^1
r(\nu)$. We already have $\mu \vee \nu = \infty$, so we deduce that $Z(\mu, \frac12)$ and
$Z(\nu, \frac12)$ are disjoint neighbourhoods of $\theta^\infty(x, s)$ and
$\theta^\infty(y, t)$. Now suppose that $s \not= t$; without loss of generality, $s \not=
0$. Fix $\mu$ with $x \in \mu E^\infty$ and $y \in \nu E^\infty$. Choose $0 < a < s < b <
1$ such that $t \not\in (a,b)$. If $t = 0$, then for any $\varepsilon < \min\{a, 1-b\}$
and any $f \in E^1$ with $s(f) = r(\nu)$, the sets $Z(\mu, (a,b))$ and $Z(f\nu,
\varepsilon)$ are disjoint neighbourhoods of $\theta^\infty(x, s)$ and $\theta^\infty(y,
t)$; and if $t \not= 0$ then for any $0 < c < t < d < 1$ such that $(a, b) \cap (c, d) =
\emptyset$, the sets $Z(\mu, (a,b))$ and $Z(\nu, (c,d))$ are disjoint neighbourhoods of
$\theta^\infty(x, s)$ and $\theta^\infty(y, t)$.
\end{proof}

Let $\sim_\sigma$ be the equivalence relation on $E^\infty \times [0,1]$ defined by $(x,
s) \sim_\sigma (y, t)$ if and only if either $x = y$ and $s = t$ or $y = \sigma(x)$, $s =
1$ and $t = 0$ or $x = \sigma(y)$, $s = 0 $ and $t = 1$; that is, the smallest
equivalence relation such that $(x, 1) \sim_\sigma (\sigma(x), 0)$ for all $x \in
E^\infty$. The suspension of $(E^\infty, \sigma)$ is the topological quotient space
\[
M(\sigma) := (E^\infty \times [0,1])/{\sim_\sigma}.
\]
For $t \in [0,1]$, we write
\[
M(\sigma)_t := \{[x,t] : x \in E^\infty\} \subseteq M(\sigma).
\]
So $M(\sigma)_1 = M(\sigma)_0 \cong E^\infty$.

\begin{rmk}
We can identify $M(\sigma)$ with the quotient space $(E^\infty \times
[0,\infty))/{\sim_\sigma}$ where $(x, s) \sim_\sigma (y, t)$ if and only if $s - \lfloor
s \rfloor = t - \lfloor t \rfloor$ and $\sigma^{\lfloor s\rfloor}(x) = \sigma^{\lfloor t
\rfloor}(y)$. The identification sends $[x,t] \in M_\sigma$ to the corresponding class
$\llbracket x, t\rrbracket$ in $(E^\infty \times \RR)/{\sim_\sigma}$; the inverse sends
$\llbracket x,t \rrbracket$ to $[\sigma^{\lfloor t\rfloor}(x), t - \lfloor t\rfloor]$.
\end{rmk}

\begin{prp}\label{prp:homeomorphism}
Let $E$ be a finite graph with no sources. The suspension $M(\sigma)$ is a compact
Hausdorff space, and the map $\theta^\infty : E^\infty \times [0,1) \to \SG{E}^\infty$
described in Lemma~\ref{lem:identification} induces a homeomorphism of $M(\sigma)$ onto
$\SG{E}^\infty$.
\end{prp}
\begin{proof}
The equivalence classes for $\sim_\sigma$ in $E^\infty \times [0,1]$ are finite: if $t
\not\in \{0,1\}$ then the equivalence class of $(x,t)$ is a singleton; if $t = 0$ then
the equivalence class of $(x,t)$ is $\{(x,0)\} \cup \{(ex, 1) : e \in E^1 r(x)\}$, and if
$t = 1$ then the equivalence class of $(x,t)$ is $\{(\sigma(x),0)\} \cup \{(e\sigma(x),
1) : e \in E^1 r(\sigma(x))\}$. In particular, the $\sim_\sigma$-equivalence classes in
$E^\infty \times [0,1]$ are discrete, and so the quotient topology on $M(\sigma)$ is
Hausdorff. Since $E^\infty \times [0,1]$ is compact, so is $M(\sigma)$.

The quotient map $q$ from $E^\infty \times [0,1]$ to $M(\sigma)$ restricts to a bijection
from $E^\infty \times [0,1)$ to $M(\sigma)$, and so $\theta^\infty$ induces a bijection
$\tau$ from $M(\sigma)$ to $\SG{E}^\infty$ such that $\tau(q(x, t)) = \theta^\infty(x,
t)$ for $x \in E^\infty$ and $t \in [0,1)$. Since $\SG{E}^\infty$ is Hausdorff and
$M(\sigma)$ is compact, to see that $\tau$ is a homeomorphism, it suffices to show that
it is continuous.

For this, observe that for $\mu \in E^*$ and $0 < a < b < 1$, we have
$q^{-1}\big(\tau^{-1}\big(Z(\mu, (a,b))\big)\big) = \mu E^\infty \times (a,b)$, and so
$\tau^{-1}\big(Z(\mu, (a,b))\big)$ is open by definition of the quotient topology.
Similarly, $q^{-1}\big(\tau^{-1}(Z(e\mu, \varepsilon))\big) = \big(e\mu E^\infty \times
(1 - \varepsilon, 1]\big) \cup \big(\mu E^\infty \times [0, \varepsilon)\big)$, which
again is open in $E^\infty \times [0,1]$; so again by definition of the quotient
topology, $\tau^{-1}(Z(e\mu, \varepsilon))$ is open, and hence $\tau$ is continuous.
\end{proof}

\section{The \texorpdfstring{$C^*$}{C*}-algebras of the suspension of a graph}\label{sec:C*SGE}

In this section we define two $C^*$-algebras associated to the suspension of a graph $E$.
We define the first of these algebras in terms of a concrete representation on a
non-separable Hilbert space, for which we use the following notational convention.

\begin{ntn}\label{ntn:ell2}
Throughout the rest of the paper, given any set $X$, we write
\[
\ell^2(X) := \{f : X \to \CC : f^{-1}(\CC \setminus \{0\})\text{ is countable, and }
    \sum_{x \in X} |v(x)|^2 < \infty\},
\]
which is a Hilbert space with inner product given by $\langle f, g\rangle = \sum_{x \in
X} v(x)\overline{w(x)}$. We denote the canonical basis elements of $\ell^2(X)$ by $\{h_x
: x \in X\}$; so $h_x(y) = \delta_{x,y}$ for $x,y \in X$.
\end{ntn}

\begin{rmk}\label{rmk:path spaces}
Using Notation~\ref{ntn:path-space fibres}, the set $\SG{E}^n_0 = \SG{E}^n_1$ is the
canonical copy $\{[\mu] : \mu \in E^n\}$ of $E^n$ in $\SG{E}^*$. In particular, there is
a unitary $U_0 : \ell^2(E^*) \to \ell^2(\SG{E}^*_0)$ satisfying
\[
U_0 h_\mu = h_{[\mu]}.
\]
For $t \in (0,1)$, and for each $n \in \NN$, there is a bijection $\SG{E}^n_t \to
\widehat{E}^n$ given by $[\mu,t] \mapsto (\mu_1\mu_2)(\mu_2\mu_3) \cdots
(\mu_n\mu_{n+1})$. In particular, for each $t \in (0,1)$ there is a unitary $U_t :
\ell^2(\widehat{E}^*) \to \ell^2(\SG{E}^*_t)$ given by
\[
U_t h_{(\mu_1\mu_2) (\mu_2\mu_3)\cdots(\mu_{|\mu|-1}\mu_{|\mu|})} = h_{[\mu,t]}.
\]
\end{rmk}

\begin{lem}\label{lem:rhoinfty,psiinfty}
Let $E$ be a locally finite graph with no sources. There is an injective nondegenerate
representation $\rho : C_0(\SG{E}^0) \to \Bb(\ell^2(\SG{E}^*)$ such that
\[
\rho(a)h_{\alpha} = a(r(\alpha)) h_{\alpha}\quad
    \text{ for all $a \in C_0(\SG{E}^0)$ and $\alpha \in \SG{E}^*$}.
\]
There is a linear map $\psi : C_c(\SG{E}^1) \to \Bb(\ell^2(\SG{E}^*))$ such that
\begin{equation}\label{eq:psiinfty def}
\psi(\xi)h_{\alpha} = \sum_{\beta \in \SG{E}^1 r(\alpha)} \xi(\beta) h_{\beta\alpha}\quad
    \text{ for all $\xi \in C_c(\SG{E}^1)$ and $\alpha \in \SG{E}^*$}.
\end{equation}
We have $\|\psi(\xi)\| \le \|\xi\|_\infty \cdot |\{ef \in E^2 : \xi([ef, t]) \not=
0\text{ for some }t \in [0,1]\}|$.
\end{lem}
\begin{proof}
The representation $\rho$ is the direct-sum of the representations $a \mapsto a(\omega)
\Id_{\ell^2(\omega\SG{E}^*)}$ indexed by $\omega \in \SG{E}^0$. It is nondegenerate
because for each $\alpha \in \SG{E}^*$ and any $a \in C_0(\SG{E}^0)$ with $a(r(\alpha)) =
1$ we have $\rho(a)h_{\alpha} = h_{\alpha}$. To see that it is injective, note that
$\|\rho(a)\| \ge \sup_{\omega \in \SG{E}^0} \|\rho(a)h_\omega\| = \|a\|_\infty$.

To see that there is a linear map $\psi$ as claimed, let $\pi$ be the path-space
representation of $\Tt C^*(E)$. For $t \in (0,1)$, the operator
\[
A_t := \sum_{ef \in \widehat{E}^1} \xi([ef, t]) \pi(t_e t_f t^*_f) \in \pi(\Tt C^*(E)) \subseteq \Bb(\ell^2(E^*))
\]
satisfies $\|A_t\| \le \sum_{ef \in \widehat{E}^1} |\xi([ef,t])| \le |\xi|_\infty \cdot
|\{ef \in E^2 : \xi([ef,t]) \not= 0\text{ for some }t\}|$ because the $\pi(t_e t_f
t^*_f)$ are partial isometries. Similarly, the operator
\[
A_0 := \sum_{e \in E^1} \xi([e]) \pi(t_e)
\]
satisfies $\|A_0\| \le \sum_{e \in E^1} |\xi([e])| \le |\xi|_\infty \cdot |\{ef \in E^2 :
\xi([ef,t]) \not= 0\text{ for some }t\}|$. Let $U_t : \ell^2(\SG{E}^*_t) \to
\ell^2(\widehat{E}^*)$, $0 < 1 < t$, and $U_0 : \ell^2(\SG{E}^*_0) \to \ell^2(E^*)$ be
the unitaries of Remark~\ref{rmk:path spaces}. Let $U := \bigoplus U_t : \ell^2(\SG{E}^*)
\to \bigoplus_{t \in [0,1)} \ell^2(E^\infty)$. Then
\[
\psi(\xi) := U^* \Big(\bigoplus_{t \in [0,1)} A_t\Big) U \in \Bb(\ell^2(\SG{E}^*))
\]
satisfies~\eqref{eq:psiinfty def}. The map $\psi_\infty$ thus defined is clearly linear,
and satisfies the desired norm estimate because the $A_t$ all do.
\end{proof}

We are now ready to define the Toeplitz algebra of $\SG{E}$.

\begin{dfn}
Let $E$ be a locally finite graph with no sources. We define $\Tt C^*(\SG{E})$ to be the
$C^*$-subalgebra of $\Bb(\ell^2(\SG{E}^*))$ generated by $\rho(C_0(\SG{E}^0))$ and
$\psi(C_c(\SG{E}^1))$.
\end{dfn}

To define $C^*(\SG{E})$ we first need to observe that the operators $\rho(a)$ and
$\psi(\xi)$ above respect the fibration of $\ell^2(\SG{E}^*)$ over $\SS$.

\begin{lem}\label{lem:invariant subspaces}
Let $E$ be a locally finite graph with no sources. For each $\omega \in \SG{E}^0$, the
subspace $\ell^2(\SG{E}^*\omega) \subseteq \ell^2(\SG{E}^*)$ is invariant for $\Tt
C^*(\SG{E})$. In particular, each $\ell^2(\SG{E}^*_t)$ is invariant for $\Tt
C^*(\SG{E})$. For each $t \in \SS$ and each $\omega \in \SG{E}^0_t$, we have $\Tt
C^*(\SG{E})|_{\ell^2(\SG{E}^* \omega)} \cap \Kk(\ell^2(\SG{E}^*_t)) =
\Kk(\ell^2(\SG{E}^*\omega))$.
\end{lem}
\begin{proof}
For $a \in C_0(\SG{E}^0)$, $\xi \in C_c(\SG{E}^1)$ and $\alpha \in \SG{E}^*$, the element
$\rho(a) h_\alpha$ is a scalar multiple of $h_\alpha$. We have $\psi(\xi)h_\alpha =
\sum_{\beta \in \SG{E}^1 r(\alpha)} \xi(\beta) h_{\beta\alpha}$, and a quick calculation
using inner-products shows that
\[
\psi(\xi)^* h_\alpha
    = \begin{cases}
        \xi(\alpha_1)h_{\alpha_2 \cdots \alpha_{|\alpha|}} &\text{ if $|\alpha| \ge 1$}\\
        0 &\text{ otherwise.}
    \end{cases}
\]
In particular, $\rho(a)h_\alpha$, $\psi(\xi) h_\alpha$ and $\psi(\xi)^*h_\alpha$ all
belong to $\ell^2(\SG{E}^* s(\alpha))$. Since the elements $\rho(a)$ and $\psi(\xi)$
generate $\Tt C^*(\SG{E})$, it follows that each $\ell^2(\SG{E}^*\omega)$ is invariant.
Consequently each $\ell^2(\SG{E}^*_t) = \bigoplus_{\omega \in \SG{E}^0_t}
\ell^2(\SG{E}^*_\omega)$ is invariant as well.

To prove the final statement, we first show that $\theta_{\omega, \omega} \in \Tt
C^*(\SG{E})|_{\ell^2(\SG{E}^*_t)}$ for each $t \in \SS$ and $\omega \in \SG{E}^0_t$. For
this, fix $\omega \in \SG{E}^0$. Fix $a \in C_0(\SG{E}^0)$ such that $a|_{\SG{E}^0_t} =
\delta_\omega$ (this is possible since $\SG{E}^0_t$ is a discrete subset of $\SG{E}^0$).
Likewise, the set $\SG{E}^1_t$ is discrete in $\SG{E}^1$, so for each $\alpha \in
\omega\SG{E}^1$ we can choose $\xi_\alpha \in C_c(\SG{E}^1)$ such that
$\xi_\alpha|_{\SG{E}^1_t} = \delta_\alpha$. Using the formulas described in the preceding
paragraph for the actions of $\rho(a)$ and the $\psi(\xi_\alpha)$ and their adjoints on
basis elements, we see that
\[
\theta_{h_\omega, h_\omega}
    = \Big(\rho(a) - \sum_{\alpha \in \omega\SG{E}^1} \psi(\xi_\alpha)\psi(\xi_\alpha)^*\Big)\Big|_{\ell^2(\SG{E}^*_t)}
    \in \Tt C^*(\SG{E})|_{\ell^2(\SG{E}^*_t)}.
\]

Now fix $t \in \SS$, $\omega \in \SG{E}^0_t$, and $\alpha \in \SG{E}^*\omega
\setminus\{\omega\}$. Factor $\alpha = \alpha_1 \cdots \alpha_m$ where each $\alpha_i \in
\SG{E}^1$. Using that $\SG{E}^1_t$ is discrete, we choose $\xi_1, \dots, \xi_m \in
C_c(\SG{E}^1)$ such that $\xi_i|_{\SG{E}^1_t} = \delta_{\alpha_i}$. Again calculating
with basis vectors, we see that $\theta_{h_\alpha, h_\omega} =
\psi(\xi_1)\cdots\psi(\xi_m)\theta_{h_\omega, h_\omega} \in \Tt
C^*(\SG{E})|_{\ell^2(\SG{E}^*_t)}$. Since each ${\ell^2(\SG{E}^*\omega)}$ is invariant
for $\Tt C^*(\SG{E})$, we deduce that $\theta_{h_\omega, h_\alpha} = \theta_{h_\omega,
h_\omega}(\psi(\xi_1)\cdots\psi(\xi_m))^* \in \Tt C^*(\SG{E})|_{\ell^2(\SG{E}^*_t)}$. Now
for arbitrary $\alpha,\beta \in \SG{E}^*\omega$ we have $\theta_{h_\alpha, h_\beta} =
\theta_{h_\alpha, h_\omega}\theta_{h_\omega, h_\beta} \in \Tt
C^*(\SG{E})|_{\ell^2(\SG{E}^*_t)}$.

It follows that $\Kk(\ell^2(\SG{E}^*_\omega)) \subseteq \Tt
C^*(\SG{E})|_{\ell^2(\SG{E}^*_t)} \cap \Kk(\ell^2(\SG{E}^*_t))$ for each $t \in \SS$ and
$\omega \in \SG{E}^0_t$. Since we also know that each $\ell^2(\SG{E}^*\omega) \subseteq
\ell^2(\SG{E}^*_t)$ is invariant for $\Tt C^*(\SG{E})$, we have the reverse containment
as well.
\end{proof}

Given a Hilbert space $\Hh$, we write $\Qq(\Hh)$ for the Calkin algebra
$\Bb(\Hh)/\Kk(\Hh)$.

\begin{dfn}
Let $E$ be a locally finite graph with no sources. We define $\tilde\rho : C_0(\SG{E}^0)
\to \bigoplus_{t \in \SS} \Qq(\ell^2(\SG{E}^*_t))$ by
\[
\tilde\rho(a) = \bigoplus_{t \in \SS} \rho(a)|_{\ell^2(\SG{E}^*)_t} + \Kk(\ell^2(\SG{E}^*_t)),
\]
and for $\xi \in C_c(\SG{E}^1)$, we define
\[
\tilde\psi(\xi) = \bigoplus_{t \in \SS} \psi(\xi)|_{\ell^2(\SG{E}^*)_t} + \Kk(\ell^2(\SG{E}^*_t)).
\]
We define $C^*(\SG{E})$ to be the $C^*$-subalgebra of $\bigoplus_{t \in \SS}
\Qq(\ell^2(\SG{E}^*_t))$ generated by $\tilde\rho(C_0(\SG{E}^0))$ and
$\tilde\psi(C_c(\SG{E}^1))$.
\end{dfn}

\begin{rmk}
For each $t \in \SS$ and each $\omega \in \SG{E}^0_t$, the subspace
$\ell^2(\SG{E}^*\omega) \subseteq \ell^2(\SG{E}^*_t)$ is invariant for $\Tt C^*(\SG{E})$.
It follows that there is an injective homomorphism from $C^*(\SG{E})$ to
$\bigoplus_{\omega \in \SG{E}^0} \Qq(\ell^2(\SG{E}^*\omega))$ that carries $\bigoplus_{t
\in \SS} \big(a|_{\ell^2(\SG{E}^*_t)} + \Kk(\ell^2(\SG{E}^*_t))\big)$ to
$\bigoplus_{\omega \in \SG{E}^0} \big(a|_{\ell^2(\SG{E}^*\omega)} +
\Kk(\ell^2(\SG{E}^*\omega))\big)$ for all $a \in \Tt C^*(\SG{E})$.
\end{rmk}

We link our definition of $C^*(\SG{E})$ to the suspension of one-sided shift of $E$ using
the infinite-path space of $\SG{E}$ described in the preceding section.

\begin{prp}
Let $E$ be a locally finite graph with no sources, and suppose that every cycle in $E$
has an entrance. Then there is a faithful representation $\Theta : C^*(\SG{E}) \to
\ell^2(M(\sigma))$ such that for $x \in E^\infty$ and $t \in [0,1)$
\[
\Theta(\tilde\rho(a))h_{[x,t]} = a([x_1, t]) h_{[x,t]}\text{ for all $a \in C_0(\SG{E}^0)$}
\]
and
\[
\Theta(\tilde\psi(\xi))h_{[x,t]} = \sum_{e \in E^1 r(x)} \xi([e,t]) h_{[ex, t]}\text{ for all $\xi \in C_c(\SG{E}^1)$.}
\]
\end{prp}
\begin{proof}
Since every cycle in $E$ has an entrance, the Cuntz--Krieger uniqueness theorem
\cite[Theorem~3.7]{KPR} shows that the infinite-path-space representations $\pi_\infty :
C^*(E) \to \Bb(\ell^2(E^\infty))$ and $\hat\pi_\infty : C^*(\widehat{E}) \to
\Bb(\ell^2(\widehat{E}^\infty))$ are both faithful. As discussed in
Section~\ref{sec:graph algs}, the Calkin representations $\pi_\Qq : C^*(E) \to
\Qq(\ell^2(E^*))$ and $\hat\pi_\Qq : C^*(\widehat{E}) \to \Qq(\ell^2(\widehat{E}^*))$ are
also faithful. Hence $\theta := \pi_\infty \circ \pi_\Qq^{-1} : \pi_\Qq(C^*(E)) \to
\pi_\infty(C^*(E))$ is an isomorphism that carries $Q_v + \Kk(\ell^2(E^*))$ to $P_v :=
\operatorname{proj}_{\ell^2(vE^\infty)}$ and carries $T_e + \Kk(\ell^2(E^*))$ to $S_e :
h_x \mapsto \delta_{s(e), r(x)} h_{ex}$. Similarly, $\hat{\theta} := \hat\pi_\infty \circ
\hat\pi_\Qq^{-1} : \hat\pi_\Qq(C^*(\widehat{E})) \to \hat\pi_\infty(C^*(\widehat{E}))$ is
an isomorphism carrying $Q_e + \Kk(\ell^2(\widehat{E}^*))$ to $P_e :=
\operatorname{proj}_{\ell^2(e\widehat{E}^\infty)}$ and carrying $T_{ef} +
\Kk(\ell^2(\widehat{E}^*))$ to $S_{ef} : h_x \mapsto \delta_{f, r(x)} h_{(ef)x}$. It
follows that $\theta \oplus (\bigoplus_{t \in \SS \setminus\{0\}} \hat\theta) :
C^*(\SG{E}) \to \ell^2(E^\infty) \oplus \big(\bigoplus_{t \in \SS \setminus \{0\}}
\ell^2(\widehat{E}^\infty)\big)$ is a faithful representation.

As in the proof of Lemma~\ref{lem:rhoinfty,psiinfty}, let $U_t : \ell^2(\SG{E}^*_t) \to
\ell^2(\widehat{E}^*)$, $0 < 1 < t$, and $U_0 : \ell^2(\SG{E}^*_0) \to \ell^2(E^*)$ be
the unitaries of Remark~\ref{rmk:path spaces}, and let $U := \bigoplus U_t :
\ell^2(\SG{E}^*) \to \bigoplus_{t \in [0,1)} \ell^2(E^\infty)$. Then $\Theta(x) := U^*
\big(\theta \oplus (\bigoplus_{t \in \SS \setminus\{0\}} \widehat \theta)\big)(x) U$
defines a faithful representation of $C^*(\SG{E})$ on $\bigoplus_{t \in \SS}
\ell^2(\SG{E}^\infty_t) = \ell^2(\SG{E}^\infty)$. Direct calculation using the
definitions of $\tilde\rho$, $\tilde\psi$, $\theta$ and $\hat\theta$ shows that $\Theta$
satisfies the prescribed formulae.
\end{proof}

\section{Fractional higher-power graphs and their \texorpdfstring{$C^*$}{C*}-algebras}\label{sec:SG[l]E}

In this section, given a graph $E$, we generalise the construction of
Section~\ref{sec:the graph torus} by constructing, for each real number $l$ a quiver
$\SG[l]{E}$ in such a way that $\SG[1]{E} = \SG{E}$ and $\SG[-1]{E} = \SG{E^{\op}}$. To
each $\SG[l]{E}$ we associate a Toeplitz algebra $\Tt C^*(\SG[l]{E})$ and a $C^*$-algebra
$C^*(\SG[l]{E})$ so that $\Tt C^*(\SG[1]{E}) \cong \Tt C^*(\SG{E})$ and $\Tt
C^*(\SG[-1]{E}) \cong \Tt C^*(\SG{E^{\op}})$, and similarly at the level of
Cuntz--Krieger algebras. We show that $\Tt C^*(\SG[0]{E}) \cong C_0(\SG{E}^0) \otimes
\Tt$ and $C^*(\SG[0]{E}) \cong C_0(\SG{E}^0) \otimes C(\TT)$.

We then show that if $l = \frac{m}{n}$ is rational, then there is a graph $F$, closely
related to $E$, such that $\Tt C^*(\SG[l]{E}) \cong \Tt C^*(\SG[|m|]{F})$, and this
isomorphism descends to an isomorphism $C^*(\SG[l]{E}) \cong C^*(\SG[|m|]{F})$. This
motivates the next section, in which we give a concrete description of each of $\Tt
C^*(\SG[m]{E})$ and $C^*(\SG[m]{E})$ for a large class of graphs $E$ and all positive
integers $m$. Combining this result with those of the current section yields an explicit
description of $\Tt C^*(\SG[l]{E})$ and $C^*(\SG[l]{E})$ for all rational $l$ and for a
large class of graphs $E$.

We use the following notation: given integers $0 \le m \le n \le L$, if $\mu = \mu_1
\mu_2 \dots \mu_L \in E^L$, then
\[
\mu(m,n) = \begin{cases}
    \mu_{m+1}\cdots \mu_n &\text{ if $n > m$}\\
    r(\mu_{m+1}) &\text{ if $n = m$.}
    \end{cases}
\]

It will be convenient in this section to use the following alternative description of
$\SG{E}^0$: it is homeomorphic to the quotient of $\bigsqcup_{n \ge 0} E^n \times [0, n]$
by the equivalence relation $(\mu,s) \sim (\nu,t)$ if $s - \lfloor s\rfloor = t - \lfloor
t \rfloor$ and $\mu(\lfloor s \rfloor, \lceil s \rceil) = \nu(\lfloor t \rfloor, \lceil t
\rceil)$ (see \cite{KKQS}).

\begin{dfn}
Let $E$ be a locally finite graph with no sources. Fix $l \in \RR$. Let $\approx_l$ be
the equivalence relation on $\{(\mu,t) : \mu \in E^*\text{ and }t, t+l \in [0, |\mu|]\}$
given by $(\mu,s) \approx_l (\nu,t)$ if and only if $s - \lfloor s \rfloor = t - \lfloor
t \rfloor$ and
\[
\begin{cases}
\mu(\lfloor s \rfloor, \lceil s + l\rceil) = \nu(\lfloor t \rfloor, \lceil t + l\rceil)
    &\text{ if $l \ge 0$,}\\
\mu(\lfloor s+l \rfloor, \lceil s\rceil) = \nu(\lfloor t+l \rfloor, \lceil t\rceil)
    &\text{ if $l \le 0$.}\\
\end{cases}
\]
We define
\[
\SG[l]{E}^1 := \{(\mu,t) : \mu \in E^*\text{ and }t, t+l \in [0, |\mu|]\}/{\approx_l}.
\]
We write $[\mu, t]_l$ for the equivalence class of $(\mu, t)$ under $\approx_l$. Define
$r_l, s_l : \SG{E}^l \to \SG{E}^0$ by $r_l([\mu,t]_l) := [\mu, t]_l$ and $s_l([\mu,t]_l)
= [\mu, t + l]_l$. The quadruple $(\SG{E}^0, \SG[l]{E}^1, r, s)$ is denoted $\SG[l]{E}$.
\end{dfn}

\begin{rmk}\label{rmk:special cases}
The special case $l = 1$ coincides with the space $\SG{E}^1$ and range and source maps
$r,s : \SG{E}^1 \to \SG{E}^0$ of Section~\ref{sec:the graph torus}. When $l = -1$ there
is a homeomorphism from the space $\SG[-1]{E}^1$ to the space $(\SG{E^{\op}})^1$
associated to the opposite graph of $E$ satisfying $[ef, t] \mapsto [f^{\op} e^{\op},
1-t]$, and this homeomorphism intertwines $r_{-1}$ with $r : (\SG{E^{\op}})^1 \to
(\SG{E^{\op}})^0$ and intertwines $s_{-1}$ with $s : (\SG{E^{\op}})^1 \to
(\SG{E^{\op}})^0$. When $l = 0$, we see that $\SG[0]{E}^1$ is a copy of the vertex space,
and $r_0, s_0$ are both just the identity map $\SG{E}^0 \to \SG{E}^0$.
\end{rmk}

\begin{ntn}
We will frequently just write $r,s$ in place of $r_l, s_l$, and $[\mu, t]$ in place of
$[\mu, t]_l$, when the parameter $l$ is clear from context.
\end{ntn}

As with $\SG{E}$, the quivers $\SG[l]{E}$ are not usually topological graphs, except if
$l = 0$.

\begin{lem}
Let $E$ be a locally finite graph with no sources. The range and source maps $s_0$ and
$r_0$ on $\SG[0]{E}$ are homeomorphisms. Fix $l \in \RR \setminus \{0\}$. Then $r_l$ and
$s_l$ are continuous maps. If $l > 0$, then the map $s_l$ is an open map if and only if
each $|E^1 v| = 1$ and $r_l$ is an open map if and only if each $|vE^1| = 1$. If $l < 0$,
then the map $s_l$ is open if and only if each $|vE^1| = 1$, and $r_l$ is open if and
only if each $|E^1 v| = 1$.
\end{lem}
\begin{proof}
The final statement of Remark~\ref{rmk:special cases} shows that $s_0$ and $r_0$ are
homeomorphisms. The proof of the remaining statements is very similar to that of
Lemma~\ref{lem:cts structure}.
\end{proof}

A \emph{path} in $\SG[l]{E}$ is a sequence $[\mu_1, t_1][\mu_2, t_2][\mu_3, t_3] \cdots
[\mu_k, t_k]$ of elements of $\SG[l]{E}^1$ such that $s_l([\mu_i, t_i]) = r_l([\mu_{i+1},
t_{i+1}])$ for all $1 \le i < k$. We write $\SG[l]{E}^*$ for the space
$\bigsqcup_{i=0}^\infty \SG[l]{E}^{i}$ of all paths in $\SG[l]{E}$, including the
vertices, which are regarded as paths of length 0. So $\SG[l]{E}^0 = \SG{E}^0$. We define
$r(v) = s(v) = v$ for each $v \in \SG{E}^0$, and for
\[
\alpha = [\mu_1, t_1][\mu_2, t_2][\mu_3, t_3] \cdots [\mu_k, t_k] \in
\SG[l]{E}^{k}
\]
we define $r(\alpha) = r([\mu_1, t_1])$ and $s(\alpha) = s([\mu_k, t_k])$.

\begin{rmk}\label{rmk:m path spaces}
For each $t \in \SS$ we write $\SG[l]{E}^n_t$ for $\{\alpha \in \SG[l]{E}^n :
\varpi(s(\alpha)) = t\} =  (\SG[l]{E}^n)(\SG{E}^0_t)$. We have $r(\SG[l]{E}^n_t) =
\SG{E}^0_{t-l}$, where addition on the subscript is modulo $\ZZ$. If $l = m \in \ZZ$
\emph{is} an integer, then as in Remark~\ref{rmk:path spaces}, the space $\SG[m]{E}^n_0 =
\SG[m]{E}^n_1$ is the canonical copy $\{[\mu] : \mu \in E^{nm}\}$ of $E^{nm}$ in
$\SG[m]{E}^*$. In particular there is a unitary $U_0 : \ell^2(E(0,m)^*) \to
\ell^2(\SG[m]{E}^*_0)$ satisfying
\[
U_0 h_\mu = h_{[\mu]}.
\]
For $m \ge 0$, $t \in (0,1)$, and each $n \in \NN$, there is a bijection $\SG[m]{E}^n_t
\to E(1,m+1)^n$ such that $[\mu,t] \mapsto (\mu_1\cdots\mu_{m+1})(\mu_m \cdots\mu_{2m+1})
\cdots (\mu_{(n-1)m} \cdots \mu_{nm+1})$. Thus for each $t \in (0,1)$ there is a unitary
$U_t : \ell^2(E(1, m+1)^*) \to \ell^2(\SG[m]{E}^*_t)$ given by
\[
U_t h_{(\mu_1\cdots\mu_{m+1})(\mu_m \cdots\mu_{2m+1}) \cdots (\mu_{(n-1)m} \cdots \mu_{nm+1})} = h_{[\mu,t]}
\]
\end{rmk}

The space $\SG[l]{E}^1$ is endowed with the quotient topology inherited from
$\SG{E}^{\lceil l \rceil + 1} \times [0,1]$.

\begin{lem}\label{lem:psil,rhol defs}
Let $E$ be a locally finite graph with no sources and fix $l \in \RR$. Let $\{h_{\alpha}
: \alpha \in \SG[l]{E}^{*}\}$ denote the canonical orthonormal basis for
$\ell^2(\SG[l]{E}^{*})$. There is an injective nondegenerate representation $\rho_l :
C_0(\SG{E}^0) \to \Bb(\ell^2(\SG[l]{E}^{*})$ such that
\[
\rho_l(a)h_{\alpha} = a(r(\alpha)) h_{\alpha}\quad
    \text{ for all $a \in C_0(\SG{E}^0)$ and $\alpha \in \SG[l]{E}^{*}$},
\]
and there is a linear map $\psi_l : C_c(\SG[l]{E}^1) \to \Bb(\ell^2(\SG[l]{E}^{*}))$ such
that
\[
\psi_l(\xi)h_{\alpha} = \sum_{\beta \in \SG[l]{E}^1 r(\alpha)} \xi(\beta) h_{\beta\alpha}\quad
    \text{ for all $\xi \in C_c(\SG[l]{E}^1)$ and $\alpha \in \SG[l]{E}^{*}$}.
\]
We have
\begin{equation}\label{eq:psil norm est}
 \|\psi_l(\xi)\| \le \|\xi\|_\infty \cdot |\{\mu \in E^{\lceil |l|\rceil + 1}
    : \xi([\mu, t]) \not= 0\text{ for some }t \in [0,1]\}|.
\end{equation}
For each $\omega \in \SG{E}^0$, the subspace $\ell^2(\SG[l]{E}^{*}\omega) \subseteq
\ell^2(\SG[l]{E}^{*})$ is invariant for the $C^*$-subalgebra of
$\Bb(\ell^2(\SG[l]{E}^{*}))$ generated by the images of $\rho_l$ and $\psi_l$.
\end{lem}
\begin{proof}
The first two statements follow from an argument almost identical to that of
Lemma~\ref{lem:rhoinfty,psiinfty}. The final statement follows from an argument nearly
identical to the proof of the first statement of Lemma~\ref{lem:invariant subspaces}.
\end{proof}

\begin{dfn}
Let $E$ be a locally finite graph with no sources.
\begin{enumerate}
\item We define $\Tt C^*(\SG[l]{E}) := C^*\big(\rho_l(C_0(\SG{E}^0)) \cup
    \psi_l(C_c(\SG[l]{E}^1))\big)$, the $C^*$-subalgebra of
    $\Bb(\ell^2(\SG[l]{E}^{*}))$ generated by the images of $\rho_l$ and $\psi_l$.
\item We define $\tilde\rho_l : C_0(\SG{E}^0) \to \oplus_{\omega \in \SG{E}^0}
    \Qq(\ell^2(\SG[l]{E}^{*}\omega))$ by
    \[
    \tilde\rho_l(a) = \bigoplus_{\omega \in \SG{E}^0} \Big(\rho(a)|_{\ell^2(\SG[l]{E}^{*}\omega)} + \Kk(\ell^2(\SG[l]{E}^{*}\omega))\Big)
    \]
    and we define $\tilde\psi_l : C_c(\SG[l]{E}^1) \to \oplus_{\omega \in \SG{E}^0}
    \Qq(\ell^2(\SG[l]{E}^{*}\omega))$ by
    \[
    \tilde\psi_l(\xi) = \bigoplus_{\omega \in \SG{E}^0} \Big(\psi(\xi)|_{\ell^2(\SG[l]{E}^{*}\omega)} + \Kk(\ell^2(\SG[l]{E}^{*}\omega))\Big).
    \]
\item We define $C^*(\SG[l]{E})$ to be the subalgebra of $\bigoplus_{\omega \in
    \SG{E}^0} \Qq(\ell^2(\SG[l]{E}^{*}\omega))$ generated by
    $\tilde\rho_l(C_0(\SG{E}^0))$ and $\tilde\psi_l(C_c(\SG[l]{E}^1))$.
\end{enumerate}
\end{dfn}

\begin{exa}\label{exa:rotation}
If $E$ consists of a single vertex $v$ and a single edge $e$, then each $\SG[l]{E}$ is a
copy of the topological graph $F_l := (\SS, \SS, t \mapsto t - l, \operatorname{id})$
determined by the rotation homeomorphism $t \mapsto t - l$ of $\SS$. We have
$C_c(\SG[l]{E}^1) = C(\SG[l]{E}^1) = C(F_l^1) = X(F_l)$, the topological-graph bimodule
of $F_l$. It is routine to verify that $(\rho_l, \psi_l)$ is a representation of $X(F_l)$
in the sense of Section~\ref{sec:topgraphs}. So there is a surjective homomorphism
$\psi_l \times \rho_l : \Tt C^*(F_l) \to \Tt C^*(\SG[l]{E})$ that carries
$i_{X(F_l)}(\xi)$ to $\psi_l(\xi)$ and carries $i_{C(\SS)}(a)$ to $\rho_l(a)$. The space
$\SG[l]{E}^{*}$ can be identified with $\SS \times \NN$ by the map that sends $\alpha \in
\SG[l]{E}^{n}$ to $(s(\alpha), n)$. Under this identification, the map $a \mapsto
\rho_l(a)(1 - \psi_l(1)\psi_l(1)^*)$ is the canonical faithful representation of $C(\SS)$
on $\ell^2(\SS \times \{1\}) \cong \ell^2(\SS)$, so the uniqueness theorem
\cite[Theorem~2.1]{FR} shows that $\psi_l \times \rho_l$ is injective. Therefore $\Tt
C^*(\SG[l]{E})$ is isomorphic to the Toeplitz algebra $\Tt C^*(F_l)$. Since $1 -
\psi_l(1)\psi_l(1)^*$ belongs to $\bigoplus_{t \in \SS} \Kk(\SG[l]{E}^{*} t)$, we see
that $\tilde{\psi}_l(1)\tilde{\psi}_l(1)^* = 1$ in $C^*(\SG[l]{E})$. For any $a \in
C(\SG[l]{E}^0)$ the left action of $a$ on $F_l$ is given by $a \cdot \xi = \theta_{a,
1}(\xi)$, and so we see that if $\phi : C(\SG[l]{E}^0) \to \Ll(F_l)$ denotes the
homomorphism implementing the left action, we have
\begin{align*}
\tilde{\rho}_l(a) - \tilde{\psi}_l^{(1)}(\phi(a))
    = \tilde{\rho}_l(a) - \tilde{\psi}_l(a)\tilde{\psi}_l(1)^*
    = \tilde{\rho}_l(a)(\tilde{\rho}_l(1) - \tilde{\psi}_l(1)\tilde{\psi}_l(1)^*) = 0.
\end{align*}
Hence $(\tilde\psi_l, \tilde\rho_l)$ is a covariant representation of $F_l$, and
therefore induces a homomorphism $\tilde\psi_l \times \tilde\rho_l : C^*(F_l) \to
C^*(\SG[l]{E})$. The action of $\TT$ on $\Bb(\ell^2(\SG[l]{E}^{*}))$ determined by
conjugation by the unitaries $\{W_z : z \in \TT\}$ given by $W_z(h_{t, n}) = z^n h_{t,n}$
induces an action on $\bigoplus_{t \in \SS} \Qq(\ell^2(\SG[l]{E}^{*} t))$, and it is
routine to check that $\tilde\psi_l \times \tilde\rho_l$ is equivariant for this action
and the gauge action on $C^*(F_l)$. Since $\tilde\rho_l : a \mapsto \bigoplus_\omega
a(\omega) 1_{\Qq(\ell^2(\SG[l]{E}^*\omega))}$ is injective, it follows from the
gauge-invariant uniqueness theorem \cite[Theorem~4.5]{Kat1} that $\tilde\psi \times
\tilde\rho$ is injective. So $C^*(\SG[l]{E}) \cong C^*(F_l)$. By
\cite[Proposition~10.5]{Kat2} there is an isomorphism of the rotation algebra $A_l$---the
universal $C^*$-algebra generated by unitaries $U, V$ such that $UV = e^{2\pi i l}
VU$---onto $C^*(F_l)$ that carries $U$ to $i_{C(\SS)}(t \mapsto e^{2\pi i t})$ and
carries $V$ to $i_{X(F_1)}(1)$. So we deduce that there is an isomorphism $A_l \cong
C^*(\SG[l]{E})$ that carries $U$ to $\tilde\rho_l(t \mapsto e^{2\pi i t})$ and carries
$V$ to $\tilde\psi_l(1_{\SG[l]{E}^1})$.
\end{exa}

\begin{rmk}\label{rmk:univalent}
More generally, if $E$ is a row-finite graph with no sources and satisfying $|E^1v| = 1$
for all $v$, then for each $l > 0$ the quadruple $F := (\SG{E}^0, \SG{E}^l, r, s)$ is a
topological graph in the sense of Katsura, and an analysis like that of
Example~\ref{exa:rotation} shows that $\Tt C^*(\SG[l]{E})$ and $C^*(\SG[l]{E})$ coincide
with the topological-graph $C^*$-algebras $\Tt C^*(F)$ and $C^*(F)$ respectively.
\end{rmk}

In the next few sections we will give a recipe for describing both $\Tt C^*(\SG[l]{E})$
and $C^*(\SG[l]{E})$ for rational values of $l$ provided that sufficiently many vertices
of $E$ both emit and receive at least two edges. We do not yet have a concrete
description of $C^*(\SG[l]{E})$ for arbitrary $l \in \RR$ and an arbitrary graph $E$.

The following theorem relates the constructions described in this section with those of
Section~\ref{sec:the graph torus} and Section~\ref{sec:C*SGE}. For the following result,
we denote by $\Tt$ the classical Toeplitz algebra generated by a non-unitary isometry
$S$.

\begin{thm}\label{thm:C* at 1,0,-1}
Let $E$ be a locally finite graph with no sources.
\begin{enumerate}
\item\label{it:C* at 1,0,-1 i} There is an isomorphism $\Tt C^*(\SG[1]{E}) \cong \Tt
    C^*(\SG{E})$ that carries $\rho_1(a)$ to $\rho(a)$ for $a \in C_0(\SG{E}^0)$ and
    carries $\psi_1(\xi)$ to $\psi(\xi)$ for $\xi \in C_c(\SG[1]{E}^1)$. This
    isomorphism descends to an isomorphism $C^*(\SG[1]{E}) \cong C^*(\SG{E})$.
\item\label{it:C* at 1,0,-1 ii} There is an isomorphism $\Tt C^*(\SG[0]{E}) \cong
    C_0(\SG{E}^0) \otimes \Tt$ that carries $\rho_0(a)$ to $a \otimes 1$ and carries
    $\psi_0(\xi)$ to $(\xi \circ r^{-1}_0) \otimes S$ for $\xi \in C_c(\SG[0]{E}^1) =
    C_c(\SG{E}^0)$, and this isomorphism descends to an isomorphism $C^*(\SG{E}^0)
    \cong C_0(\SG{E}^0) \otimes C(\TT)$.
\item\label{it:C* at 1,0,-1 iii} There is an isomorphism $\Tt C^*(\SG[-1]{E}) \cong
    \Tt C^*(\SG{E^{\op}})$ that carries $\rho_{-1}(a)$ to the element $\rho([e^{\op},
    t] \mapsto a([e, 1-t]))$ for $a \in C_0(\SG{E}^0)$ and carries $\psi_{-1}(\xi)$
    to $\psi\big([ef, t] \mapsto \xi([f^{\op}e^{\op}, (1-t)])\big)$ for $\xi \in
    C_c(\SG[-1]{E}^{1})$. This isomorphism descends to an isomorphism
    $C^*(\SG[-1]{E}) \cong C^*(\SG{E^{\op}})$.
\end{enumerate}
\end{thm}
\begin{proof}
The proofs of the first and third statements are almost identical. For the first
statement, observe that the identification $\SG[1]{E}^1 \cong \SG{E}^1$ of
Remark~\ref{rmk:special cases} intertwines $r,s$ with $r_1$ and $s_1$, and so induces a
homeomorphism $\SG[1]{E}^* \cong \SG{E}^*$, which induces a unitary $U_1 :
\ell^2(\SG[1]{E}^*) \cong \ell^2(\SG{E}^*)$. This unitary intertwines $\rho$ and $\rho_1$
and intertwines $\psi$ and $\psi_1$, and so $\Ad_{U_1}$ restricts to the desired
isomorphism $\Tt C^*(\SG[1]{E}) \cong \Tt C^*(\SG{E})$. Since $U_1$ carries each
$\ell^2(\SG[1]{E}^* \omega)$ to $\ell^2(\SG{E}^* \omega)$, the map $\Ad_{U_1}$ carries
each $\Kk(\ell^2(\SG[1]{E}^* \omega))$ to $\Kk(\ell^2(\SG{E}^* \omega))$, and so descends
to an isomorphism $C^*(\SG[1]{E}) \cong C^*(\SG{E})$ as claimed. For the third statement,
we argue exactly the same way, using the homeomorphism $\SG[-1]{E}^1 \cong
(\SG{E^{\op}})^1$ of Remark~\ref{rmk:special cases} to induce a unitary $U_{-1} :
\ell^2(\SG[-1]{E}^*) \cong \ell^2((\SG{E^{\op}})^*)$.

For the second statement, first identify $\ell^2(\SG[0]{E}^*)$ with $\ell^2(\SG{E}^0)
\otimes \ell^2(\NN)$ by the unitary $U_0$ that carries $h_{[\mu,t]}$ to $h_{[\mu_1, t]}
\otimes h_{|\mu|-1}$. Let $\pi : C_0(\SG{E}^0) \to \Bb(\ell^2(\SG{E}^0))$ be the
canonical faithful representation $\pi(a) h_\omega = a(\omega)h_\omega$. Direct
calculation shows that $U_0 \rho_0(a) U^*_0 = \pi(a) \otimes \id$, and $U_0 \psi_0(\xi)
U^*_0 = \pi(\xi \circ r_0^{-1}) \otimes S$ for $a \in C_0(\SG{E}^0)$ and $\xi \in
C_c(\SG[0]{E}^1)$. So $\Ad_{U_0}$ carries $\Tt C^*(\SG[0]{E})$ onto the subalgebra of
$\Bb(\ell^2(\SG{E}^0) \otimes \ell^2(\NN))$ generated by products of the form $(\pi(a)
\otimes \id)(\id \otimes S)$. This is precisely the tensor product $C_0(\SG{E}^0) \otimes
\Tt$. Moreover, $\Ad_{U_0}$ carries each $\Kk(\ell^2(\SG[0]{E}^* \omega))$ to $\Kk(\CC
h_{\omega} \otimes \ell^2(\NN))$, and so it carries the kernel of the quotient map $\Tt
C^*(\SG[0]{E}) \to C^*(\SG[0]{E})$ to $C_0(\SG{E}^0) \otimes \Kk(\ell^2(\NN)) \lhd
C_0(\SG{E}^0) \otimes \Tt$. It therefore descends to an isomorphism
\[
C^*(\SG[0]{E}) \cong (C_0(\SG{E}^0) \otimes \Tt) / (C_0(\SG{E}^0) \otimes \Kk)
     \cong C_0(\SG{E}^0) \otimes (\Tt/\Kk)
     \cong C_0(\SG{E}^0) \otimes C(\TT).\qedhere
\]
\end{proof}

The remainder of the section is devoted to reducing the study of the $C^*$-algebras $\Tt
C^*(\SG[l]{E})$ and $C^*(\SG[l]{E})$ for rational values of $l$ to the study of the
$C^*$-algebras $\Tt C^*(\SG[m]{F})$ and $C^*(\SG[m]{F})$ for nonnegative integers $m$ and
appropriate graphs $F$. We will analyse these latter in the next two sections.

Our first step is to show that we need only consider $l \ge 0$ by showing that $\Tt
C^*(\SG[l]{E}) \cong \Tt C^*(\SG[|l|]{E^{\op}})$ for $l \in (-\infty, 0)$.

\begin{lem}\label{lem:opposite}
Let $E$ be a locally finite graph with no sources. Fix $l \in (-\infty, 0)$. There is an
isomorphism $\Tt C^*(\SG[l]{E}) \cong \Tt C^*(\SG[|l|]{E^{\op}})$ that carries
$\rho_{l}(a)$ to $\rho_{|l|}\big([e^{\op}, t] \mapsto a([e, 1-t])\big)$ for $a \in
C_0(\SG{E}^0)$ and carries $\psi_{l}(\xi)$ to $\psi_{|l|}\big([\mu^{\op}, t] \mapsto
\xi([\mu, |\mu|-t])\big)$ for $\xi \in C_c(\SG[|l|]{E}^{1})$. This isomorphism descends
to an isomorphism $C^*(\SG[l]{E}) \cong C^*(\SG[|l|]{E^{\op}})$.
\end{lem}
\begin{proof}
This follows the argument of Theorem~\ref{thm:C* at 1,0,-1}(3): The map $[\mu, t] \mapsto
[\mu^{\op}, |\mu| - t]$ defines homeomorphism $\tau_l$ of $\SG[l]{E}^*$ onto
$(\SG[|l|]{E^{\op}})^*$ that intertwines the range and source maps. This homeomorphism
determines a unitary $U : \ell^2(\SG[l]{E}^*) \to \ell^2(\SG[|l|]{E^{\op}}$, conjugation
by which implements an isomorphism $\Tt C^*(\SG[l]{E}) \cong \Tt C^*(\SG[|l|]{E^{\op}})$
that satisfies the desired formulae. Since $U(\ell^2(\SG[l]{E}^* \omega)) =
\ell^2(\SG[|l|]{E^{\op}}\tau_l(\omega))$ for each $\omega \in \SG{E}^0$, we have
\[
U \Kk(\ell^2(\SG[l]{E}^* \omega)) U^* = \Kk(\ell^2(\SG[|l|]{E^{\op}}\tau_l(\omega)))
\]
for each $\omega$. Hence $\Ad_U$ descends to the desired isomorphism $C^*(\SG[l]{E})
\cong C^*(\SG[|l|]{E^{\op}})$.
\end{proof}

In the remainder of the section we must show that if $m \in \NN$ and $n \in \NN \setminus
\{0\}$, then there is a graph $F$ such that $\Tt C^*(\SG[\frac{m}{n}]{E}) \cong \Tt
C^*(\SG[m]{F})$ and similarly at the level of Cuntz--Krieger algebras.

Given a graph $E$ and an integer $n \ge 1$, the $n$\textsuperscript{th} \emph{delay} of
$E$ is the graph $D_n(E)$ described as follows. We set
\begin{align*}
D_n(E)^0 &:= E^0 \sqcup \{w_{e,j} : e \in E^1, 1 \le j \le n-1\}\qquad\text{ and}\\
D_n(E)^1 &:= \{f_{e, j} : e \in E^1, 1 \le j \le n\}.
\end{align*}
The range and source maps are given by
\[
r(f_{e,j}) = \begin{cases}
        w_{e, j-1} &\text{ if $j \ge 2$}\\
        r(e) &\text{ if $j = 1$}
    \end{cases}
    \qquad\text{ and }\qquad
s(f_{e,j}) = \begin{cases}
        w_{e, j} &\text{ if $j < n$}\\
        s(e) &\text{ if $j = n$.}
    \end{cases}
\]
In words, $D_n(E)$ is the graph obtained by inserting $n-1$ new vertices along each edge
of $E$. The example below pictures a graph $E$ on the left and the delayed graph $D_3(E)$
on the right.
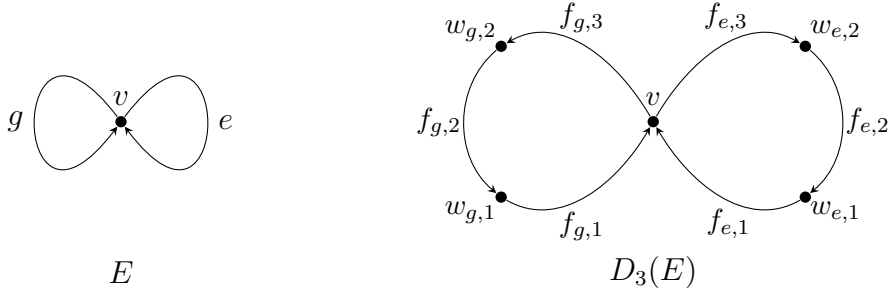
\[
\begin{tikzpicture}[>=stealth]
    \node[circle, fill=black, inner sep=1.5pt] (v0) at (0,0) {};
    \node[anchor=south] at (v0.north) {$v$};
    \draw[->] (v0) .. controls +(1.5, 2) and +(1.5, -2) .. (v0) node[pos=0.5, right] {$e$};
    \draw[->] (v0) .. controls +(-1.5, 2) and +(-1.5, -2) .. (v0) node[pos=0.5, left] {$g$};
    \node at (0, -2) {$E$};
\begin{scope}[xshift=7cm]
    \node[circle, fill=black, inner sep=1.5pt] (v) at (0,0) {};
    \node[anchor=south] at (v.north) {\small$v$};
    \node[circle, fill=black, inner sep=1.5pt] (e1) at (2, 1) {};
    \node[anchor=south west, inner sep=0pt] at (e1.north east) {\small$w_{e,2}$};
    \node[circle, fill=black, inner sep=1.5pt] (e2) at (2, -1) {};
    \node[anchor=north west, inner sep=0pt] at (e2.south east) {\small$w_{e,1}$};
    \node[circle, fill=black, inner sep=1.5pt] (f1) at (-2, 1) {};
    \node[anchor=south east, inner sep=0pt] at (f1.north west) {\small$w_{g,2}$};
    \node[circle, fill=black, inner sep=1.5pt] (f2) at (-2, -1) {};
    \node[anchor=north east, inner sep=0pt] at (f2.south west) {\small$w_{g,1}$};
    \draw[->, in=150, out=60] (v) to node[pos=0.7, anchor=south east, inner sep=1pt] {\small$f_{e, 3}$} (e1);
    \draw[->, in=40, out=320] (e1) to node[pos=0.5, right, inner sep=1pt] {\small$f_{e, 2}$} (e2);
    \draw[->, in=300, out=210] (e2) to node[pos=0.3, anchor=north east, inner sep=0pt] {\small$f_{e, 1}$} (v);
    \draw[->, in=30, out=120] (v) to node[pos=0.7, anchor=south west, inner sep=1pt] {\small$f_{g, 3}$}  (f1);
    \draw[->, in=140, out=220] (f1) to node[pos=0.5, left, inner sep=1pt] {\small$f_{g, 2}$}  (f2);
    \draw[->, in=240, out=330] (f2) to node[pos=0.3, anchor=north west, inner sep=1pt] {\small$f_{g, 1}$}  (v);
    \node at (0, -2) {$D_3(E)$};
\end{scope}
\end{tikzpicture}
\]

We will prove that for $m \ge 0$ and $n > 0$, the graph $\SG[\frac{m}{n}]{E}$ is
isomorphic to $\SG[m](D_n(E))$.

Observe that there is a range and source preserving map $D_n^* : E^* \to \bigcup_{k =
0}^\infty D_n(E)^{kn}$ given by
\[
D_n^*(e_1 \dots e_k) = f_{e_1,1}\dots f_{e_1,n}\,f_{e_2,1} \cdots f_{e_2,n} \cdots f_{e_k,1} \cdots f_{e_k,n}.
\]

\begin{lem}\label{lem:rational SG identification}
Let $E$ be a locally finite graph with no sources, and fix integers $n \ge 1$ and $m \ge
0$. There are homeomorphisms $\SG{D_n}^0 : \SG[\frac{m}{n}]{E}^0 \to \SG[m]{D_n(E)}^0$
and $\SG{D_n}^1 : \SG[\frac{m}{n}]{E}^1 \to \SG[m]{D_n(E)}^1$ such that for $j \in \{1,
\dots, n\}$ and $t \in [0,1]$,
\[
\SG{D_n}^0\Big(\Big[e, \frac{j - 1 + t}{n}\Big]\Big) = [f_{e, j}, t]
\]
for all $e \in E^1$, and
\[
\SG{D_n}^1\Big(\Big[\mu, \frac{j-1+t}{n}\Big]\Big) = [D_n^*(\mu), j-1+t]
\]
for all $\mu \in E^*$ such that $0 \le \frac{j-1+t}{n}, \frac{j-1+m+t}{n} \le |\mu|$. We
have $\SG{D_n}^0(s(\alpha)) = s(\SG{D_n}^1(\alpha))$ and $\SG{D_n}^0(r(\alpha)) =
r(\SG{D_n}^1(\alpha))$ for all $\alpha \in \SG[m]{D_n(E)}^1$.
\end{lem}
\begin{proof}
Define $\overline{\SG{D_n}}^0 : E^0 \sqcup (E^1 \times [0,1]) \to \SG{D_n(E)}^0 =
\SG[m]{D_n(E)}^0$ by $\overline{\SG{D_n}}^0(v) := [v]$ for $v \in E^0$, and
$\overline{\SG{D_n}}^0\Big(\Big(e, \frac{j - 1 + t}{n}\Big)\Big) = [f_{e, j}, t]$ for $e
\in E^1$, $1 \le j \le n$ and $t \in [0,1]$. Then
\[
\overline{\SG{D_n}}^0((e,0))
    = [f_{e, 0}, 0] = [r(f_{e,0})] = [r(e)] = \overline{\SG{D_n}}^0(r(e)),
\]
and similarly, $\overline{\SG{D_n}}^0((e,1)) = \overline{\SG{D_n}}^0(s(e))$, so
$\overline{\SG{D_n}}^0$ descends to a map $\SG{D_n}^0 : \SG[\frac{m}{n}]{E}^0 \to
\SG[m]{D_n(E)}^0$. Since $\overline{\SG{D_n}}^0$ is continuous, so is $\SG{D_n}^0$. It is
routine using Lemma~\ref{lem:normal form} to see that $\SG{D_n}^0$ is injective, and it
is clearly surjective. To see that it is open, fix $\omega \in \SG{E}^0$. If $\omega =
[e, \frac{j-1+t}{n}]$ with $t \in (0,1)$, then the sets $\{[e, \frac{j-1+t+s}{n}] : s \in
(-\delta, \delta)\}$ indexed by sufficiently small $\delta$ form a neighbourhood basis at
$\omega$ and are carried to the open sets $\{[f_{e, j}, t + s] : s \in (-\delta,
\delta)\}$. If $\omega = [e, \frac{j-1}{n}]$ for some $1 < j \le n$, then the sets $\{[e,
\frac{j-1+s}{n}] : s \in (-\delta, \delta)\}$ indexed by $\delta \in (0, \frac{1}{2})$
form a neighbourhood base at $\omega$ and are carried to the open sets $\{[f_{e, j-1}, t]
: t > 1-\delta\} \cup \{[f_{e, j}, t] : t < \delta\}$. Finally, if $\omega = [v]$ for $v
\in E^0$, then the sets $\{[e, s] : e \in E^1v : s > 1-\delta\} \cup \{[e,s] : e \in vE^1
: s < \delta\}$ indexed by $\delta < \frac{1}{n}$ are a neighbourhood base at $E$, and
are carried to the open sets $\{[f_{e, n} , t] : e \in E^1 v, t > 1-n\delta\} \cup
\{[f_{e,1}, t] : e \in vE^1, t < n\delta\}$. Therefore $\SG{D_n}^0$ is an open map, and
therefore a homeomorphism.

A similar argument shows that $\SG{D_n}^1$ exists and is a homeomorphism, and a simple
comparison of formulas shows that $\SG{D_n}^0$ and $\SG{D_n}^1$ are compatible with the
range and source maps as claimed.
\end{proof}

\begin{cor}\label{cor:rational C*-identification}
Let $E$ be a locally finite graph with no sources. Fix integers $n \ge 1$ and $m \ge 0$.
There is a unitary $U_{m,n} : \ell^2(\SG[\frac{m}{n}]{E}^*) \to \ell^2(\SG[m]{D_n(E)}^*)$
such that
\[
U_{m,n} \rho_{m/n}(a) U_{m,n}^* = \rho_m(a \circ \SG{D_n}^0)\quad\text{ for $a \in C_0(\SG[m]{D_n(E)}^0)$}
\]
and
\[
U_{m,n} \psi_{m/n}(\xi) U_{m,n}^* = \psi_m(\xi \circ \SG{D_n}^0)\quad\text{ for $\xi \in C_c(\SG[m]{D_n(E)}^1)$}.
\]
Conjugation by $U_{m,n}$ restricts to an isomorphism $\Theta_{m,n} : \Tt
C^*(\SG[\frac{m}{n}]{E}) \to \Tt C^*(\SG[m]{D_n(E)})$. This $\Theta_{m,n}$ induces an
isomorphism $\widetilde{\Theta}_{m,n} : C^*(\SG[\frac{m}{n}]{E}) \to C^*(\SG[m]{D_n(E)})$
such that
\[
\widetilde{\Theta}_{m,n}(\tilde\rho_{m/n}(a)) = \tilde\rho_m(a \circ \SG{D_n}^0)\text{ for $a \in C_0(\SG[m]{D_n(E)}^0)$}
\]
and
\[
\widetilde{\Theta}_{m,n}(\tilde\psi_{m/n}(\xi)) = \tilde\psi_m(\xi \circ \SG{D_n}^0)\text{ for $\xi \in C_c(\SG[m]{D_n(E)}^1)$}.
\]
\end{cor}
\begin{proof}
For $k \ge 1$ there is a bijection $\SG{D_n}^k : \SG[\frac{m}{n}]{E}^k \to
\SG[m]{D_n(E)}^k$ given by $\SG{D_n}^k(\alpha_1 \cdots \alpha_k) = \SG{D_n}^1(\alpha_1)
\cdots \SG{D_n}^1(\alpha_k)$. Combining these bijections $\SG{D_n}^k$ for $k > 0$ with
the bijection $\SG{D_n}^0$ we obtain a length-preserving bijection $\SG{D_n}^* :
\SG[\frac{m}{n}]{E}^* \to \SG[m]{D_n(E)}^*$, which induces a unitary
\[
U_{m,n} : \ell^2(\SG[\frac{m}{n}]{E}^*) \to \ell^2(\SG[m]{D_n(E)}^*).
\]
This $U_{m,n}$ intertwines $\rho_{m/n}$ and $a \mapsto \rho_{m}(a \circ \SG{D_n}^0)$ and
intertwines $\psi_{m,n}$ with $\xi \mapsto \psi_{m}(\xi \circ \SG{D_n}^1)$. This proves
the first statement. Since the unitary $U_{m,n}$ carries
$\ell^2(\SG[\frac{m}{n}]{E}^*\omega)$ to $\ell^2(\SG[m]{D_n(E)}^*\SG{D_n}^0(\omega))$ for
each $\omega \in \SG[\frac{m}{n}]{E}^0$, we have
\[
\Ad_{U_{m,n}}\Big(\bigoplus_{\omega \in \SG[\frac{m}{n}]{E}^0} \Kk(\ell^2(\SG[\frac{m}{n}]{E}^*\omega)\Big)
    = \bigoplus_{\omega \in \SG[m]{D_n(E)}^0} \Kk(\ell^2(\SG[m]{D_n(E)}^*\omega).
\]
Hence $\Theta_{m,n}$ carries the kernel of the quotient map $\Tt C^*(\SG[\frac{m}{n}]{E})
\to C^*(\SG[\frac{m}{n}]{E})$ to the kernel of the quotient map $\Tt C^*(\SG[m]{D_n(E)})
\to C^*(\SG[m]{D_n(E)})$. It follows that $\Theta_{m,n}$ descends to the desired
isomorphism $\widetilde{\Theta}_{m,n}$.
\end{proof}

\section{Analysis of \texorpdfstring{$\Tt C^*(\SG[m]{E})$}{TC*(S[m]E)}}\label{sec:TC*
analysis}

We now analyse the $C^*$-algebras $\Tt C^*(\SG[m]{E})$ and $C^*(\SG[m]{E})$ for integers
$m \ge 1$ (both $\Tt C^*(\SG[0]{E})$ and $C^*(\SG[0]{E})$ are described by part~(2) of
Theorem~\ref{thm:C* at 1,0,-1}). This will complete our analysis of the $C^*$-algebras
$\Tt C^*(\SG[l]{E})$ and $C^*(\SG[l]{E})$ for rational $l$.

To analyse $\Tt C^*(\SG[m]{E})$ we will first establish that the ideal of $\Tt
C^*(\SG[m]{E})$ generated by the image of $C_0(\SG{E}^0 \setminus E^0)$ is isomorphic to
$\Tt C^*(E(1, m+1)) \otimes C_0((0,1))$, and show that $\Tt C^*(\SG[m]{E})$ itself is a
$C(\SS)$-algebra. We begin with some preliminary structural results. Given a locally
compact Hausdorff space $X$, we write $C_b(X)$ for the algebra of bounded continuous
complex-valued functions on $X$.

\begin{lem}\label{lem:nondegenerate actions}
Let $E$ be a locally finite graph with no sources and fix $m \in \NN \setminus \{0\}$.
The space $C_c(\SG[m]{E}^1)$ is a $C_b(\SG{E}^0)$-bimodule with respect to the actions
$(a \cdot \xi)(\alpha) = a(r(\alpha)) \xi(\alpha)$ and $(\xi \cdot a)(\alpha) =
\xi(\alpha) a(s(\alpha))$. For each $\xi \in C_c(\SG[m]{E}^1)$ there exists $a \in
C_c(\SG{E}^0)$ such that $a \cdot \xi = \xi = \xi \cdot a$.
\end{lem}
\begin{proof}
For $a \in C_b(\SG{E}^0)$ and $\xi \in C_c(\SG[m]{E}^1)$, the function $a \cdot \xi$ is
the pointwise product of $a \circ r$ and $\xi$ and therefore a continuous function. Since
its support is contained in that of $\xi$ it belongs to $C_c(\SG[m]{E}^1)$. Similarly
$\xi \cdot a \in C_c(\SG[m]{E}^1)$. It is routine that these actions make
$C_c(\SG[m]{E}^1)$ into a $C_b(\SG{E}^0)$-bimodule. For the final assertion, fix $\xi \in
C_c(\SG[m]{E}^1)$. Since $r,s : \SG[m]{E}^1 \to \SG{E}^0$ are continuous, $K :=
r(\supp(\xi)) \cup s(\supp(\xi)) \subseteq \SG{E}^0$ is compact, so Tietze's theorem
yields a function $a \in C_c(\SG{E}^0)$ such that $a|_K \equiv 1$. We then have $a \cdot
\xi = \xi = \xi \cdot a$ by definition of the actions of $C_0(\SG{E}^0)$ on
$C_c(\SG[m]{E}^1)$.
\end{proof}

To analyse the ideal of $\Tt C^*(\SG[m]{E})$ generated by $C_0(\SG{E}^0 \setminus E^0)$,
we first observe that the subgraph of $\SG[m]{E}$ with vertex set $\SG{E}^0 \setminus
E^0$ is a topological graph in the sense of Katsura.

\begin{lem}\label{lem:subgraph}
Let $E$ be a locally finite graph with no sources and fix $m \in \NN \setminus \{0\}$.
Then $(\SG{E}^0 \setminus E^0, \SG[m]{E}^1 \setminus E^1, r, s)$ is a topological graph
isomorphic to the product $E(1, m+1) \times (0,1)$.
\end{lem}
\begin{proof}
Lemma~\ref{lem:normal form} shows that the quotient maps from $E^1 \times [0,1]$ to
$\SG{E}^0$ and from $E^{m+1} \times [0,1]$ to $\SG[m]{E}^1$ restrict range and source
preserving homeomorphisms from $E(1, m+1)^0 \times (0,1)$ to $\SG{E}^0 \setminus E^0$ and
from $E(1,m+1)^1 \times (0,1)$ to $\SG[m]{E}^1 \setminus E^1$.
\end{proof}

We can now describe the ideal of $\Tt C^*(\SG[m]{E})$ generated by $C_0(\SG{E}^0
\setminus E^0)$. In the following proof, given $e \in E^1$ and $g \in C_0((0,1))$, we
denote by $1_e \times g$ the element of $C_0(\SG{E}^0 \setminus E^0) \subseteq
C_0(\SG{E}^0)$ given by
\[
(1_e \times g)([f,t]) := \delta_{e,f} g(t)\quad\text{ for all $f \in E^1$ and $t \in (0,1)$},
\]
and likewise for $\mu \in E^{m+1}$ and $g \in C_0((0,1))$, we write $1_\mu \times g$ for
the element of $C_0(\SG[m]{E}^1 \setminus E^m) \subseteq C_0(\SG[m]{E}^1)$ given by
\[
(1_\mu \times g)([\nu,t]) := \delta_{\mu,\nu} g(t)\quad\text{ for all $\nu \in E^{m+1}$ and $t \in (0,1)$}.
\]

\begin{lem}\label{lem:quotient of tensor}
Let $E$ be a row-finite graph with no sources and fix $m \in \NN \setminus \{0\}$. Let
$J$ be the ideal of $\Tt C^*(\SG[m]{E})$ generated by $\rho_m(C_0(\SG{E}^0))$. There is
an isomorphism $\kappa_0 : C_0((0,1)) \otimes \Tt C^*(E(1, m+1)) \to J$ such that
\[
\kappa_0(g \otimes Q_e) = \pi(1_e \times g)
    \qquad\text{ and }\qquad
\kappa_0(g \otimes T_\mu) = \psi(1_{\mu} \times g)
\]
for all $g \in C_0((0,1))$, all $e \in E(1, m+1)^1 = E^1$ and all $\mu \in E(1,m+1)^1 =
E^{m+1}$.
\end{lem}
\begin{proof}
By Lemma~\ref{lem:subgraph}, we have $(\SG{E}^0 \setminus E^0, \SG[m]{E}^1 \setminus E^m,
r, s) \cong E(1, m+1) \times (0,1)$ as topological graphs. Recall from \cite{LPS} that an
$s$-section in a topological graph $F$ is an open set $U \subseteq F^1$ such that $s : U
\to s(U)$ is a homeomorphism onto an open subset of $F^0$. The collection
\[
    \Bb := \big\{\{[\nu, t] : t \in (a,b)\} : \nu \in E(1, m)^1\text{ and }0 < a < b < 1\big\}
\]
is a basis of open $s$-sections for the topology on $\SG[m]{E}^1 \setminus E^m$. Let
$\mathcal{F} := \{1_{\nu} \times g : \nu \in E(1,m+1)^1\text{ and }g \in C_0((0,1))\}$.
Then this $\Ff$ and $\Bb$ satisfy \cite[Equation~(4.4)]{LPS}. Direct computation with
basis elements on each $\ell^2(\SG[m]{E}_t)$ show that $(\rho_m, \psi_m|_{\lsp \Ff})$
satisfy \cite[Equation~(4.5)]{LPS} and \cite[Equation~(4.6)]{LPS}. Thus
\cite[Proposition~4.12]{LPS} shows that $(\rho_m|_{C_0(\SG{E}^0 \setminus E^0)},
\psi_m|_{C_c(\SG[m]{E}^1 \setminus E^1)})$ is a representation of the topological graph
$E(1, m+1) \times (0,1)$.

The range of $\rho_m|_{C_0(\SG{E}^0 \setminus E^0)}$ belongs to $J$ by definition. The
image of $\psi_m|_{C_c(\SG[m]{E}^1 \setminus E^m)}$ belongs to $J$ by the argument of
Lemma~\ref{lem:nondegenerate actions}. These elements generate $J$ because $C_0(\SG{E}^0
\setminus E^0) \cdot C_c(\SG[m]{E}^1)$ is contained in the space $C_d(\SG[m]{E}^1
\setminus E^m)$ described at~\eqref{eq:CdX}. Thus \cite[Theorem~2.4]{LPS} shows that
there is a surjective homomorphism $(\rho| \times \psi|) : \Tt C^*(E(1,m+1) \times (0,1))
\to J$ such that $(\rho| \times \psi|) \circ i_{E(1, m+1)^0 \times (0,1)} =
\rho_m|_{C_0(\SG{E}^0 \setminus E^0)}$ and $(\rho| \times \psi|) \circ i_{E(1, m+1)^1
\times (0,1)} = \psi_m|_{C_c(\SG[m]{E}^1 \setminus E^m)}$. To see that this homomorphism
is injective, recall that $J$ is a subalgebra of $\Bb(\ell^2(\SG[m]{E}^* \setminus
E(0,m)^*))$, and observe that
\[
\ell^2(\SG{E}^0 \setminus E^0) \subseteq \overline{\{\psi_m(\xi)v : \xi \in C_c(\SG[m]{E}^1 \setminus E^m), v \in \ell^2(\SG[m]{E}^*)\}}^\perp.
\]
For $a \in C_0(\SG{E}^0 \setminus E^0)$, the restriction of $\rho_m(a)$ to
$\ell^2(\SG{E}^0 \setminus E^0)$ is given by $\rho_m(a) h_\omega = a(\omega)h_\omega$,
and so the reduction of $\rho_m$ to this subspace is faithful. Hence
\cite[Theorem~2.1]{FR} shows that $(\rho| \times \psi|)$ is injective.

The argument of \cite[Proposition~7.7]{Kat2} shows that $\Tt C^*(E(1, m+1) \times (0,1))
\cong C_0((0,1)) \otimes \Tt C^*(E(1, m+1))$, so we obtain a surjective representation of
$C_0((0,1)) \otimes \Tt C^*(E(1, m+1))$ in $J$ that carries $g \otimes Q_e$ to $\pi(1_e
\times g)$ and $g \otimes T_{\nu}$ to $\psi(1_{\nu} \times g)$.
\end{proof}

We observe next that the actions of $C_b(\SG{E}^0)$ on $C_c(\SG[m]{E}^1)$ induce a
central action of $C(\SS)$ via the surjection $\SG{E}^0 \to \SS$ of
Lemma~\ref{lem:surjection}.

\begin{cor}\label{cor:C(S) action}
Let $E$ be a locally finite graph with no sources and fix $m \in \NN \setminus 0$. There
are left and right actions of $C(\SS)$ on $C_c(\SG[m]{E}^1)$ given by $(g \cdot
\xi)([\mu, t]) = g(t)\xi([\mu,t]) = (\xi \cdot g)([\mu,t])$ for $g \in C(\SS)$ and $\xi
\in C_c(\SG[m]{E}^1)$.
\end{cor}
\begin{proof}
The surjection $\varpi : \SG{E}^0 \to \SS$ of Lemma~\ref{lem:surjection} induces an
injection $\varpi^* : C(\SS) \to C_b(\SG{E}^0)$ given by $\varpi^*(g)([e,t]) =
g(\varpi([e,t])) = g(t)$. So $g \cdot \xi := \varpi^*(g) \cdot \xi$ and $\xi \cdot g :=
\xi \cdot \varpi^*(g)$ satisfy the formulae given for the desired action. The definition
of $\varpi$ shows that $g \cdot \xi = \xi \cdot g$.
\end{proof}

\begin{prp}\label{prp:bimodule maps}
Let $E$ be a locally finite graph with no sources and fix $m \in \NN \setminus 0$. The
pair $(\rho_m, \psi_m)$ is a bimodule homomorphism in the sense that
$\rho_m(a)\psi_m(\xi) = \psi_m(a \cdot \xi)$ and $\psi_m(\xi)\rho_m(a) = \psi_m(\xi \cdot
a)$ for all $a \in C_0(\SG{E}^0)$ and $\xi \in C_c(\SG[m]{E}^1)$. Writing $\bar\rho_m$
for the extension of $\rho_m$ to $C_b(\SG{E}^0) = \Mm(C_0(\SG{E}^0))$ and $\varpi_* :
C(\SS) \to C_b(\SG{E}^0)$ for the homomorphism induced by the map $\varpi : \SG{E}^0 \to
\SS$ of Lemma~\ref{lem:surjection}, and writing $\iota_m := \bar\rho_m \circ \varpi_*$,
the action of $C(\SS)$ on $C_c(\SG{E}^1)$ of Corollary~\ref{cor:C(S) action} satisfies
\[
\iota_m(g)\psi_m(\xi) = \psi_m(g \cdot \xi) = \psi_m(\xi)\iota_m(g)
    \quad\text{ for all $g \in C(\SS)$ and $\xi \in C_c(\SG[m]{E}^1)$.}
\]
In particular, $\iota_m$ is an injective unital inclusion of $C(\SS)$ in $\Zz\Mm(\Tt
C^*(\SG[m]{E}))$.
\end{prp}
\begin{proof}
We just calculate with basis vectors: for $a,b \in C_0(\SG{E}^0)$ and $\xi \in
C_c(\SG[m]{E}^1)$, and for $\alpha \in \SG[m]{E}^*$, we have
\begin{align*}
\rho_m(a) \psi_m(\xi) \rho_m(b) h_{\alpha}
    &= \sum_{\beta \in \SG[m]{E}^1 r(\alpha)} a(r(\beta)) \xi(\beta) b(r(\alpha)) h_{\beta\alpha}\\
    &= \sum_{\beta \in \SG[m]{E}^1 r(\alpha)} (a \cdot \xi \cdot b)(\beta) h_{\beta\alpha}
     = \psi_m(a \cdot \xi \cdot b) h_{\alpha}.
\end{align*}
Taking $a$ such that $a \cdot \xi = \xi$ as in Lemma~\ref{lem:nondegenerate actions}
gives $\psi_m(\xi) \rho_m(b) = \psi_m(\xi \cdot b)$. Likewise, taking $b$ such that $\xi
\cdot b = \xi$ gives $\rho_m(a) \psi_m(\xi) = \psi_m(a \cdot \xi)$.

Now fix $\xi \in C_c(\SG[m]{E}^1)$ and $g \in C(\SS)$. Choose $a \in C_c(\SG{E}^0)$ such
that $a \cdot \xi = \xi = \xi \cdot a$. Then by definition of $\bar\rho$,
\begin{align*}
\iota_m(g)\psi_m(\xi)
    &= \iota_m(g) \psi_m(a\cdot \xi)
    = \bar\rho_m(\varpi^*(g)) \rho_m(a) \psi_m(\xi)\\
    &= \rho_m(\varpi_*(g)a) \psi_m(\xi)
    = \psi_m((\varpi_*(g)a)\cdot \xi)
    = \psi_m(g\cdot (a\cdot \xi))
    = \psi_m(g \cdot \xi).
\end{align*}
Likewise $\psi_m(\xi)\iota_m(g) = \psi_m(\xi \cdot g)$. Corollary~\ref{cor:C(S) action}
shows that $g \cdot \xi = \xi \cdot g$, so we obtain $\iota_m(g)\psi_m(\xi) =
\psi_m(g\cdot \xi) = \psi_m(\xi)\iota_m(g)$ as claimed.

Since $\psi_m(C_c(\SG{E}^1)$ and $\rho_m(C_0(\SG{E}^0)$ generate $\Tt C^*(\SG[m]{E})$, we
deduce that
\[
    \iota_m(C(\SS))\Tt C^*(\SG[m]{E}) \subseteq \Tt C^*(\SG[m]{E}),
\]
and taking adjoints gives $\Tt C^*(\SG[m]{E}) \iota_m(C(\SS)) \subseteq \Tt
C^*(\SG[m]{E})$ as well. So we can regard $\iota_m$ as a homomorphism of $C(\SS)$ into
$\Mm(\Tt C^*(\SG[m]{E}))$. Since $C(\SG{E}^0)$ is abelian, the elements of
$\iota_m(C(\SS))$ commute with the elements $\rho_m(a)$, and we have just established
that they commute with the elements of $\psi_m(C_c(\SG[m]{E}^1)$. Again, since the
$\rho_m(a)$ and the $\psi_m(\xi)$ generate $\Tt C^*(\SG[m]{E})$ we see that $\iota$ takes
values in $\Zz \Mm(\Tt C^*(\SG[m]{E}))$. Finally, the identity function $1 \in C(\SS))$
satisfies $1 \cdot \xi = \xi = \xi \cdot 1$ for all $\xi \in C_c(\SG[m]{E}^1)$. So we
obtain $\iota_m(1)\psi_m(\xi) = \psi_m(1 \cdot \xi) = \psi_m(\xi)$ for all $\xi$, and
clearly $\iota_m(1)\rho_m(a) = \rho_m(a)$ for all $a$ by definition of $\iota_m$ and
$\varpi^*$. Hence $\iota_m$ is unital.
\end{proof}

The general theory of $C(X)$-algebras (see Section~\ref{sec:C(X)-algebras}) now implies
that $\Tt C^*(\SG[m]{E})$ is isomorphic to the algebra of continuous sections of an
upper-semicontinuous bundle of $C^*$-algebras over $\SS$.

\begin{ntn}
Let $E$ be a locally finite graph with no sources and fix $m \in \NN \setminus \{0\}$.
For each $t \in \SS$ we write $J_t$ for the ideal of $\Tt C^*(\SG[m]{E})$ generated by
$\iota(\{g \in C(\SS) : g(t) = 0\})$. Following the standard conventions for
$C(X)$-algebras, we then write $\Tt C^*(\SG[m]{E})_t$ for the quotient $\Tt
C^*(\SG[m]{E})/J_t$. For each $a \in \Tt C^*(\SG[m]{E})$, we write $\gamma_a : \SS \to
\bigsqcup_{t \in \SS} \Tt C^*(\SG[m]{E})_t$ for the section given by $\gamma_a(t) := a +
J_t$.
\end{ntn}

We first show that for $t \in (0,1)$, the fibre $\Tt C^*(\SG[m]{E})_t$ is a copy of $\Tt
C^*(E(1,m))$, and describe standard representatives in $\Tt C^*(\SG[m]{E})$ of its
canonical generators.

The following notation will be helpful for the next few results.

\begin{ntn}\label{ntn:pi_t}
Let $E$ be a locally finite graph with no sources. By Remark~\ref{rmk:m path spaces} the
space $\SG[m]{E}^*_0$ can be identified with $E(0, m)^*$ via the map $[\mu] \mapsto \mu$.
Let $\pi_0$ be the representation of $\Tt C^*(E(0,m))$ on $\ell^2(\SG[m]{E}^*_0)$
obtained from the path-space representation of $\Tt C^*(E(0,m))$ and this identification;
so
\[
    \pi_0(Q_v) h_{[\nu]} = \delta_{v, r(\nu)} h_{[\nu]}
        \quad\text{ and }\quad
    \pi_0(T_\mu) h_{[\nu]} = \delta_{s(\mu),r(\nu)} h_{[\mu\nu]}.
\]
For $t \in \SS\setminus\{0\}$, the set $\SG[m]{E}^*_t$ can be identified with
$E(1,m+1)^*$ via the map $\SG[m]{E}^k_t \owns [\mu, t] \mapsto (\mu_1 \cdots
\mu_{m+1})(\mu_{m+1}\cdots\mu_{2m+1}) \cdots (\mu_{(k-1)m}\mu_{km+1}) \in E(1, m+1)^k$
for $k \in \NN$, $\mu \in E^{km+1}$ and $t \in (0,1)$. We write $\pi_t$ for the
representation of $\Tt C^*(E(1, m+1))$ on $\ell^2(\SG[m]{E}^*_t)$ obtained from this
identification and the path-space representation of $\Tt C^*(E(1, m+1))$. So
\[
    \pi_t(Q_e) h_{[\mu,t]} = \delta_{e, \mu_1} h_{[\mu,t]}
        \quad\text{ and }\quad
    \pi_t(T_{\mu}) h_{[\nu,t]} = \delta_{\mu_{|\mu|}, \nu_1} h_{[\mu\nu_2 \cdots \nu_{|\nu|}, t]}.
\]
\end{ntn}

\begin{lem}\label{lem:representatives}
Let $E$ be a locally finite graph with no sources, fix $m \in \NN \setminus \{0\}$, and
take $t \in [0,1)$. If $a, a' \in C_0(\SG{E}^0)$ satisfy $a|_{\SG{E}^0_t} =
a'|_{\SG{E}^0_t}$, then $\pi_m(a) + J_t = \pi_m(a') + J_t$ in $\Tt C^*(\SG[m]{E})_t$. If
$\xi,\xi' \in C_c(\SG[m]{E}^1)$ satisfy $\xi|_{C_c(\SG[m]{E}^1_t)} =
\xi'|_{C_c(\SG[m]{E}^1_t)}$, then $\psi_m(\xi) + J_t = \psi_m(\xi') + J_t$ in $\Tt
C^*(\SG[m]{E})_t$.
\end{lem}
\begin{proof}
Let $d$ be the quotient metric on $\SS$ induced by the usual metric on $\RR$. For each
$n$, fix a function $f_n \in C_0(\SS \setminus \{t\})$ such that $0 \le f_n \le 1$ and
$f_n(s) = 1$ whenever $d(s,t) \ge 1/n$.

For the first statement, note that $C_0(\SG{E}^0 \setminus \{[e,t] : e \in E^1\})$
belongs to the ideal generated by the $\varpi_*(f_n)$, and so $\rho(C_0(\SG{E}^0
\setminus \SG{E}^0_t)) \subseteq J_t$. Since $a - a' \in C_0(\SG{E}^0 \setminus
\SG{E}^0_t)$ this proves the first statement.

For the second statement, let
\[
    N := |\{\mu \in E(1,m+1)^1 :  \max\{|\xi([\mu,s])|, |\xi'([\mu,s])|\} > 0 \text{ for some } s \in [0,1)\}|.
\]
Since $\xi$ and $\xi'$ have compact support, $N$ is finite. Fix $\varepsilon > 0$. The
set $X_\varepsilon := r(\{\alpha \in \SG[m]{E}^1 : |(\xi - \xi')(\alpha)| \ge
\varepsilon/N\})$ is a compact subset of $\SG{E}^0 \setminus \SG{E}^0_t$ and so there
exists $n > 0$ such that $f_n|_{X_\varepsilon} \equiv 1$. For this $n$, we have
\[
\big\|(\xi - \xi') - f_n \cdot (\xi - \xi')\|\infty \le \varepsilon/N,
\]
and $\supp\big((\xi - \xi') - f_n \cdot (\xi - \xi')\big) \subseteq \supp(\xi) \cup
\supp(\xi')$.

For any $\eta \in C_c(\SG[m]{E}^1)$, we have, using the representations $\pi_t$ of
Notation~\ref{ntn:pi_t},
\begin{align*}
\|\psi_m(\eta)\|
    &= \sup_{s \in \SS} \|\psi_m(\eta)|_{\ell^2(\SG[m]{E}^*_s)}\|\\
    &\le \max\Big\{\sup_{s \in (0,1)} \Big\|\sum_{\mu \in E^{m+1}} \eta([\mu,s])
        \pi_s(T_{\mu})\Big\|, \Big\|\sum_{\nu \in E^m} \eta([\nu])\pi_0(T_\nu)\Big\|\Big\}\\
    &\le \sup_{s \in \SS} \sum_{\alpha \in \SG[m]{E}^1_s} |\eta(\alpha)|.
\end{align*}
Applying this to $\eta = (\xi - \xi') - f_n \cdot (\xi - \xi')$ and using the definition
of $N$, we deduce that $\big\|\psi_m\big((\xi - \xi') - f_n \cdot (\xi - \xi')\big)\big\|
\le \varepsilon$. Since the $f_n$ all vanish at $t$, it follows that $\xi - \xi' \in J_t$
as claimed.
\end{proof}

We can now prove that for each $t \in (0,1)$, the corresponding fibre $\Tt
C^*(\SG[m]{E})_t$ is isomorphic to $\Tt C^*(E(1, m+1))$.

\begin{prp}\label{prp:nonzero fibres}
Let $E$ be a locally finite graph with no sources and fix $m \in \NN \setminus \{0\}$.
Take $t \in (0,1)$. For each $e \in E^1$, fix a function $a_{e,t} \in C_0(\SG{E}^0,
[0,1])$ such that $\supp(a_{e,t}) \subseteq \{[e,s] : 0 < s < 1\}$ and $a_{e,t}([e,t]) =
1$. For each $\mu \in E^{m+1}$, fix a function $\xi_{\mu,t} \in C_c(\SG[m]{E}^1)$ such
that $\supp(\xi_{\mu,t}) \subseteq \{[\mu,s] : 0 < s < 1\}$ and $\xi_{\mu,t}([\mu,t]) =
1$. There is an isomorphism $\theta_t : \Tt C^*(E(1, m+1)) \to \Tt C^*(\SG[m]{E})_t$ such
that
\[
\theta_t(Q_e) = \rho_m(a_{e,t}) + J_t\quad\text{ for all $e \in E(1,m+1)^0 = E^1$,}
\]
and such that
\[
\theta_t(T_{\mu}) = \psi_m(\xi_{\mu,t}) + J_t\quad\text{ for all $\mu \in E(1, m+1)^1 = E^{m+1}$.}
\]
\end{prp}
\begin{proof}
Lemma~\ref{lem:representatives} shows that the elements $\rho_m(a_{e,t}) + J_t$ and
$\psi_m(\xi_{ef,t}) + J_t$ generate $\Tt C^*(\SG[m]{E})_t$, so it suffices to construct
an injective homomorphism $\theta_t$ satisfying the given formulae.

For this, define $q_e := \rho_m(a_{e,t}) + J_t$ for each $e \in E^1$ and $t_{\mu} :=
\psi_m(\xi_{\mu,t}) + J_t$ for each $\mu \in E^{m+1}$. We will show that $(q, t)$ is a
Toeplitz--Cuntz--Krieger $E(1, m+1)$-family.

Since $a^2_{e,t}([f,t]) = \overline{a_{e,t}}([f,t]) = a_{e,t}([f,t])$ for all $f$,
Lemma~\ref{lem:representatives} shows that the $q_e$ are projections. We have $a_{e,t}
a_{f,t} = 0$ in $C_0(\SG{E}^0)$ for $e \not=f$, so the $q_e$ are mutually orthogonal.

Fix $\mu \in E^{m+1}$. Define $a'_{\mu, t} : \SG{E}^0 \to \CC$ by
\[
    a'_{\mu,t}([e,s]) = \begin{cases}
        |\xi_{\mu,t}(s)|^2 &\text{ if $\mu_{m+1} = e$}\\
        0 &\text{ otherwise.}
    \end{cases}
\]
Since $\supp(\xi_{\mu,t})$ is a compact subset of $\{[\mu, s] : 0 < s < 1\}$, we have
$a'_{\mu,t} \in C_c(\SG{E}^0)$, and by construction, $a'_{\mu,t}$ and $a_{\mu_{m+1},t}$
agree at $[g,t]$ for every $g \in E^1$. So Lemma~\ref{lem:representatives} implies that
$q_{\mu_{m+1}} = \rho_m(a'_{\mu,t}) + J_t$. Using the representations $\pi_t$ of
Notation~\ref{ntn:pi_t}, we have
\begin{align*}
\psi_m(\xi_{\mu,t})^* \psi(\xi_{\mu,t})
    &=  \Big(\sum_{\alpha,\beta \in E^m} \overline{\xi_{\mu,t}([\alpha])}\xi_{\mu,t}([\beta]) \pi_0(T_\alpha^*T_\beta)\Big)\\
    &\qquad\qquad \oplus
        \bigoplus_{0 < s < 1}
        \sum_{\eta, \zeta \in E^{m+1}} \overline{\xi_{\mu,t}}([\eta,s])\xi_{\mu,t}([\zeta,s]) \pi_s(T_\eta^*T_\zeta)\\
    &= 0 \oplus \bigoplus_{0 < s < 1} |\xi_{\mu,t}([\mu,s])|^2 \tilde\pi^\infty(Q_{\mu_{m+1}})
    = \rho_m(a'_{\mu, t}).
\end{align*}
Thus $t_\mu^* t_\mu = \rho(a'_{\mu,t}) + J_t = \rho(a_{\mu_{m+1},t}) + J_t  =
q_{s_{1,m+1}(\mu)}$.

Now fix $e \in E^1$. For each $\mu \in s(e) E^m$, define $\xi'_{e\mu, t}$ by
$\xi'_{e\mu,t}([\nu,s]) = \delta_{e\mu, \nu} \sqrt{a_{e, t}([e,s])}$ for $\nu \in
E^{m+1}$ and $s \in [0,1)$. Each $\xi'_{e\mu,t} \in C_c(\SG[m]{E}^1)$ because $a_{e, t}$
is supported on $\{[e,s] : 0 < s < 1\}$. Lemma~\ref{lem:representatives} shows that
$t_{e\mu} = \psi_m(\xi'_{e\mu,t})$ for each $\mu \in E(1, m+1)^1$. Arguing as above, we
see that
\begin{align}
\sum_{e\mu \in e E(1,m)^1} t_{e\mu} t^*_{e\mu}
    &= \sum_{\mu \in s(e)E^m} \psi(\xi'_{e\mu,t}) \psi(\xi'_{e\mu,t})^*\nonumber\\
    &= 0 \oplus \bigoplus_{0 < s < 1} \sum_{\mu \in s(e)E^m} |\xi'_{e\mu,t}(e\mu,s)|^2 \pi_s(T_{e\mu}T^*_{e\mu})\nonumber\\
    &= 0 \oplus \bigoplus_{0 < s < 1} a_{e,t}(e,s) 1_{\clsp\{h_{[\mu,s]} : \mu \in eE(1,m+1)^* \setminus \{e\}\}}.
    \label{eq:delta1}
\end{align}
Also,
\begin{equation}\label{eq:delta2}
\rho_m(a_{e,t}) = 0 \oplus \bigoplus_{0 < s < 1} a_{e,t}(e,s) 1_{\clsp\{h_{[\mu,s]} : \mu \in eE(1,m+1)^*\}}.
\end{equation}
We deduce that $\rho_m(a_{e,t}) > \sum_{e\mu \in eE(1,m+1)^1} \psi_m(\xi'_{e\mu,t})
\psi_m(\xi'_{e\mu,t})^*$. In particular, in the quotient, $q_e \ge \sum_{e\mu \in
eE(1,m+1)^1} t_{e\mu} t^*_{e\mu}$.

So $(q, t)$ is a Toeplitz--Cuntz--Krieger $E(1,m+1)$-family as claimed. The universal
property of $\Tt C^*(E(1, m+1))$ therefore yields a homomorphism $\theta_t : \Tt C^*(E(1,
m+1)) \to \Tt C^*(\SG[m]{E})_t$ such that $\theta_t(Q_e) = \rho_m(a_{e,t}) + J_t$ and
$\theta_t(T_{\mu}) = \psi_m(\xi_{\mu,t}) + J_t$.

It remains to prove that $\theta_t$ is injective. Since $J_t$ is contained in the kernel
of the restriction map $x \mapsto x|_{\ell^2(\SG[m]{E}^*_t)}$ on $\Tt C^*(\SG[m]{E})_t$,
we see that (with the functions $\xi'_{ef}$ used in the calculation~\eqref{eq:delta1}
above), each
\[
\|q_e - \sum_{e\mu \in eE(1,m+1)^1} t_{e\mu}t^*_{e\mu}\| \ge
    \Big\|\Big(\rho_m(a_{e,t}) - \sum_{e\mu \in eE(1,m+1)^1} \psi_m(\xi'_{e\mu,t})\psi_m(\xi'_{e\mu,t})\Big)\Big|_{\ell^2(\SG[m]{E}^*_t)}\Big\|.
\]
The calculations \eqref{eq:delta1}~and~\eqref{eq:delta2} therefore show that
\[
\Big\|q_e - \sum_{e\mu \in eE(1,m+1)^1} t_{e\mu}t^*_{e\mu}\Big\|
    \ge \|1_{\CC h_{e}}\| = 1.
\]
So each $q_e - \sum_{e\mu \in eE(1,m+1)^1} t_{e\mu}t^*_{e\mu} \not= 0$, and the
uniqueness theorem \cite[Theorem~4.1]{FR} shows that $\theta_t$ is injective.
\end{proof}

\begin{cor}\label{cor:res=fibre map}
Let $E$ be a locally finite graph with no sources and fix $m \in \NN \setminus \{0\}$.
Take $t \in (0,1)$. Let $U_t : \ell^2(E(1,m+1)^*) \to \ell^2(\SG[m]{E}^*_t)$ be the
unitary of Remark~\ref{rmk:m path spaces}. Let $\pi : \Tt C^*(E(1, m+1)) \to
\Bb(\ell^2(E(1, m+1)^*))$ be the path-space representation. Then the map $a + J_t \mapsto
\pi^{-1}(U_t^* a|_{\ell^2{\SG[m]{E}^*_t}} U_t)$ is an isomorphism of $\Tt
C^*(\SG[m]{E})_t$ onto $\Tt C^*(E(1, m+1))$.
\end{cor}
\begin{proof}
Consider the inverse $\theta_t^{-1}$ of the isomorphism described in
Proposition~\ref{prp:nonzero fibres}. It is straightforward to check that for $a \in
C_0(\SG{E}^0)$ and $\xi \in C_c(\SG[m]{E}^1)$, we have
\[
\theta_t^{-1}(\rho_m(a) + J_t) = \sum_{e \in E^1} a([e,t]) Q_e = \pi^{-1} (U_t^* \rho_m(a)|_{\ell^2(\SG[m]{E}^*_t)} U_t)
\]
and
\[
\theta_t^{-1}(\psi_m(\xi) + J_t) = \sum_{\mu \in E(1,m+1)^1} \xi([\mu,t]) T_{ef} = \pi^{-1} (U_t^* \psi_m(\xi)|_{\ell^2(\SG{E}^*_t)} U_t).
\]
Since the elements $\rho_m(a) + J_t$ and $\psi_m(\xi) + J_t$ generate $\Tt
C^*(\SG[m]{E})_t$, it follows that $x \mapsto \pi^{-1} (U_t^* x|_{\ell^2(\SG[m]{E}^*_t)}
U_t)$ agrees with $\theta_t^{-1}$, so is an isomorphism as claimed.
\end{proof}

We must now describe the fibre $\Tt C^*(\SG[m]{E})_0$. The idea is that for $a \in \Tt
C^*(\SG[m]{E}^0)$, using the unitaries $U_t : \ell^2(E(1, m+1)^*) \to
\ell^2(\SG[m]{E}^*_t)$ of Remark~\ref{rmk:m path spaces}, the function $t \mapsto U_t^*
a|_{\ell^2(\SG[m]{E}^*_t)} U_t$ from $(0,1)$ to $\Bb(\ell^2(E(1,m+1)^*))$ converges in
norm as $t \to 0$ and as $t \to 1$, and the limits $\varepsilon_0(a)$ and
$\varepsilon_1(a)$ belong to the image of $\Tt C^*(E(1, m+1))$ in its path-space
representation. We use these limits to construct an injective homomorphism of $\Tt
C^*(\SG[m]{E})_0$ into $\Tt C^*(E(1, m+1)) \oplus \Tt C^*(E(1, m+1))$.

\begin{lem}\label{lem:lim from right}
Let $E$ be a locally finite graph with no sources and fix $m \in \NN \setminus \{0\}$.
For $t \in (0,1)$, let $U_t : \ell^2(E(1,m+1)^*) \to \ell^2(\SG[m]{E}^*_t)$ be the
unitary of Remark~\ref{rmk:m path spaces}. Let $\pi : \Tt C^*(E(1, m+1)) \to
\Bb(\ell^2(E(1, m+1)^*))$ be the path-space representation. For $a \in C_0(\SG{E}^0)$ we
have
\[
\lim_{t \searrow 0} U_t^* \rho_m(a)|_{\ell^2(\SG[m]{E}^*_t)} U_t = \sum_{e \in E^1} a([r(e)]) \pi(Q_e).
\]
For $\xi \in C_c(\SG[m]{E}^1)$, we have
\[
\lim_{t \searrow 0} U_t^* \psi_m(\xi)|_{\ell^2(\SG[m]{E}^*_t)} U_t = \sum_{\mu \in E^m, e \in s(\mu)E^1} \xi{[\mu]} \pi(T_{\mu e}).
\]
For $a \in \Tt C^*(\SG[m]{E})$, the limit $\lim_{t \searrow 0} U_t^*
a|_{\ell^2(\SG[m]{E}^*_t)} U_t$ exists and belongs to $\pi(\Tt C^*(E(1, m+1)))$, and
$\varepsilon_0 : a \mapsto \pi^{-1}(\lim_{t \searrow 0} U_t^* a|_{\ell^2(\SG[m]{E}^*_t)}
U_t)$ is a homomorphism from $\Tt C^*(\SG[m]{E})$ to $\Tt C^*(E(1, m+1))$.
\end{lem}
\begin{proof}
Fix $a \in C_0(\SG{E}^0)$. Fix $\varepsilon > 0$, and let
\[
F := \{e \in E^1 : |a([e,t])| \ge \varepsilon/2\text{ for some } t \in [0,1]\}.
\]
Then $F$ is finite, and for each $e \in F$ there exists $\delta_e > 0$ such that $0 < t <
\delta_e \implies |a([e,t]) - a([e,0])| < \varepsilon$. Let $\delta = \min_{e \in F}
\delta_e$. Then for $e \in E^1$ and $0 < t < \delta$, if $e \in F$ then $|a([e,t]) -
a([e,0])| < \varepsilon$ by choice of $\delta$ and if $e \not\in F$, then $|a([e,t]) -
a([e,0])| \le |a([e,t])| + |a([e,0])| < \varepsilon$ by choice of $F$. By definition of
$\rho_m$, we have $U_t^* \rho_m(a)|_{\ell^2(\SG[m]{E}^*_t)}U_t = \sum_{e \in E^1} a([e,
t]) \pi(Q_e)$. Since $[e,0] = [r(e)]$ for each $e$ and since each $\|Q_e\| = 1$ it
follows that
\begin{equation}\label{eq:selim 1}
\lim_{t \searrow 0} U_t^*\rho_m(a)|_{\ell^2(\SG[m]{E}^*_t)} U_t = \sum_{e \in E^1} a([r(e)]) \pi(Q_e).
\end{equation}

Now fix $\xi \in C_c(\SG[m]{E}^1)$. Let $F = \{\mu \in E^{m+1} : \xi([\mu,t]) \not=
0\text{ for some } t \in [0,1]\}$. Fix $\varepsilon > 0$. For each $\mu \in F$ there
exists $\delta_{\mu} > 0$ such that $0 < t < \delta_{\mu}$ implies $|\xi([\mu,t]) -
\xi([\mu,0])| < \varepsilon/|F|$. Let $\delta := \min_{\mu \in F} \delta_{\mu}$. Fix $t
\in (0, \delta)$. We have
\begin{align*}
\Big\|U_t^* \psi_m(\xi)|&_{\ell^2(\SG[m]{E}^*_t)} U_t - \sum_{\mu \in E^m, e \in s(\mu)E^1} \xi{[\mu]} \pi(T_{\mu e})\Big\| \\
    &= \Big\|\sum_{\mu e \in F} (\xi{[\mu e, t]} - \xi([\mu])) \pi(T_{\mu e})\Big\|\\
    &\le \sum_{\mu e \in F} |\xi{[\mu e, t]} - \xi([\mu])| \|\pi(T_{\mu})\|
    < \varepsilon
\end{align*}
since each $\|\pi(T_{\mu})\| = 1$. Hence
\begin{equation}\label{eq:selim 2}
\lim_{t \searrow 0} U_t^* \psi_m(\xi)|_{\ell^2(\SG[m]{E}^*_t)} U_t = \sum_{\mu \in E^m, e \in s(\mu)E^1} \xi{[\mu]} \pi(T_{\mu e}).
\end{equation}

For the final statement, first consider a finite linear combination $x = \sum_i
\alpha_{i,1}\alpha_{i,2} \cdots \alpha_{i,k_i}$ where each $\alpha_{i,j} \in
\rho_m(C_0(\SG{E}^0)) \cup \psi_m(C_c(\SG[m]{E}^1)) \cup \psi_m(\SG[m]{E}^1)^*$. By the
first two statements, for each $i,j$ we have $\lim_{t \searrow 0} U_t^*
\alpha_{i,j}|_{\ell^2(\SG[m]{E}^*_t)} U_t = \pi(\beta_{i,j})$ for some $\beta_{i,j} \in
\Tt C^*(E(1, m+1))$, and it follows that $\lim_{t \searrow 0} U_t^*
x|_{\ell^2(\SG[m]{E}^*_t)} U_t = \pi\big(\sum_i \beta_{i,1} \dots \beta_{i, k_j}\big) \in
\pi(\Tt C^*(E(1, m+1)))$.

Now fix $x \in \Tt C^*(\SG[m]{E})$. Fix $\varepsilon > 0$. Fix a linear combination $a =
\sum_i \alpha_{i,1}\alpha_{i,2} \cdots \alpha_{i,k_i}$ where each $\alpha_{i,j} \in
\rho_m(C_0(\SG{E}^0)) \cup \psi_m(C_c(\SG[m]{E}^1)) \cup \psi_m(\SG[m]{E}^1)^*$ such that
$\|a - x\| < \varepsilon/4$. Then in particular, $\|(a - x)|_{\ell^2(\SG[m]{E}^*_t)}\| <
\varepsilon/4$ for all $t \in (0,1)$. By the preceding paragraph, $a_0 := \lim_{t
\searrow 0} U_t^*a|_{\ell^2(\SG[m]{E}^*_t)}U_t$ exists and belongs to $\Tt C^(E(1,
m+1))$, so there exists $\delta > 0$ such that $\|U_t^*a|_{\ell^2(\SG[m]{E}^*_t)}U_t -
a_0\| < \varepsilon/4$ for all $t < \delta$. In particular, there exists $\delta > 0$
such that $0 < s, t < \delta$ implies
\begin{align*}
\|U_s^*&x|_{\ell^2(\SG[m]{E}^*_s)}U_s - U_t^*x|_{\ell^2(\SG[m]{E}^*_t)}U_t\|\\
    &\le \|U_s^*(x|_{\ell^2(\SG[m]{E}^*_s)} - a|_{\ell^2(\SG[m]{E}^*_s)}U_s)\|
        + \|U_s^* a|_{\ell^2(\SG[m]{E}^*_s)} U_s - a_0\|\\
        &\qquad\qquad + \|a_0 - U_t^* a|_{\ell^2(\SG[m]{E}^*_t)} U_t\|
        + \|U_t^*(a|_{\ell^2(\SG[m]{E}^*_t)} - x|_{\ell^2(\SG[m]{E}^*_t)})U_t\|
     < \varepsilon.
\end{align*}
Hence $\big(U_{1/n}^* x|_{\ell^2(\SG[m]{E}^*_{1/n})}U_{1/n}\big)^\infty_{n=1}$ is a
Cauchy sequence, and therefore converges to some $x_0 \in \pi(\Tt C^*(E(1, m+1)))$.

We show that $\lim_{t \searrow 0} U_t^* x|_{\ell^2(\SG[m]{E}^*_t)} U_t = x_0$. Fix
$\varepsilon > 0$. Using that $U_{1/n}^* x|_{\ell^2(\SG[m]{E}^*_{1/n})} U_{1/n} \to x_0$
and also the preceding paragraph, we can choose $\delta > 0$ such that $\|U_{1/n}^*
x|_{\ell^2(\SG[m]{E}^*_{1/n})} U_{1/n} - x_0\| < \varepsilon/2$ whenever $n >
\delta^{-1}$, and such that $\|U_s^* x|_{\ell^2(\SG[m]{E}^*_s)} U_s - U_t^*
x|_{\ell^2(\SG[m]{E}^*_t)} U_t\| < \varepsilon/2$ whenever $s,t < \delta$. In particular,
for $t < \delta$, and any choice of $n > \delta^{-1}$, we have
\begin{align*}
\|U_t^* x&|_{\ell^2(\SG[m]{E}^*_t)} U_t - x_0\|\\
    &\le \|U_t^* x|_{\ell^2(\SG[m]{E}^*_t)} U_t - U_{1/n}^* x|_{\ell^2(\SG[m]{E}^*_{1/n})}U_{1/n}\|
        + \|U_{1/n}^* x|_{\ell^2(\SG[m]{E}^*_{1/n})} U_{1/n} - x_0\|
    < \varepsilon.
\end{align*}
Hence $U_t^* x|_{\ell^2(\SG[m]{E}^*_t)} U_t \to x_0 \in \pi(\Tt C^*(E(1, m+1)))$ as
claimed.

Since $\pi$ is injective, we deduce that the map $\varepsilon_0$ exists. It is a
homomorphism because each $a \mapsto a|_{\ell^2(\SG[m]{E}^*_t)}$ is a homomorphism and
the algebraic operations in $\Tt C^*(\SG[m]{E})$ are continuous.
\end{proof}

\begin{lem}\label{lem:lim from left}
Let $E$ be a locally finite graph with no sources and fix $m \in \NN \setminus \{0\}$.
For $t \in (0,1)$, let $U_t : \ell^2(E(1, m+1)^*) \to \ell^2(\SG[m]{E}^*_t)$ be the
unitary of Remark~\ref{rmk:m path spaces}. Let $\pi : \Tt C^*(E(1, m+1)) \to
\Bb(\ell^2(E(1, m+1)^*))$ be the path-space representation. For $a \in C_0(\SG{E}^0)$ we
have
\[
\lim_{t \nearrow 1} U_t^* \rho_m(a)|_{\ell^2(\SG[m]{E}^*_t)} U_t = \sum_{e \in E^1} a([s(e)]) \pi(Q_e).
\]
For $\xi \in C_c(\SG[m]{E}^1)$, we have
\[
\lim_{t \nearrow 1} U_t^* \psi_m(\xi)|_{\ell^2(\SG[m]{E}^*_t)} U_t = \sum_{\mu \in E^m, e \in E^1r(\mu)} \xi{[\mu]} \pi(T_{e\mu}).
\]
For $a \in \Tt C^*(\SG[m]{E})$, the limit $\lim_{t \nearrow 1} U_t^*
a|_{\ell^2(\SG[m]{E}^*_t)} U_t$ exists and belongs to $\pi(\Tt C^*(E(1, m+1)))$, and
$\varepsilon_1 : a \mapsto \pi^{-1}(\lim_{t \nearrow 1} U_t^* a|_{\ell^2(\SG[m]{E}^*_t)}
U_t)$ is a homomorphism from $\Tt C^*(\SG[m]{E})$ to $\Tt C^*(E(1, m+1))$.
\end{lem}
\begin{proof}
The proof is essentially identical to that of Lemma~\ref{lem:lim from right}.
\end{proof}

\begin{prp}\label{prp:zero evaluation}
Let $E$ be a locally finite graph with no sources and fix $m \in \NN \setminus\{0\}$.
There is an injective homomorphism $\eta : \Tt C^*(\SG[m]{E})_0 \to \Tt C^*(E(1, m+1))
\oplus \Tt C^*(E(1, m+1))$ such that, for any $a \in C_0(\SG{E}^0)$ we have
\[
\eta(\rho_m(a)_0) = \sum_{v \in E^0} a([v]) \Big(\Big(\sum_{e \in E^1 v} Q_e\Big) \oplus \Big(\sum_{f \in v E^1} Q_f\Big)\Big),
\]
and such that for any $\xi \in C_c(\SG[m]{E}^1)$ we have
\[
\eta(\psi_m(\xi)_0) = \sum_{\mu \in E^m} \xi([\mu]) \Big(\Big(\sum_{e \in E^1 r(\mu)} T_{e\mu}\Big)
                                            \oplus \Big(\sum_{f \in s(\mu)E^1} T_{\mu f}\Big)\Big).
\]
\end{prp}
\begin{proof}
Let $\varepsilon_0, \varepsilon_1 : \Tt C^*(\SG[m]{E}) \to \Tt C^*(E(1, m+1))$ be the
homomorphisms of Lemmas \ref{lem:lim from right}~and~\ref{lem:lim from left}. Since
$\varepsilon_0, \varepsilon_1$ vanish on $\rho_m(C_0(\SG{E}^0 \setminus E^0))$, they
descend to homomorphisms $\tilde\varepsilon_0, \tilde\varepsilon_1 : \Tt C^*(\SG[m]{E})_0
\to \Tt C^*(E(1, m+1))$. The homomorphism $\eta := \tilde\varepsilon_0 \oplus
\tilde\varepsilon_1$ satisfies the formulae above, so it suffices to show that this
homomorphism $\eta$ is injective.

For this, fix $x \in \Tt C^*(\SG[m]{E})$ such that $\eta(x_0) = 0$; so $\varepsilon_0(x)
= \varepsilon_1(x) = 0$. We must show that $x_0 = 0$. We have
\begin{equation}\label{eq:res is norm}
0 = \|\varepsilon_0(x)\| = \lim_{t \searrow 0} \|x|_{\ell^2(\SG[m]{E}^*_t)}\|,
    \qquad\text{ and }\qquad
0 = \|\varepsilon_1(x)\| = \lim_{t \nearrow 1} \|x|_{\ell^2(\SG[m]{E}^*_t)}\|.
\end{equation}
Corollary~\ref{cor:res=fibre map} implies that $\|x|_{\ell^2(\SG[m]{E}^*_t)}\| = \|x_t\|$
for $x \in \Tt C^*(\SG[m]{E})$ and $t \in (0,1)$, and so~\eqref{eq:res is norm} implies
that $\lim_{t \to 0} \|x_t\| = 0$. It now follows from the properties of upper
semicontinuous $C^*$-bundles---see equation~\ref{eq:norm limit zero}---that $x_t \to 0
\in C^*(\SG{E})_0$. Since $t \mapsto x_t$ is a continuous section, we deduce that $x_0 =
0$.
\end{proof}

We will show that the image of $\eta$ is isomorphic to $\Tt C^*(E(1, m+1)) \oplus \Tt
C^*(E(0,m))$ provided that enough vertices in $E$ admit at least two edges.

\begin{prp}\label{prp:0ev isomorphism}
Let $E$ be a locally finite graph with no sources and fix $m \in \NN \setminus \{0\}$.
Suppose that for every $v \in E^0$ there exist $n \ge 1$ and $\mu \in E^{nm} v$ such that
$|E^1 r(\mu)| \ge 2$. Let $\eta : \Tt C^*(\SG[m]{E})_0 \to \Tt C^*(E(1, m+1)) \oplus \Tt
C^*(E(1, m+1))$ be the homomorphism of Proposition~\ref{prp:zero evaluation}. Then the
range of $\eta$ is $\Tt C^*(E(1, m+1)) \oplus \jmath_{1, m+1}(\Tt C^*(E(0,m)))$, and
$(\operatorname{id} \oplus \jmath_{1,m+1}^{-1})\circ \eta$ is an isomorphism of $\Tt
C^*(\SG[m]{E})_0$ onto $\Tt C^*(E(1, m+1)) \oplus \Tt C^*(E(0,m))$.
\end{prp}
\begin{proof}
By Proposition~\ref{prp:zero evaluation}, we just need to show that the range of $\eta$
is $\Tt C^*(E(1,m+1)) \oplus \jmath(\Tt C^*(E(0,m)))$.

For each $v \in E^0$, let $w_v := \sum_{e \in E^1 v} Q_e \in \Tt C^*(E(1,m+1))$, and let
$x_v := \sum_{f \in v E^1} Q_f \in \Tt C^*(E(1,m+1))$. For each $\mu \in E^m$, let $y_\mu
:= \sum_{e \in E^1 r(\mu)} T_{e\mu} \in \Tt C^*(E(1,m+1))$ and let $z_\mu := \sum_{f \in
s(\mu)E^1} T_{\mu f} \in \Tt C^*(E(1,m+1))$.

We first show that the elements $w_v$ and $y_\mu$ generate $\Tt C^*(E(1,m+1))$. For this,
let
\[
A := C^*(\{w_v : v \in E^0\} \cup \{y_\mu : \mu \in E^m\}) \subseteq \Tt C^*(E(1,m+1)).
\]
We must show that $\Tt C^*(E(1,m+1)) \subseteq A$. Fix $\mu \in E^m$. For $e, f \in E^1
r(\mu)$, we have $T_{e\mu}^* T_{f\mu} = \delta_{e,f} Q_{\mu_m}$. Therefore
\begin{equation}\label{eq:ye*ye}
y_\mu^* y_\mu = \sum_{e,f \in E^1 r(\mu)} T_{e\mu}^* T_{f\mu} = |E^1 r(\mu)| Q_{\mu_m}.
\end{equation}
Since $E$ has no sinks, $|E^1 r(\mu)| \not=0$, so $Q_{\mu_m} \in A$. Again since $E$ has
no sinks, for each $e \in E^1$, the set $E^{m-1}r(e)$ is nonempty, so for any $e \in
E^1$, we have $Q_e = y^*_{\lambda e} y_{\lambda e}\in A$ for any $\lambda \in
E^{m-1}r(e)$. Since the $Q_e$ are mutually orthogonal projections, for $e \in E^1$ and
$\mu \in s(e) E^m$ we have
\[
T_{e\mu} = Q_e T_{e\mu} = Q_e \sum_{f \in E^1 r(\mu)} T_{f\mu} = Q_e y_\mu \in A.
\]
We have now established that all the generators of $\Tt C^*(E(1, m+1))$ belong to $A$,
and so $\Tt C^*(E(1, m+1)) \subseteq A$.

Next note that we have $x_v = \jmath_{1,m+1}(Q_v)$ for $v \in E^0$ and $z_\mu =
\jmath_{1, m+1}(T_\mu)$ for $\mu \in E^m$, so $(x, z)$ is a Toeplitz--Cuntz--Krieger
$E(0,m)$-family, and $C^*(\{x_v : v \in E^0\} \cup \{z_\mu : \mu \in E^m\}) =
\jmath_{1,m+1}(\Tt C^*(E(0,m)))$.

We show next that $(0, x_v) \in \eta(\Tt C^*(\SG{E})_0)$ for each $v \in E^0$. First, fix
$v \in E^0$ and $\mu \in E^m$ satisfying $|E^1 r(\mu)| \ge 2$. Equation~\eqref{eq:ye*ye}
and that $(x,z)$ is a Toeplitz--Cuntz--Krieger family show that
\begin{align*}
\eta(\Tt C^*(\SG{E})_0) &\owns ((y_{\mu}, z_{\mu})^*(y_{\mu}, z_{\mu}))^2 - (y_{\mu}, z_{\mu})^*(y_{\mu}, z_{\mu})\\
    &= (|E^1r(\mu)|^2 Q_{\mu_m},  x_{v}) - (|E^1r(\mu)| Q_{\mu_m}, x_{v})\\
    &= \big((|E^1r(\mu)|^2 - |E^1r(\mu)|) Q_{\mu_m}, 0\big).
\end{align*}
Since $|E^1 r(\mu) \ge 2|$, we have $|E^1r(\mu)|^2 - |E^1r(\mu)| > 0$, and so
$(Q_{\mu_m}, 0) \in \eta(\Tt C^*(\SG{E})_0)$. We then obtain
\[
\eta(\Tt C^*(\SG{E})_0) \owns (y_e, z_e)^*(y_e, z_e) - |E^1r(\mu)|(Q_{\mu_m}, 0) = (0, x_{v}).
\]

Now take any $v \in E^0$. By hypothesis, there exists $n \in \NN$ and $\mu \in
E^{nm}r(e)$ such that $|E^1 r(\mu)| \ge 2$. Write $\mu = \mu_1 \cdots \mu_{n}$ where each
$\mu_i \in E^m$. We have
\[
(0, x_{s(\mu)})
    = (y_{\mu_{n}}, z_{\mu_{n}})^* \cdots (y_{\mu_2}, z_{\mu_2})^* (0, x_{s(\mu_1)}) (y_{\mu_2}, z_{\mu_2}) \cdots (y_{\mu_n}, z_{\mu_n})
    \in \eta(\Tt C^*(\SG{E})_0).
\]
This shows that $(0, x_v) \in B$ for every $v \in E^0$ as claimed.

It follows that $(0, z_e) = (0, x_{r(e)})(y_e, z_e) \in \eta(\Tt C^*(\SG{E})_0)$ for each
$e \in E^1$. Since each $x_v = \jmath(Q_v)$ and each $z_e = \jmath(T_e)$, we deduce that
$0 \oplus \jmath(\Tt C^*(E(0,m))) \subseteq \eta(\Tt C^*(\SG{E})_0)$.

It now suffices to show that $\Tt C^*(E(1, m+1)) \oplus 0 \subseteq \eta(\Tt
C^*(\SG{E})_0)$ as well. Since we have already proved that $0 \oplus \jmath_{1,m}(\Tt
C^*(E(0,m))) \subseteq \eta(\Tt C^*(\SG{E})_0)$, we know that each $(w_v, 0) = (w_v, x_v)
- (0, x_v)$ and each $(y_\mu, 0) = (y_\mu, z_\mu) - (0, z_\mu)$ belongs to $\eta(\Tt
C^*(\SG{E})_0)$. We saw above that the elements $w_v$ and $y_\mu$ generate $\Tt
C^*(E(1,m+1))$. This completes the proof.
\end{proof}

If $x \in \Tt C^*(\SG[m]{E})$ belongs to the ideal generated by $C_0(\SS \setminus
\{0\})$, then $x|_{\ell^2(\SG[m]{E}^*_0)} = 0$. Consequently, there is a homomorphism
$\Tt C^*(\SG[m]{E})_0 \to \Bb(\ell^2(\SG[m]{E}^*_0))$ such that $x_0 \mapsto
x|_{\ell^2{\SG[m]{E}^*_0}}$ for $x \in \Tt C^*(\SG[m]{E})$. It therefore follows from
Proposition~\ref{prp:zero evaluation} that $\big\|x|_{\ell^2{\SG[m]{E}^*_0}}\big\| \le
\|\eta(x_0)\|$ for all $x \in \Tt C^*(\SG[m]{E})$. The following result gives a direct
proof of this by showing that in fact we can use the injection $\jmath_{1, m+1}$ to see
that the $x \mapsto x_{\ell^2(\SG[m]{E}^*_0)}$ can be identified with the map obtained by
following $\eta$ with the second-coordinate projection $\Tt C^*(E(1, m+1)) \oplus \Tt
C^*(E(1, m+1)) \to 0 \oplus \Tt C^*(E(1, m+1))$. We will also make use of this
identification in our analysis of $C^*(\SG[m]{E})$ in Section~\ref{sec:C* analysis}.

\begin{lem}\label{lem:res vs fibre map at zero}
Let $E$ be a locally finite graph with no sources. Let $U_0 : \ell^2(E(0,m)^*) \to
\ell^2(\SG[m]{E}^*_0)$ and, for $0 < t < 1$, $U_t : \ell^2(E(1,m+1)^*) \to
\ell^2(\SG[m]{E}^*_t)$ be the unitaries of Remark~\ref{rmk:m path spaces}. Let
$\pi_{E(0,m)} : \Tt C^*(E(0,m))) \to \Bb(\ell^2(E(0,m)^*))$ and $\pi_{E(1,m+1)} : \Tt
C^*(E(1,m+1))) \to \Bb(\ell^2(E(1,m+1)^*))$ be the path-space representations. For $x \in
\Tt C^*(\SG{E})$, we have
\[
\lim_{t \searrow 0} U_t^* x|_{\ell^2{(\SG[m]{E}^*_t)}} U_t = \pi_{E(1,m+1)}(\jmath_{1, m+1}(\pi_{E(0,m)}^{-1}(U_0^* x|_{\ell^2(\SG[m]{E}^*_0)}U_0))).
\]
\end{lem}
\begin{proof}
For $x \in \rho_m(C_0(\SG{E}^0))$, this follows from Equation~\ref{eq:selim 1} in the
proof of Lemma~\ref{lem:lim from right}, and for $x \in \psi_m(C_c(\SG[m]{E}^1))$, it
follows from Equation~\ref{eq:selim 2} in the same proof. Since $\Tt C^*(\SG[m]{E})$ is
generated by $\rho_m(C_0(\SG{E}^0)) \cup \psi_m(C_c(\SG[m]{E}^1))$, the result follows.
\end{proof}

We are now able to give an explicit description of $\Tt C^*(\SG[m]{E})$ provided that
enough vertices of $E$ emit at least two edges.

\begin{thm}\label{thm:TC*(SGE) isomorphism}
Let $E$ be a locally finite graph with no sources and fix $m \in \NN \setminus\{0\}$.
Suppose that for every $v \in E^0$ there exist $n \ge 1$ and $\mu \in E^{nm} v$ such that
$|E^1 r(\mu)| \ge 2$. Let $\jmath_{1,m+1} : \Tt C^*(E(0,m)) \hookrightarrow \Tt
C^*(E(1,m+1))$ be the injective homomorphism of Lemma~\ref{lem:mth Toeplitz injection}.
There is an isomorphism
\[
\kappa_m : \Tt C^*(\SG[m]{E}) \to \{f \in C([0,1], \Tt C^*(E(1, m+1))) : f(0) \in \jmath_{1,m+1}(\Tt C^*(E(0,m)))\}
\]
such that for $a \in C_0(\SG{E}^0)$ and $\xi \in C_c(\SG[m]{E}^1)$, we have
\[
\kappa_m(\rho_m(a))(t)
    = \begin{cases}
        \sum_{v \in E^0} a([v]) \jmath_{1,m+1}(Q_v) &\text{ if $t = 0$}\\
        \sum_{e \in E^1} a([e,t]) Q_e &\text{ if $t \in (0,1)$}\\
        \sum_{v \in E^0} a([v]) \sum_{e \in E^1 v} Q_e &\text{ if $t = 1$}
    \end{cases}
\]
and
\[
\kappa_m(\psi_m(\xi))(t)
    = \begin{cases}
        \sum_{\mu \in E^m} \xi([\mu]) \jmath(T_\mu) &\text{ if $t = 0$}\\
        \sum_{\nu \in E^{m+1}} \xi([\nu,t]) T_{\nu} &\text{ if $t \in [0,1)$}\\
        \sum_{\mu \in E^m} \xi([\mu]) \sum_{e \in E^1 r(\mu)} T_{e\mu} &\text{ if $t = 1$.}
    \end{cases}
\]
\end{thm}
\begin{proof}
For $t \in \SS$ let $q_t : \Tt C^*(\SG[m]{E}) \to \Tt C^*(\SG[m]{E})_t$ be the quotient
map $a \mapsto a + J_t$. For $t \in (0,1)$, let $\theta_t : \Tt C^*(E(1,m+1)) \to \Tt
C^*(\SG[m]{E})_t$ be the isomorphism of Proposition~\ref{prp:nonzero fibres}, and define
$\varepsilon_t := \theta^{-1}_t \circ q_t : \Tt C^*(\SG[m]{E}) \to \Tt C^*(E(1,m+1))$.
Let $\varepsilon_1 : \Tt C^*(\SG[m]{E}) \to \Tt C^*(E(1,m+1))$ be the map of
Lemma~\ref{lem:lim from left}, and let $\varepsilon_0 : \Tt C^*(\SG[m]{E}) \to \Tt
C^*(E(1,m+1))$ be the homomorphism of Lemma~\ref{lem:lim from right}.
Proposition~\ref{prp:0ev isomorphism} shows that $\varepsilon_1$ is surjective and that
the range of $\varepsilon_0$ is $\jmath_{1,m+1}(\Tt C^*(E(0,m)))$.

Fix $a \in C_0(\SG{E}^0)$ and $\xi \in C_c(\SG[m]{E}^1)$. Lemma~\ref{lem:quotient of
tensor} implies that for any $a \in \Tt C^*(\SG[m]{E})$ the function $t \mapsto
\varepsilon_t(a)$ is continuous at each $t \in (0,1)$; Lemmas \ref{lem:lim from
right}~and~\ref{lem:lim from left} show that it is continuous at $0$ and $1$ as well.
Hence there is a homomorphism $\kappa_m : \Tt C^*(\SG[m]{E}) \to C([0,1], \Tt
C^*(E(1,m+1)))$ given by $\kappa_m(a)(t) = \varepsilon_t(a)$ for all $a \in \Tt
C^*(\SG[m]{E})$ and $t \in [0,1]$.

To see that $\kappa$ is injective, suppose that $\kappa(a) = 0$. We must show that $a =
0$. Proposition~C.10(c) of \cite{Williams} shows that, $\|a\| = \sup_{t \in \SS}
\|q_t(a)\|$, so it suffices to show that each $q_t(a) = 0$. Since $\kappa(a) = 0$, we
have $\varepsilon_t(a) = 0$ for all $t$. Since $\theta_t$ is an isomorphism for $t \in
(0,1)$, we deduce that $q_t(a) = 0$ for $t \not= 0$, and Proposition~\ref{prp:0ev
isomorphism} shows that $\|q_0(a)\| = \max\{\|\varepsilon_0(a)\|, \|\varepsilon_1(a)\|\}
= 0$.

It remains to show that
\begin{equation}\label{eq:kappa image}
\kappa(\Tt C^*(\SG[m]{E})) = \{f \in C([0,1], \Tt C^*(E(1,m+1))) : f(0) \in \jmath(\Tt C^*(E(0,m)))\}.
\end{equation}
The containment $\subseteq$ follows from Proposition~\ref{prp:0ev isomorphism}. For the
reverse containment, fix an element $f$ of the right-hand side of~\eqref{eq:kappa image}.
Proposition~\ref{prp:0ev isomorphism} shows that there exists $a \in \Tt C^*(\SG[m]{E})$
such that $\eta_0(a) = f(0)$ and $\eta_1(a) = f(1)$. Hence $f - \kappa(a) \in C_0((0,1),
\Tt C^*(E(1, m+1)))$. Consequently Lemma~\ref{lem:quotient of tensor} shows that there
exists $b \in \Tt C^*(\SG[m]{E})$ such that $\kappa(b) = f - \kappa(a)$. Hence $f =
\kappa(a + b) \in \kappa(\Tt C^*(\SG[m]{E}))$.
\end{proof}

We deduce that under the hypotheses of the preceding theorem, $\Tt C^*(\SG[m]{E})$ is
homotopy equivalent to $\Tt C^*(E(0,m))$, and hence compute its $K$-theory.

For this, recall from \cite[Definition~3.2.5]{RLL} that $C^*$-homomorphisms $\varphi_0,
\varphi_1 : A \to B$ are \emph{homotopic} if there is a homomorphism $\varphi : A \to
C([0,1], B)$ such that $\varphi(a)(0) = \varphi_0(a)$ and $\varphi(a)(1) = \varphi_1(a)$
for all $a \in A$. Also recall that $C^*$-algebras $A$ and $B$ are \emph{homotopy
equivalent} if there are homomorphisms $\varphi : A \to B$ and $\psi : B \to A$ such that
$\varphi \circ \psi$ is homotopic to $\id_B$ and $\psi \circ \varphi$ is homotopic to
$\id_A$. We use the following elementary lemma.

\begin{lem}\label{lem:homotopy}
Let $A$ and $B$ be $C^*$-algebras, and let $\iota : B \to A$ be an injective
homomorphism. Then the $C^*$-algebra $C_\iota := \{f \in C([0,1], A) : f(0) \in
\iota(B)\}$ is homotopy equivalent to $B$.
\end{lem}
\begin{proof}
Define $\varphi : C_\iota \to B$ by $\psi(f) = \iota^{-1}(f(0))$ and define $\psi : B \to
C_\iota$ by $\psi(b)(t) = \iota(b)$ for all $t \in [0,1]$. Then $\varphi \circ \psi$ is
equal, and in particular homotopic, to $\id_B$. Define $\rho : C_\iota \to C([0,1],
C_\iota)$ by
\[
\big(\rho(f)(s)\big)(t)
    = \begin{cases}
        f(0) &\text{ if $t \le s$}\\
        f(t - s) &\text{ if $t > s$.}
    \end{cases}
\]
Then $\rho$ is a homomorphism, and $\rho(f)(0) = f = \id_{C_\iota}(f)$ and $\rho(f)(1) =
\psi(\varphi(f))$ for all $f \in C_\iota$. So $\psi \circ \varphi$ is homotopic to
$\id_{C_\iota}$.
\end{proof}

\begin{cor}
Let $E$ be a locally finite graph with no sources and fix $m \in \NN \setminus\{0\}$.
Suppose that for every $v \in E^0$ there exist $n \ge 1$ and $\mu \in E^{nm} v$ such that
$|E^1 r(\mu)| \ge 2$. Then $\Tt C^*(\SG[m]{E})$ is homotopy equivalent to $\Tt
C^*(E(0,m))$, and we have $K_0(\Tt C^*(\SG[m]{E})) \cong \ZZ E^0$, and $K_1(\Tt
C^*(\SG[m]{E})) = 0$.
\end{cor}
\begin{proof}
By Theorem~\ref{thm:TC*(SGE) isomorphism}, for the first statement we just have to show
that the algebra $A := \{f \in C([0,1], \Tt C^*(E(1,m+1))) : f(0) \in \jmath(\Tt
C^*(E(0,m)))\}$ is homotopy equivalent to $\Tt C^*(E(0,m))$. This follows from
Lemma~\ref{lem:homotopy} applied to $\iota = \jmath_{1, m+1}$.

Now \cite[Proposition~3.2.6]{RLL} shows that $K_0(A) \cong K_0(\Tt C^*(E(0,m)))$, and
\cite[Proposition~8.2.2(vi)]{RLL} shows that $K_1(A) \cong K_1(\Tt C^*(E(0,m)))$. By
\cite[Theorem~4.1]{FR}, the algebra $\Tt C^*(E(0,m))$ is isomorphic to the Toeplitz
algebra of a Hilbert bimodule over $C_0(E(0,m)^0) = C_0(E^0)$. Theorem~4.4 of
\cite{Pimsner} therefore implies that $\Tt C^*(E(0,m))$ is $KK$-equivalent to $C_0(E^0)$,
and hence $K_*(\Tt C^*(E(0,m))) \cong K_*(C_0(E^0)) \cong (\ZZ E^0, 0)$.
\end{proof}

\section{Analysis of \texorpdfstring{$C^*(\SG[m]{E})$}{C*(S[m]E)}}\label{sec:C*
analysis}

In this section we analyse the quotient $C^*(\SG[m]{E})$, using the analysis of $\Tt
C^*(\SG[m]{E})$ from the previous section. The quotient map $\Tt C^*(\SG[m]{E}) \to
C^*(\SG[m]{E})$ induces a homomorphism $\pi_t : \Tt C^*(\SG[m]{E})_t \to
C^*(\SG[m]{E})_t$ on each fibre. We show that under the isomorphisms $\Tt
C^*(\SG[m]{E})_t \cong \Tt C^*(E(0,m+1))$ (for $t \not= 0$) and $\Tt C^*(\SG[m]{E})_0
\cong \Tt C^*(E(1, m+1)) \oplus \Tt C^*(E(0,m))$ described in the preceding section,
these homomorphisms $\pi_t$ become the canonical quotient maps $\Tt C^*(E(1,m+1)) \to
C^*(E(1,m+1))$ (for $t \not= 0$), and $\Tt C^*(E(0,m)) \oplus \Tt C^*(E(1,m+1)) \to
C^*(E(0,m)) \oplus C^*(E(1,m+1))$ (for $t = 0$).

Recall that given any graph $E$, we denote by $I_E$ the ideal of $\Tt C^*(E)$ generated
by the projections $\Delta_v := Q_v - \sum_{e \in vE^1} T_e T^*_e$ indexed by $v \in
E^0$. The path-space representation $\pi_E : \Tt C^*(E) \to \Bb(\ell^2(E^*))$ restricts
to an isomorphism of $I_E$ onto $\oplus_{v \in E^0} \Kk(\ell^2(E^*v)) \subseteq
\Kk(\ell^2(E^*))$, and in particular satisfies $\pi(T_\mu \Delta_{s(\mu)} T^*_\nu) =
\theta_{\mu,\nu} \in \Kk(\ell^2(E^* s(\mu)))$ for all $\mu,\nu \in E^*$ with $s(\mu) =
s(\nu)$.

\begin{lem}\label{lem:image of K}
Let $E$ be a locally finite graph with no sources and fix $m \in \NN \setminus \{0\}$.
Let $K \lhd \Tt C^*(\SG[m]{E})$ be the ideal
\[
K := \{x \in \Tt C^*(\SG[m]{E}) : x|_{\ell^2(\SG[m]{E}^*_t)} \in \Kk(\ell^2(\SG[m]{E}^*_t))\text{ for all }t \in \SS\}.
\]
Let $\kappa_m : \Tt C^*(\SG[m]{E}) \to \{f \in C([0,1], \Tt C^*(E(1,m+1))) : f(0) \in
\jmath_{1,m+1}(\Tt C^*(E(0,m)))\}$ be the isomorphism of Theorem~\ref{thm:TC*(SGE)
isomorphism}. Then $\kappa_m(K) = \{f \in \kappa_m(\Tt C^*(\SG[m]{E})) : f(t) \in
I_{E(1,m+1)}\text{ for all }t\}$.
\end{lem}
\begin{proof}
Fix $x \in K$. Corollary~\ref{cor:res=fibre map} shows that for $t \not\in \{0,1\}$, we
have
\[
\kappa_m^{-1}(x)_t = \pi_{E(1,m+1)}^{-1}(U_t^* x|_{\ell^2(\SG[m]{E}^*_t)} U_t).
\]
So the discussion preceding this lemma shows that $\kappa_m(x)_t \in
\Kk(\ell^2(E(1,m+1)))$ for $t \not\in \{0,1\}$. Since $\Kk(\ell^2(E(1,m+1)))$ is closed,
we deduce from the definitions of $\varepsilon_0$ and $\varepsilon_1$ that
$\kappa_m^{-1}(x)_0$ and $\kappa_m^{-1}(x)_1$ belong to $\Kk(\ell^2(E(1,m+1)))$ as well.
So $\kappa_m(K) \subseteq \{f \in \kappa_m(\Tt C^*(\SG[m]{E})) : f(t) \in
I_{E(1,m+1)}\text{ for all }t\}$.

For the reverse inclusion, suppose that $f \in \kappa_m(\Tt C^*(\SG[m]{E}))$ and that
$f(t) \in I_{E(1,m+1)}$ for all $t$. Let $x := \kappa_m^{-1}(f)$.
Corollary~\ref{cor:res=fibre map} shows that for $t \not= 0$ we have
$x|_{\ell^2(\SG[m]{E}^*_t)} \in \Kk(\ell^2(\SG[m]{E}^*_t)$, and Lemma~\ref{lem:res vs
fibre map at zero} shows that
\[
\pi_{E(1,m+1)}(\jmath_{1, m+1}(\pi_{E(0,m)}^{-1}(U_0^* x|_{\ell^2(\SG[m]{E}^*_0)}U_0)))
    = \lim_{t \searrow 0} U_t^* x|_{\ell^2{(\SG[m]{E}^*_t)}} U_t.
\]
Since each $U_t^* x|_{\ell^2{(\SG[m]{E}^*_t)}} U_t \in \Kk(\ell^2(E(1,m+1)^*))$, we
deduce that
\begin{equation}\label{eq:jmath on K}
\pi_{E(1,m+1)}(\jmath_{1, m+1}(\pi_{E(0,m)}^{-1}(U_0^* x|_{\ell^2(\SG[m]{E}^*_0)}U_0))) \in \Kk(\ell^2(E(1,m+1)^*)).
\end{equation}
Lemma~\ref{lem:mth Toeplitz injection} shows that $\jmath_{1, m+1}^{-1}(I_{E(1,m+1)}) =
I_{E(0,m)}$. Since $\pi_{E(1,m+1)}$ carries $I_{E(1,m+1)}$ into $\Kk(\ell^2(E(1,m+1)^*))$
and $\pi_{E(0,m)}$ carries $I_{E(0,m)}$ to $\Kk(\ell^2(E(0,m)^*))$, we deduce that
\[
\pi_{E(0,m)}(\jmath_{1, m+1}^{-1}(\pi_{E(1,m+1)}^{-1}(\Kk(\ell^2(E(1,m+1)^*))) \subseteq \Kk(\ell^2(E(0,m)^*)).
\]
So~\eqref{eq:jmath on K} gives $U_0^* x|_{\ell^2(\SG[m]{E}^*_0)}U_0 \in
\Kk(\ell^2(E(0,m)^*))$ and hence $x|_{\ell^2(\SG[m]{E}^*_0)} \in
\Kk(\ell^2(\SG[m]{E}^*_0))$. So we have $x|_{\ell^2(\SG[m]{E}^*_t)} \in
\Kk(\ell^2(\SG[m]{E}^*_t))$ for all $t \in \SS$, and therefore $f = \kappa_m(x) \in
\kappa_m(K)$.
\end{proof}

\begin{thm}\label{thm:C*SE structure}
Let $E$ be a locally finite graph with no sources and fix $m \in \NN \setminus \{0\}$.
Suppose that for every $v \in E^0$ there exist $n \ge 1$ and $\mu \in E^{nm} v$ such that
$|E^1 r(\mu)| \ge 2$. There is an isomorphism $\tilde\kappa_m : C^*(\SG[m]{E}) \to
C([0,1], C^*(E(1,m+1)))$ such that for $a \in C_0(\SG{E}^0)$ and $\xi \in
C_c(\SG[m]{E}^1)$, we have
\[
\tilde{\kappa}_m(\tilde{\rho}_m(a))(t)
    = \begin{cases}
        \sum_{v \in E^0} a([v]) \sum_{e \in vE^1} p_e &\text{ if $t = 0$}\\
        \sum_{e \in E^1} a([e,t]) p_e &\text{ if $t \in (0,1)$}\\
        \sum_{v \in E^0} a([v]) \sum_{e \in E^1 v} p_e &\text{ if $t = 1$}
    \end{cases}
\]
and
\[
\tilde{\kappa}_m(\tilde{\psi}_m(\xi))(t)
    = \begin{cases}
        \sum_{\mu \in E^m} \xi([\mu]) \sum_{e \in s(\mu)E^1} s_{\mu e} &\text{ if $t = 0$}\\
        \sum_{\nu \in E^{m+1}} \xi([\nu,t]) s_\nu &\text{ if $t \in [0,1)$}\\
        \sum_{\mu \in E^m} \xi([\mu]) \sum_{e \in E^1 r(\mu)} s_{e\mu} &\text{ if $t = 1$.}
    \end{cases}
\]
\end{thm}
\begin{proof}
Let
\begin{align*}
A &:= \{f \in C([0,1], \Tt C^*(E(1,m+1))) : f(0) \in \jmath_{1,m+1}(\Tt C^*(E(0,m)))\},\quad\text{ and}\\
I &:= \{f \in A : f(t) \in I_{E(1,m+1)}\text{ for all } t\}.
\end{align*}

Lemma~\ref{lem:image of K} and the definition of $C^*(\SG[m]{E})$ show that the
isomorphism
\[
\kappa_m : \Tt C^*(\SG[m]{E}) \to \{f \in C([0,1], \Tt C^*(E(1,m+1))) : f(0)
\in \jmath(\Tt C^*(E(0,m)))\}
\]
of Theorem~\ref{thm:TC*(SGE) isomorphism} descends to an isomorphism $\kappa' :
C^*(\SG[m]{E}) \to A/I$.

The last statement of Lemma~\ref{lem:image of K} shows that $I = \{f : f(t) \in
I_{E(1,m+1)}\text{ for }t \not=0\text{ and }f(0) \in I_{E(0,m)}\}$. It follows that there
is an injective homomorphism $\kappa'' : A/I \to C([0,1], C^*(E(1,m+1)))$ that carries $x
+ I$ to the function
\[
t \mapsto \begin{cases}
    x(t) + I_{E(1,m+1)} &\text{ if $t \not= 0$}\\
    \jmath_{1, m+1}^{-1}(x(0)) + I_{E(0,m)} &\text{ if $t = 0$.}
    \end{cases}
\]
We claim that $\kappa''$ is surjective. For each $t \in (0,1]$, we have $\{\kappa''(x)(t)
: x \in A/I\} = C^*(E(1,m+1))$ because $\{x(t) : x \in A\} = \Tt C^*(E(1,m+1))$. At $t =
0$ we have $\kappa''(\{x(0) : x \in A/I\}) = \tilde\jmath_{1, m+1}(C^*(E)) =
C^*(E(1,m+1))$ because $\{x(0) : x \in A\} = \jmath_{1, m+1}(\Tt C^*(E(0,m)))$. The
extension $\widetilde{\kappa}''$ of $\kappa''$ to $\Mm(A/I)$ carries the canonical copy
of $C([0,1])$ in $\Mm(A/I)$ to the canonical copy of $C([0,1]) \in \Mm(C([0,1], C^*(E(1,
m+1))))$. Thus for $f \in C([0,1])$ and $x \in A$ we have $f \cdot \kappa''(a + I) =
\kappa''((f \cdot a) + I) \in \kappa''(A/I)$. It therefore follows from
\cite[Proposition~C.24]{Williams} that $\kappa''(A/I)$ is dense in $C([0,1],
C^*(E(1,m+1)))$, and therefore all of $C([0,1], C^*(E(1,m+1)))$ because the range of a
$C^*$-homomorphism is closed.
\end{proof}

To finish this section, we use our earlier results to describe, up to Morita equivalence,
the $C^*$-algebras $C^*(\SG[l]{E})$ for rational values of $l$ and for locally finite
graphs $E$ with no sources or sinks such that for every $v$ and every $m$ there exist
$p,q \ge 1$, $\mu \in vE^{pm}$ and $\nu \in E^{qm} v$ such that $|s(\mu)E^1| \ge 2$ and
$|E^1r(\nu)| \ge 2$. In particular, we show that this applies to any finite, strongly
connected graph with period $1$ (in the sense of Perron--Frobenius theory).

We first need the following elementary result about the Cuntz--Krieger algebras of the
delayed graphs associated to a graph $E$.

\begin{lem}\label{lem:Dn(E) analysis}
Let $E$ be a locally finite graph with no sinks or sources, and let $m,n$ be coprime
positive integers. Then $C^*(D_n(E)(1, m+1))$ is Morita equivalent to $C^*(E(0,m))$.
\end{lem}
\begin{proof}
By \cite[Theorem~3.1]{Bates}, we have $C^*(D_n(E)(1, m+1)) \cong C^*(D_n(E)(0,m))$, so it
suffices to show that the latter is Morita equivalent to $C^*(E(0,m))$.

To see this, observe that, by \cite[Lemma~1.1]{BPRS}, the series $\sum_{v \in E^0} p_v$
converges to a multiplier projection $P \in \Mm C^*(D_n(E)(0,m))$.

We claim that $P$ is full. For this, first partition $D_n(E)^0$ as $D_n(E)^0 = \bigcup_{j
\in \ZZ/n\ZZ} V_j$ by setting
\[
V_j := \begin{cases}
        E^0 &\text{ if $j = 0$}\\
        \{w_{e, j} : e \in E^1\} &\text{ if $j \not= 0$.}
    \end{cases}
\]
Then for $\alpha \in D_n(E)^1$, we have $r(\alpha) \in V_j$ if and only if $s(\alpha) \in
V_{j+1}$, and it follows that for $\lambda \in D_n(E)^*$, we have $r(\lambda) \in V_j$ if
and only if $s(\lambda) \in V_{j + [|\lambda|]_n}$. Now fix $u \in D_n(E)^0$, say $u \in
V_j$. Since $m,n$ are coprime, there exists $k \in \NN$ such that $km \equiv j\
(\operatorname{mod}\ n)$. Since $E$ has no sinks, there exists $\lambda \in E^{km} u$,
and since $s(\lambda) \in V_j$, it follows that $r(\lambda) \in V_{j - [km]} = V_0$, and
therefore $p_{r(\lambda)} \le P$. Hence $p_u = s_{\lambda}^* p_{r(\lambda)} s_\lambda$
belongs to the ideal generated by $P$. Now for $\alpha \in D_n(E)^1$, the generator
$s_\alpha = s_\alpha p_{s(\alpha)}$ also belongs to the ideal generated by $P$, and it
follows that $P$ is full.

To complete the proof, it suffices to show that $P C^*(D_n(E)(0,m)) P \cong C^*(E(0,m))$.
We begin by constructing a Cuntz--Krieger $E(0,m)$-family in $P C^*(D_n(E)(0,m)) P$.
First, for $\mu \in E^m$, we define $\alpha(\mu) \in D_n(E)(0,m)^n$ by
\[
\alpha(\mu) = f_{\mu_1, 1} \cdots f_{\mu_1, n} f_{\mu_2, 1} \cdots f_{\mu_2, n} \cdots f_{\mu_m, 1}
\cdots f_{\mu_m, n}.
\]
For $v \in E(0,m)^0 = E^0$, we define $P_v := p_v \in P C^*(D_n(E)(0,m)) P$, and for $\mu
\in E(0,m)^1 = E^m$, we define $S_\mu := s_{\alpha(\mu)} \in P C^*(D_n(E)(0,m)) P$. It is
routine to see that $(P, S)$ is a Cuntz--Krieger $E(0,m)$-family, so the universal
property of $C^*(E(0,m))$ implies that there is a homomorphism $\pi : C^*(E(0,m)) \to P
C^*(D_n(E)(0,m)) P$ such that $\pi(p_v) = P_v$ for all $v \in E^0$ and $\pi(s_\mu) =
S_\mu$ for all $\mu \in E(0,m)^1$. We have $\alpha(E^m) = (D_n(E)(0,m)^n)E^0 \subseteq
D_n(E)(0,m)^n$. The universal property of $C^*(D_n(E)(0,m))$ shows that there is an
action $\beta$ of $\TT$ on $C^*(D_n(E)(0,m))$ such that $\beta_z(s_\nu) = s_\nu$ for all
$\nu \in (D_n(E)(0,m)^1) \setminus (D_n(E)(0,m)^1)E^0$, and such that $\beta_z(s_\nu) = z
s_\nu$ for all $\nu \in (D_n(E)(0,m)^1)E^0$. Since $\gcd(m,n) = 0$, for each $\mu \in
D_n(E)(0,m)^n$, if we factor $\alpha(\mu) = \alpha_1 \cdots \alpha_n$ with each $\alpha_i
\in D_n(E)(0,m)^1$, we have $s(\alpha_n) \in E^0$ and $s(\alpha_i) \not\in E^0$ for $i <
n$. Consequently $\beta_z(s_{\alpha(\mu)}) = s_{\alpha_1} s_{\alpha_2} \cdots
s_{\alpha_{n-1}} (zs_{\alpha_n}) = z s_{\alpha(\mu)}$. Hence, writing $\gamma$ for the
gauge action on $C^*(E(0,m))$ we have $\pi \circ \gamma_z = \beta_z \circ \pi$. The
gauge-invariant uniqueness theorem \cite[Theorem~2.1]{BPRS} therefore implies that $\pi$
is injective.

It now suffices to show that the range of $\pi$ is $P C^*(D_n(E)(0,m)) P$. We have
\[
P C^*(D_n(E)(0,m)) P =
    \clsp\{s_\eta s^*_\zeta : \eta,\zeta \in V_0 D_n(E)(0,m)^*, s(\eta) = s(\zeta)\}.
\]
Fix $\eta, \zeta \in V_0(D_n(E)(0,m)^*)$ such that $s(\eta) = s(\zeta)$, say $\eta \in
D_n(E)(0,m)^k = D_n(E)^{km}$ and $\zeta \in D_n(E)(0,m)^l = D_n(E)^{lm}$. Then $s(\eta)
\in V_{[km]}$ and $s(\zeta) \in V_{[lm]}$ forcing $km \equiv lm\ (\operatorname{mod}\
n)$. Since $\gcd(m,n) = 1$, we deduce that $k \equiv l\ (\operatorname{mod}\ n)$. Fix $p$
such that $k+p \in n\ZZ$. Then $l+p \in n\ZZ$ too, and the Cuntz--Krieger relation forces
\[
s_\eta s^*_\zeta = \sum_{\xi \in s(\eta)D_n(E)^{pm}} s_{\eta\xi}s^*_{\zeta\xi}
    = \sum_{\xi \in s(\eta)D_n(E)(0,m)^p} s_{\eta\xi}s^*_{\zeta\xi}
\]
By construction, each $\eta\xi$ has the form $\alpha(\mu_1) \cdots \alpha(\mu_{k+p})$ for
some $\mu_i \in E(0,m)^1$, and so each $s_{\eta\xi} \in \pi(C^*(E(0,m)))$, and similarly
for each $s_{\zeta\xi}$. So $s_\eta s^*_\zeta \in \pi(C^*(E(0,m)))$. Thus $\pi$ is an
isomorphism of $C^*(E(0,m))$ onto $P C^*(D_n(E)(0,m)) P$ as required.
\end{proof}

\begin{cor}\label{cor:summary}
Let $E$ be a locally finite graph with no sinks or sources. Suppose that for every $v \in
E^0$ and every $m \in \ZZ \setminus \{0\}$, there exist $n \ge 1$ and $\mu \in E^{nm} v$
such that $|E^1 r(\mu)| \ge 2$. For $n \in \NN \setminus \{0\}$ and $m \in \ZZ$ such that
$\gcd(m,n) = 1$,
\begin{align}
C^*(\SG[\frac{m}{n}]{E}) &\sim_{\Me} C([0,1], C^*(E(0,m))) &&\text{if $m > 0$,}\nonumber\\
C^*(\SG[\frac{m}{n}]{E}) &\cong C_0(\SG{E}^0) \otimes C(\TT) &&\text{if $m = 0$, and}\label{eq:C*-structure}\\
C^*(\SG[\frac{m}{n}]{E}) &\sim_{\Me} C([0,1], C^*(E^{\op}(0,-m))) &&\text{if $m < 0$.}\nonumber
\end{align}
In particular,
\[
K_*(C^*(\SG[\frac{m}{n}]{E}))
    \cong\begin{cases}
        \big(\coker(1 - (A^t_E)^m), \ker(1 - (A^t_E)^m\big) &\text{ if $m > 0$,}\\
        \big(K^0(\SG{E}^0) \oplus K^1(\SG{E}^0), K^0(\SG{E}^0) \oplus K^1(\SG{E}^0)\big) &\text{ if $m = 0$, and}\\
        \big(\coker(1 - A_E^m), \ker(1 - A_E^m)\big) &\text{ if $m < 0$.}\\
    \end{cases}
\]
\end{cor}
\begin{proof}
For $m > 0$, Theorem~\ref{thm:C*SE structure} combined with Lemma~\ref{lem:Dn(E)
analysis} shows that $C^*(\SG[\frac{m}{n}]) \sim_{\Me} C([0,1], C^*(E(0,m)))$, and then
for $m < 0$, it follows from Lemma~\ref{lem:opposite} that $C^*(\SG[\frac{m}{n}])
\sim_{\Me} C([0,1], C^*(E^{\op}(0,-m)))$. Combining this with Theorem~\ref{thm:C* at
1,0,-1}(\ref{it:C* at 1,0,-1 ii}) proves the first statement.

For the second statement, since $K_*(C([0,1], C^*(E(0,m)))) \cong K_*(C^*(E(0,m)))$ (for
example, apply Lemma~\ref{lem:homotopy} to $\iota = \id_{C^*(E(0,m))}$, and then use
\cite[Proposition~3.2.6 and Theorem~8.2.2(vi)]{RLL}), the computation
\cite[Theorem~3.2]{RS} of $K$-theory for graph $C^*$-algebras gives
\[
K_*(C([0,1], C^*(E(0,m)))) \cong \big(\coker(1 - (A_E^t)^m), \ker(1 - A_E^t)^m\big)\qquad\text{ for $m > 0$.}
\]
A similar argument gives
\[
K_*(C([0,1], C^*(E(0,m)))) \cong \big(\coker(1 - A_E^m), \ker(1 - A_E^m)\big)\qquad\text{ for $m < 0$}
\]
because $A_{E^{\op}}^t = A_E$ and our hypotheses are symmetrical in $E$ and $E^{\op}$.
Finally, the K\"unneth theorem and that operator $K$-theory agrees with topological
$K$-theory for commutative $C^*$algebras implies that $K_i(C_0(\SG{E}^0) \otimes C(\TT))
\cong K^0(C_0(\SG{E}^0)) \oplus K^1(C_0(\SG{E}^0))$ for $i = 1,2$.
\end{proof}

We finish by applying Corollary~\ref{cor:summary} to strongly connected finite graphs
with period 1.

We say that a graph $E$ is strongly connected if for all $v,w \in E^0$ the set $v E^* w
\setminus E^0$ is nonempty (we take the convention that a graph consisting of a single
vertex and no edges is not strongly connected). If $E$ is a strongly connected finite
graph, then the \emph{period} of $E$ is defined as $P(E) := \gcd\{|\mu| : \mu \in E^*
\setminus E^0, r(\mu) = s(\mu)\}$.

Also recall that if $E$ is a graph, then there is a map $\partial : \ZZ E^0 \to \ZZ E^1$
given by $\partial(a)(e) = a(r(e)) - a(s(e))$. The $0$\textsuperscript{th} and
$1$\textsuperscript{st} homology groups of $E$ are defined by $H_1(E) = \ZZ E^1 /
\partial_1(\ZZ E^0)$, and $H_0(E) = \ker(\partial_1)$. The higher homology groups
$H_n(E)$, $n \ge 2$ are trivial (see, for example, \cite[Remark~3.6]{KPS3}). The group
$H_0(E)$ is isomorphic to the free abelian group generated by the connected components of
$E$.

We say that a finite graph $E$ is a \emph{simple cycle} if, putting $n = |E^0|$, there
are bijections $i \mapsto v_i$ and $i \mapsto e_i$ of $\ZZ/n\ZZ$ onto $E^0$ and $E^1$
respectively such that $r(e_i) = v_i$ and $s(e_i) = v_{i+1}$ for all $i$.

\begin{cor}
Let $E$ be a finite strongly connected graph that is not a simple cycle, and suppose that
$P(E) = 1$. Then for every $v \in E^0$ and every $m \in \ZZ \setminus \{0\}$, there exist
$n \ge 1$ and $\mu \in E^{nm} v$ such that $|E^1 r(\mu)| \ge 2$. For $n > 0$ and $m \in
\ZZ$ with $m,n$ coprime,
\begin{align*}
C^*(\SG[\frac{m}{n}]{E}) &\sim_{\Me} C([0,1], C^*(E(0,|m|))) &&\text{if $m \not= 0$, and}\nonumber\\
C^*(\SG[\frac{m}{n}]{E}) &\cong C_0(\SG{E}^0) \otimes C(\TT) &&\text{if $m = 0$.}
\end{align*}
In particular,
\[
K_*(C^*(\SG[\frac{m}{n}]{E}))
    \cong\begin{cases}
        \big(\coker(1 - A^m_E), \ker(1 - A^m_E)\big) &\text{ if $m \not= 0$}\\
        \big(\ZZ \oplus H_1(E), \ZZ \oplus H_1(E)\big) &\text{ if $m = 0$.}
    \end{cases}
\]
\end{cor}
\begin{proof}
First observe that the opposite graph $E^{\op}$ is also strongly connected with
$P(E^{\op}) = 1$, and is also not a simple cycle. To prove the first statement, it
therefore suffices to consider $m > 0$ (the case $m = 0$ follows from the second line
of~\eqref{eq:C*-structure}).

So fix $m \ge 1$ and $v \in E^0$. Since $E$ is strongly connected, we have $|E^1 v| \ge
1$ for all $v$, and since $E$ is not a simple cycle, a counting argument shows that there
exists $w \in E^0$ such that $|E^1 w| \ge 2$. Since $E$ is strongly connected, the set $w
E^* v$ is nonempty, say $\alpha \in w E^* v$. It is standard (see for example
\cite[Lemma~6.1]{LLNSW} applied with $k = 1$) that $P(E) = \{|\lambda| - |\nu| :
\lambda,\nu \in v E^* v\}$. In particular, there are cycles $\lambda,\nu \in vE^* v$ such
that $|\lambda| - |\nu| = m - |\alpha|$. It follows that $|\alpha\lambda\nu^{m-1}| =
|\alpha| + (|\lambda| - |\nu|) + m|\nu| = m(|\nu| + 1)$. So $n := |\nu| + 1$ and $\mu :=
\alpha\lambda\nu^{m-1}$ satisfy $\mu \in E^{nm} v$ and $|E^1 r(\mu)| = |E^1 w| \ge 2$.

By Corollary~\ref{cor:summary}, it now suffices to show that
\[
C([0,1], C^*(E^{\op}(0, m))) \sim_{\Me} C([0,1], C^*(E(0,m)))\quad\text{ for all $m \ge 1$,}
\]
and that $K_*(C(\SG{E}^0 \times \SS)) \cong \big(\ZZ \oplus H_1(E), \ZZ \oplus
H_1(E)\big)$.

Fix $m > 0$. Since $E$ has period 1, it is easy to see that $E(0,m)$ and $E^{\op}(0,m)$
are strongly connected finite graphs that are not simple cycles. Since
\[
\ker(A^t) \cong \ker(A)\quad\text{ and }\quad \coker(A^t) \cong \coker(A),
\]
(see, for example \cite[Section~6.2]{KamPut}), we have $K_*(C^*(E^{\op}(0,m))) \cong
K_*(C^*(E(0,m)))$, and so \cite[Theorem~6.1]{Rordam} shows that $C^*(E^{\op}(0, m))
\sim_{\Me} C^*(E(0, m))$. Thus the function algebras $C([0,1], C^*(E^{\op}(0, m)))$ and
$C([0,1], C^*(E(0,m)))$ are Morita equivalent as well.

Now take $m = 0$. We have $K^*(\SS) = (\ZZ, \ZZ)$. Thus, by the K\"unneth theorem in
K-theory, it suffices to show that $K^0(\SG{E}^0) \cong \ZZ$ and $K^1(\SG{E}^0) \cong
H_1(E)$. Since $\SG{E}^0$ and its suspension are finite CW-complexes of dimension at most
2, \cite[Theorem~1]{Thomas} shows that $K^0(\SG{E}^0) \cong \bigoplus_{n \ge 0}
H_{2n}(\SG{E}^0) = H_0(\SG{E}^0)$ and that $K^1(\SG{E}^0)$ is isomorphic to the direct
sum of the even homology groups of the suspension of $\SG{E}^0$, and hence to the direct
sum of the odd homology groups of $\SG{E}^0$. Theorem~6.3 of \cite{KPS3} gives
$H_*(\SG{E}^0) \cong H_*(E)$, and we have $H_0(E) \cong \ZZ$ because $E$ is
connected.\end{proof}

\end{document}